\documentclass{amsart}
\usepackage{amsmath, amsthm, amssymb}
\usepackage{bbm}
\usepackage{graphicx}
\usepackage{tikz}
\usetikzlibrary{positioning, shapes.geometric, arrows.meta, backgrounds}
\usepackage[hypertexnames=false]{hyperref}
\usepackage{subcaption}
\usepackage{wrapfig}
\renewcommand\thesubfigure{(\Alph{subfigure})}
\usepackage{booktabs}
\usepackage{makecell}

\theoremstyle{plain}
\newtheorem{theorem}{Theorem}[section]
\newtheorem{lemma}[theorem]{Lemma}
\newtheorem{proposition}[theorem]{Proposition}
\newtheorem{corollary}[theorem]{Corollary}
\newtheorem*{theorem*}{Theorem}
\newtheorem{namedtheoreminner}{Theorem}
\newenvironment{namedtheorem}[1]{%
  \renewcommand{\thenamedtheoreminner}{#1}%
  \begin{namedtheoreminner}%
}{%
  \end{namedtheoreminner}%
}

\theoremstyle{definition}
\newtheorem*{rexample}{Example~\ref{ex:running}}
\newtheorem*{rexample2}{Example~\ref{ex:running2}}
\newtheorem{definition}[theorem]{Definition}  
\newtheorem{example}[theorem]{Example}
\newtheorem{remark}[theorem]{Remark}

\DeclareMathOperator{\Aut}{Aut}
\DeclareMathOperator{\Conv}{Conv}
\DeclareMathOperator{\Cay}{Cay}
\DeclareMathOperator{\Cone}{Cone}
\DeclareMathOperator{\re}{Re}
\DeclareMathOperator{\Id}{Id}

\DeclareMathOperator{\diag}{diag}
\DeclareMathOperator{\Tr}{Tr}
\DeclareMathOperator{\End}{End}
\DeclareMathOperator{\GL}{GL}
\DeclareMathOperator{\Hom}{Hom}
\DeclareMathOperator{\res}{res}
\DeclareMathOperator{\ind}{ind}

\newcommand{\E}{\mathcal{E}}
\newcommand{\Om}{\mathcal{O}}
\newcommand{\set}[1]{{\{#1\}}}
\newcommand{\abs}[1]{{\vert #1\vert}}
\newcommand{\paren}[1]{{\left( #1\right)}}
\newcommand{\R}{\mathbb{R}}
\newcommand{\Z}{\mathbb{Z}}
\newcommand{\C}{\mathbb{C}}
\newcommand{\Q}{\mathbb{Q}}
\newcommand{\HH}{\mathbb{H}}
\newcommand{\s}{\mathcal{S}}
\newcommand{\one}{\mathbbm{1}}

\title{Conformal Rigidity of Graphs: \\Subdifferentials and Orbit-Isometries}

\author{Andrew Niu}
\address{Department of Mathematics, University of Washington, Seattle WA 98195}
\email{amniu@uw.edu}

\begin{document}
\begin{abstract}
    A connected undirected graph $G = (V,E)$ is lower conformally rigid if uniform edge weights maximize the second smallest Laplacian eigenvalue $\lambda_2(w)$ over all normalized edge weights $w$, and upper conformally rigid if uniform edge weights minimize the largest eigenvalue $\lambda_n(w)$ over all normalized edge weights; G is conformally rigid if it is lower or upper conformally rigid. This paper establishes a new framework for conformal rigidity through the language of subdifferentials, unifying the variational perspective on eigenvalue optimization with the geometry of edge-isometric spectral embeddings, which are known to characterize conformal rigidity. This subdifferential framework lends itself naturally to techniques of symmetry reduction that motivate the notion of an orbit-isometric embedding --- a weaker condition than edge-isometry that accounts for the symmetries of $G$ while remaining sufficient for conformal rigidity. The notion opens the door to tools from representation theory: for a large class of graphs, including all vertex-transitive ones, we show that conformal rigidity is certified by a single eigenvector, resolving an open question and explaining the conformal rigidity of previously unexplained graphs. This extra structure enables a new, algebraically exact certification method for conformal rigidity, bypassing the numerical difficulties of prior approaches. In many cases, the problem reduces to a check of linear feasibility, and in general, to solving a system of quadratic equations via Gr\"{o}bner bases.
\end{abstract}
\maketitle

\section{Introduction}
\subsection{Introduction}
Let $G = (V, E)$ be a finite, connected, undirected graph on $n= \abs{V}$ vertices $\set{v_1, \ldots, v_n}$.  A well-studied object associated to the graph $G$ is the graph Laplacian given by $L = D - A \in \mathbb{R}^{n \times n}$ where $ D = \diag \set{\deg(v_1), \ldots, \deg(v_n)}$ is the degree matrix and $A$ is the adjacency matrix of $G$ given by
\[A_{ij} =
\begin{cases}
1, & \text{if } ij \in E \\
0, & \text{otherwise.}
\end{cases}\]
Viewed as a quadratic form, the Laplacian satisfies 
\begin{equation*}
    \langle \varphi, L \varphi \rangle = \varphi^T L \varphi = \sum_{uv \in E} {(\varphi(u) - \varphi(v))}^2
\end{equation*}
for functions $\varphi: V \to \R$ on the vertices of $G$. This expression, sometimes referred to as the \textit{Dirichlet energy}, immediately implies that $L$ is positive semidefinite and $L \one = 0$, where $\one$ is the constant vector. Thus, $L$ has eigenvalues $0 = \lambda_1 \leq \lambda_2 \leq \cdots \leq \lambda_n$, and  $\lambda_2 > 0$ if and only if $G$ is connected. By the Rayleigh-Ritz theorem~\cite[\S 4.2]{horn}, we can characterize $\lambda_2$ and $\lambda_n$ variationally by
\[\lambda_2 = \min_{\substack{\|\varphi\|=1, \\ \langle \one, \varphi \rangle = 0}} \langle \varphi, L \varphi \rangle \,\,\,\textup{ and } \,\,\, \lambda_n = \max_{\|\varphi\| = 1} \langle \varphi, L \varphi \rangle.\]
The second eigenvalue $\lambda_2$ governs the global connectivity of $G$; small values of $\lambda_2$ indicate the presence of bottlenecks or weakly connected regions, whereas larger values correspond to more robust connectivity across $G$. From a dynamical perspective, we can consider a system of heat diffusion~\cite{sun} or a consensus protocol~\cite{olfati} defined by the differential equation $\dot{\varphi}(t) = - L \varphi(t)$. In this setting, $\lambda_2$ determines the asymptotic rate of convergence to the constant vector on $G$, and a larger $\lambda_2$ is equivalent to faster energy or information spread. On the opposite end of the spectrum, the largest eigenvalue $\lambda_n$ represents the most oscillatory function on $G$ and serves as a measure of the bipartiteness of $G$. In the context of dynamics, $\lambda_n$ roughly measures the stability or smoothness of the network with lower values of $\lambda_n$ corresponding to better network stability. For example, $\lambda_n$ controls the largest admissible time delay in the average-consensus problem~\cite[Theorem 10]{olfati}. We refer to~\cite{chung, godsil} for a general survey of graph Laplacians.

Now that we have established that larger values of $\lambda_2$ and smaller values of $\lambda_n$ are desirable for network performance, a natural problem to consider is how we may redistribute weights to optimize these spectral values for a fixed network topology given by $G = (V, E)$. This leads us to the definition of the \textit{weighted Laplacian}: if we assign nonnegative weights $w: E \to \R$ to each edge, then the weighted Laplacian is $L(w) = D(w) - A(w)$ where $D(w)$ is the weighted degree matrix $D(w)_{uu} = \sum_{uv \in E} w_{uv}$,
and $A(w)$ is the weighted adjacency matrix
\[A(w)_{uv} =
\begin{cases}
w_{uv}, & \text{if } uv \in E \\
0, & \text{otherwise.}
\end{cases}\]
Similar to before, the quadratic form induced by $L(w)$ is
\[\langle \varphi, L(w) \varphi \rangle = \sum_{uv \in E} w_{uv} {(\varphi(u) - \varphi(v))}^2,\]
so $L(w)$ is positive semidefinite with $0 = \lambda_1(w) \leq \lambda_2(w) \leq \cdots \leq \lambda_n(w)$. Since the spectrum scales linearly with the weights $w$, it makes sense to normalize the edge weights to satisfy $\sum_{uv \in E} w_{uv} = \abs{E}$.
We can interpret this normalization as having a fixed budget of resource allocation for our network $G$, and maximizing $\lambda_2(w)$ yields the fastest-mixing and densest representation of our network, while minimizing $\lambda_n(w)$ corresponds to  the most stable representation of our network. These optimization problems have been extensively studied in various contexts~\cite{fiedler, goring2, goring1, olfati, osting, sun}.

\subsection{Conformal Rigidity} In this paper, we consider a structural analogue to the previous optimization problems: which graphs $G$ are so intrinsically structured such that the uniform weighting $w=\one$ cannot be improved upon to further increase $\lambda_2(w)$ or decrease $\lambda_n(w)$? To make this question precise, we use $L = L(\one)$ and $\lambda_k = \lambda_k(\one), 1 \leq k \leq n$ to denote the Laplacian and its eigenvalues for the uniform weighting on a fixed graph $G$.

\begin{definition}\label{def:conformal-rigidity}
    We say a graph $G = (V, E)$ is \textit{lower conformally rigid} if for all possible edge weights $w_{uv} \geq 0$ normalized to $ \sum_{uv \in E} w_{uv} = \abs{E}$, 
    \[\lambda_2(w) \leq \lambda_2 := \lambda_2(\one).\]
    In this paper, we say $G$ is \textit{upper conformally rigid} if among all normalized edge weights $w' \geq 0$,
    \[\lambda_n(w') \geq \lambda_n := \lambda_n(\one).\] 
    We call $G$ \textit{conformally rigid} if it is lower OR upper conformally rigid, specifying which type when necessary.
\end{definition}

The notion of conformal rigidity was first introduced in~\cite{steinerberger} and further explored in~\cite{gouveia}. In these papers the authors call a graph $G$ conformally rigid if and only if it is both lower AND upper conformally rigid. These papers demonstrate that the property of being both lower and upper conformally rigid is extremely rare and can only be expected in the presence of high symmetry (in some suitable sense) in $G$; in fact, the same is true for graphs that satisfy either condition alone, which is why we broaden the definition to include graphs that  are lower OR upper conformally rigid. Perhaps the best way to begin to understand what makes a graph $G$ conformally rigid is to study what properties we want to \textit{avoid}. For $\lambda_2(\one)$ to be maximal, no perturbation of edge weights should  increase $\lambda_2(w)$. This means that there can be no weak links or bottlenecks in $G$. If there were, we could increase the weights on these edges to increase $\lambda_2(w)$.

\begin{example}\label{ex:barbell}
    Let $G$ be the barbell graph given by two copies of $K_3$ connected by a single edge. Then  $\lambda_2 = (5-\sqrt{17})/2 \approx 0.438$ where the vertices are labeled as in Figure~\ref{fig:barbell}. Clearly the edge $cd$ of the barbell is a bottleneck in $G$. Thus, if we move more weight to this edge, we can increase $\lambda_2(w)$. Shifting the weight so that $w_{ab} = w_{ef} = 0$ and $w_{cd} = 3$ yields\[\lambda_2(w) = \frac{9- \sqrt{57}}{2} \approx 0.725 > \lambda_2.\]
    
\end{example}
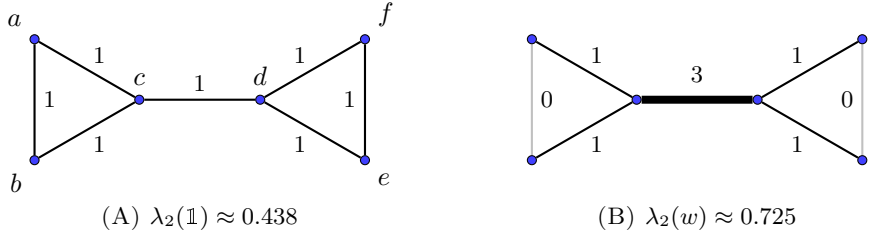
\begin{figure}[h]
    \centering
    \begin{subfigure}[c]{0.48\textwidth}
        \centering
        \begin{tikzpicture}[
                scale=0.8,
                vertex/.style={circle, draw, fill=blue!75, inner sep=1.2pt},
                edge label/.style={font=\small, fill=white, inner sep=1pt}
            ]
            \useasboundingbox (-0.3, -0.45) rectangle (5.75, 2.45);
            \node[vertex, label= above left: $a$] (A) at (0,2) {};
            \node[vertex, label= below left: $b$] (B) at (0,0) {};
            \node[vertex, label= above: $c$] (C) at (1.732, 1) {};

            \node[vertex, label= above: $d$] (D) at (3.732,1) {};
            \node[vertex, label= below right: $e$] (E) at (5.464, 0) {};
            \node[vertex, label= above right: $f$] (F) at (5.464, 2) {};

            \draw[thick] (A)--(B) node[midway, right, xshift=2pt, edge label] {$1$};
            \draw[thick] (A)--(C) node[midway, above right, xshift = 1pt, yshift=1pt, edge label] {$1$};
            \draw[thick] (B)--(C) node[midway, below right, xshift = 1pt, yshift=-1pt, edge label] {$1$};

            \draw[thick] (D)--(E) node[midway, below left, xshift = -1pt, yshift=-1pt, edge label] {$1$};
            \draw[thick] (D)--(F) node[midway, above left, xshift = -1pt, yshift=1pt, edge label] {$1$};
            \draw[thick] (E)--(F) node[midway, left, xshift=-2pt, edge label] {$1$};

            \draw[thick] (C)--(D) node[midway, above, yshift=2pt, edge label] {$1$};
        \end{tikzpicture}
        \caption{$\lambda_2(\one) \approx 0.438$}\label{fig:barbell}
    \end{subfigure}
    \hfill
    \begin{subfigure}[c]{0.48\textwidth}
        \centering
        \begin{tikzpicture}[
                scale=0.8,
                vertex/.style={circle, draw, fill=blue!75, inner sep=1.2pt},
                edge label/.style={font=\small, fill=white, inner sep=1pt}
            ]
            \useasboundingbox (-0.3, -0.45) rectangle (5.75, 2.45);
            \node[vertex] (A) at (0,2) {};
            \node[vertex] (B) at (0,0) {};
            \node[vertex] (C) at (1.732, 1) {};

            \node[vertex] (D) at (3.732,1) {};
            \node[vertex] (E) at (5.464, 0) {};
            \node[vertex] (F) at (5.464, 2) {};

            \draw[thick, opacity=0.25] (A)--(B) node[midway, right, xshift=2pt, edge label, opacity=1] {$0$};
            \draw[thick, opacity=0.25] (E)--(F) node[midway, left, xshift=-2pt, edge label, opacity=1] {$0$};

            \draw[thick] (A)--(C) node[midway, above right, xshift=1pt, yshift=1pt, edge label] {$1$};
            \draw[thick] (B)--(C) node[midway, below right, xshift=1pt, yshift=-1pt, edge label] {$1$};
            \draw[thick] (D)--(E) node[midway, below left, xshift=-1pt, yshift=-1pt, edge label] {$1$};
            \draw[thick] (D)--(F) node[midway, above left, xshift=-1pt, yshift=1pt, edge label] {$1$};

            \draw[line width=3pt] (C)--(D) node[midway, above, yshift=4pt, edge label] {$3$};
        \end{tikzpicture}
        \caption{$\lambda_2(w) \approx 0.725$}\label{fig:barbell-shifted}
    \end{subfigure}
    \caption{The barbell graph with uniform edge weights (left) and a weight redistribution that strictly increases $\lambda_2$ (right). This shows that the barbell graph is not lower conformally rigid.}\label{fig:non-examples}
\end{figure}

This example suggests that for a graph $G$ to be lower conformally rigid, there cannot be any asymmetries in the density of the graph; otherwise we could shift weight from denser regions to sparser regions. Thus, for $G$ to be lower conformally rigid, every edge should ``look'' the same to $\lambda_2$. We will see in \S~\ref{sec:subdifferential} that the notion of an \textit{edge-isometric spectral embedding} (defined in \S~\ref{sec:main-results}) gives a precise characterization of lower conformal rigidity.

\begin{example}\label{ex:friendship}
    Let $G$ be the friendship graph $F_3$ with center vertex $a$ and outer vertices $b, c, d, e, f, g$ as shown in Figure~\ref{fig:friendship-weights}. We see that $\lambda_n = 7$ is simple with eigenvector 
    $\varphi^T = \begin{pmatrix}
        -6 & 1 & 1 & 1 & 1 & 1 & 1
    \end{pmatrix}$.
    Since $\lambda_n$ measures the most oscillatory function on $G$, the eigenvector $\varphi$ tries to place neighbors as far apart as possible. The central hub $a$ is connected to every other vertex, so it gets pushed far from all the others. This forces each friendship pair $\set{b,c}$, $\set{d,e}$, $\set{f,g}$ to take the same value, and the rim edges $bc$, $de$, $fg$ carry no energy. Thus, moving weight from the spoke edges to the underloaded rim edges decreases $\lambda_n(w)$: setting $w_{av} = 0.5$ for $v \in \set{b,c,d,e,f,g}$ and $w_{bc} = w_{de} = w_{fg} = 2$ gives $\lambda_n(w) = 4.5 < 7 = \lambda_n$.
\end{example}

\begin{figure}[h]
    \centering
    \begin{subfigure}[c]{0.48\textwidth}
        \centering
        \begin{tikzpicture}[
                scale=0.8,
                vertex/.style={circle, draw, fill=blue!75, inner sep=1.2pt},
                edge label/.style={font=\small, fill=white, inner sep=1pt}
            ]
            \useasboundingbox (-2.8,-2.6) rectangle (2.8,2.6);
            \node[vertex, label=below:$a$] (A) at (0,0) {};
            \node[vertex, label={above right:$b$}] (B) at (1, 1.732) {};
            \node[vertex, label={above left:$c$}]  (C) at (-1, 1.732) {};
            \node[vertex, label={left:$d$}]         (D) at (-2, 0) {};
            \node[vertex, label={below left:$e$}]   (E) at (-1,-1.732) {};
            \node[vertex, label={below right:$f$}]  (F) at (1,-1.732) {};
            \node[vertex, label={right:$g$}]        (G) at (2, 0) {};
            \draw[thick] (A)--(B) node[midway, right,  xshift=2pt,  edge label] {$1$};
            \draw[thick] (A)--(C) node[midway, left,   xshift=-2pt, edge label] {$1$};
            \draw[thick] (A)--(D) node[midway, above,  yshift=2pt,  edge label] {$1$};
            \draw[thick] (A)--(E) node[midway, left,   xshift=-2pt, edge label] {$1$};
            \draw[thick] (A)--(F) node[midway, right,  xshift=2pt,  edge label] {$1$};
            \draw[thick] (A)--(G) node[midway, above,  yshift=2pt, edge label] {$1$};
            \draw[thick] (B)--(C) node[midway, above, yshift=2pt,  edge label] {$1$};
            \draw[thick] (D)--(E) node[midway, left,  xshift=-2pt, edge label] {$1$};
            \draw[thick] (F)--(G) node[midway, right, xshift=2pt,  edge label] {$1$};
        \end{tikzpicture}
        \caption{$\lambda_n(\one) = 7$}\label{fig:friendship-before}
    \end{subfigure}
    \hfill
    \begin{subfigure}[c]{0.48\textwidth}
        \centering
        \begin{tikzpicture}[
                scale=0.8,
                vertex/.style={circle, draw, fill=blue!75, inner sep=1.2pt},
                edge label/.style={font=\small, fill=white, inner sep=1pt}
            ]
            \useasboundingbox (-2.8,-2.6) rectangle (2.8,2.6);
            \node[vertex] (A) at (0,0) {};
            \node[vertex] (B) at (1, 1.732) {};
            \node[vertex] (C) at (-1, 1.732) {};
            \node[vertex] (D) at (-2, 0) {};
            \node[vertex] (E) at (-1,-1.732) {};
            \node[vertex] (F) at (1,-1.732) {};
            \node[vertex] (G) at (2, 0) {};
            \draw[line width=0.5pt] (A)--(B) node[midway, right,  xshift=2pt,  edge label, opacity=1] {$0.5$};
            \draw[line width=0.5pt] (A)--(C) node[midway, left,   xshift=-2pt, edge label, opacity=1] {$0.5$};
            \draw[line width=0.5pt] (A)--(D) node[midway, above,  yshift=2pt,  edge label, opacity=1] {$0.5$};
            \draw[line width=0.5pt] (A)--(E) node[midway, left,   xshift=-2pt, edge label, opacity=1] {$0.5$};
            \draw[line width=0.5pt] (A)--(F) node[midway, right,  xshift=2pt,  edge label, opacity=1] {$0.5$};
            \draw[line width=0.5pt] (A)--(G) node[midway, above,  yshift=2pt, edge label, opacity=1] {$0.5$};
            \draw[line width=2pt] (B)--(C) node[midway, above, yshift=4pt,  edge label] {$2$};
            \draw[line width=2pt] (D)--(E) node[midway, left,  xshift=-4pt, edge label] {$2$};
            \draw[line width=2pt] (F)--(G) node[midway, right, xshift=4pt,  edge label] {$2$};
        \end{tikzpicture}
        \caption{$\lambda_n(w) = 4.5$}\label{fig:friendship-after}
    \end{subfigure}
    \caption{The friendship graph $F_3$ with uniform edge weights (left) and a weight redistribution that strictly decreases $\lambda_n$ (right). This shows that $F_3$ is not upper conformally rigid.}\label{fig:friendship-weights}
\end{figure}
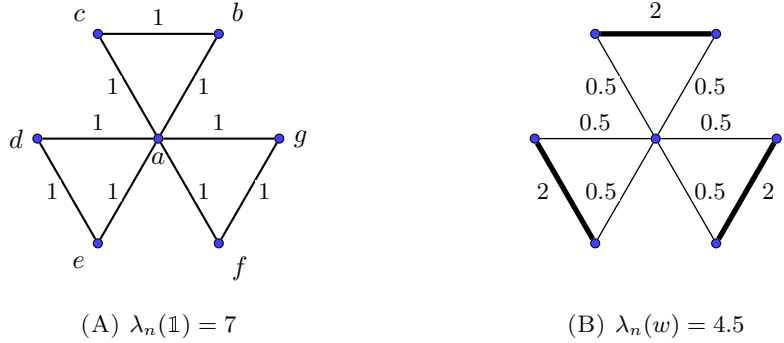

As in the case of lower conformal rigidity, every edge of $G$ must look the same to $\lambda_n$ to be upper conformally rigid. Since $\lambda_n$ measures the most oscillatory function on $G$, roughly speaking, we do not want hubs or asymmetries in the degrees of vertices. Otherwise, the corresponding eigenvector(s) will concentrate stress on edges incident to vertices of high degree as we saw in the previous example.

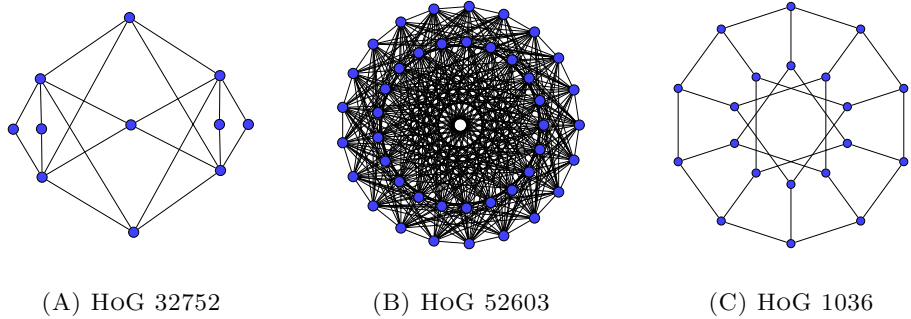
\begin{figure}
    \centering
    \begin{subfigure}[b]{0.31\textwidth}
        \centering
    \resizebox{\linewidth}{!}{%
    \begin{tikzpicture}[
        scale=1.45,
        vertex/.style={circle, draw, fill=blue!75, inner sep=1.5pt},
        every edge/.style={draw, line width=0.6pt},
    ]
        \useasboundingbox (-1.5, -1.5) rectangle (1.5, 1.5);
        \node[vertex] (0) at (0.8944, 0.0058) {};
        \node[vertex] (1) at (1.1869, 0.0061) {};
        \node[vertex] (2) at (-1.1912, -0.0418) {};
        \node[vertex] (3) at (-0.9058, -0.0391) {};
        \node[vertex] (4) at (-0.0011, 0.0015) {};
        \node[vertex] (5) at (-0.0154, 1.0873) {};
        \node[vertex] (6) at (0.0267, -1.0846) {};
        \node[vertex] (7) at (0.9053, -0.4596) {};
        \node[vertex] (8) at (0.8990, 0.5000) {};
        \node[vertex] (9) at (-0.9001, -0.5277) {};
        \node[vertex] (10) at (-0.9177, 0.4674) {};
        \draw (0)--(7) (0)--(8) (1)--(7) (1)--(8) (2)--(9) (2)--(10);
        \draw (3)--(9) (3)--(10) (4)--(7) (4)--(8) (4)--(9) (4)--(10);
        \draw (5)--(7) (5)--(8) (5)--(9) (5)--(10) (6)--(7) (6)--(8);
        \draw (6)--(9) (6)--(10);
    \end{tikzpicture}%
    }%
        \caption{\textsc{HoG 32752}}
    \end{subfigure}%
    \hfill%
    \begin{subfigure}[b]{0.31\textwidth}
        \centering
    \resizebox{\linewidth}{!}{%
    \begin{tikzpicture}[
        scale=1.45,
        vertex/.style={circle, draw, fill=blue!75, inner sep=1.5pt},
        every edge/.style={draw, line width=0.2pt},
    ]
        \useasboundingbox (-1.5, -1.5) rectangle (1.5, 1.5);
        \coordinate (0) at (0.836, 0.000);
        \coordinate (1) at (-0.428, -0.730);
        \coordinate (2) at (-0.428, 0.730);
        \coordinate (3) at (0.689, 0.475);
        \coordinate (4) at (-0.840, -0.126);
        \coordinate (5) at (-0.624, 0.573);
        \coordinate (6) at (-0.624, -0.573);
        \coordinate (7) at (0.689, -0.475);
        \coordinate (8) at (-0.840, 0.126);
        \coordinate (9) at (-0.608, -1.042);
        \coordinate (10) at (0.744, -0.941);
        \coordinate (11) at (-0.889, -0.818);
        \coordinate (12) at (0.988, -0.678);
        \coordinate (13) at (-1.091, 0.522);
        \coordinate (14) at (0.433, -1.120);
        \coordinate (15) at (1.143, -0.355);
        \coordinate (16) at (0.433, 1.120);
        \coordinate (17) at (1.197, 0.000);
        \coordinate (18) at (-1.197, 0.179);
        \coordinate (19) at (-1.091, -0.522);
        \coordinate (20) at (0.988, 0.678);
        \coordinate (21) at (-0.274, -1.173);
        \coordinate (22) at (-0.889, 0.818);
        \coordinate (23) at (1.143, 0.355);
        \coordinate (24) at (0.798, 0.248);
        \coordinate (25) at (0.056, -0.840);
        \coordinate (26) at (0.301, 0.784);
        \coordinate (27) at (0.056, 0.840);
        \coordinate (28) at (0.301, -0.784);
        \coordinate (29) at (0.798, -0.248);
        \coordinate (30) at (-0.194, 0.821);
        \coordinate (31) at (0.518, 0.659);
        \coordinate (32) at (-0.766, -0.365);
        \coordinate (33) at (0.518, -0.659);
        \coordinate (34) at (-0.194, -0.821);
        \coordinate (35) at (-0.766, 0.365);
        \coordinate (36) at (0.083, -1.200);
        \coordinate (37) at (0.083, 1.200);
        \coordinate (38) at (0.744, 0.941);
        \coordinate (39) at (-1.197, -0.179);
        \coordinate (40) at (-0.608, 1.042);
        \coordinate (41) at (-0.274, 1.173);
        \draw (0)--(1) (0)--(2) (0)--(3) (0)--(4) (0)--(5) (0)--(6);
        \draw (0)--(7) (0)--(8) (0)--(9) (0)--(10) (0)--(11) (0)--(12);
        \draw (0)--(13) (0)--(14) (0)--(15) (0)--(16) (0)--(17) (0)--(18);
        \draw (0)--(19) (0)--(20) (1)--(2) (1)--(9) (1)--(10) (1)--(11);
        \draw (1)--(18) (1)--(19) (1)--(20) (1)--(24) (1)--(25) (1)--(26);
        \draw (1)--(30) (1)--(31) (1)--(32) (1)--(36) (1)--(37) (1)--(38);
        \draw (1)--(39) (1)--(40) (1)--(41) (2)--(12) (2)--(13) (2)--(14);
        \draw (2)--(15) (2)--(16) (2)--(17) (2)--(27) (2)--(28) (2)--(29);
        \draw (2)--(33) (2)--(34) (2)--(35) (2)--(36) (2)--(37) (2)--(38);
        \draw (2)--(39) (2)--(40) (2)--(41) (3)--(6) (3)--(8) (3)--(9);
        \draw (3)--(12) (3)--(15) (3)--(16) (3)--(17) (3)--(19) (3)--(20);
        \draw (3)--(21) (3)--(23) (3)--(25) (3)--(26) (3)--(32) (3)--(34);
        \draw (3)--(35) (3)--(36) (3)--(39) (3)--(41) (4)--(6) (4)--(7);
        \draw (4)--(11) (4)--(13) (4)--(15) (4)--(16) (4)--(17) (4)--(18);
        \draw (4)--(20) (4)--(21) (4)--(22) (4)--(24) (4)--(26) (4)--(31);
        \draw (4)--(33) (4)--(35) (4)--(38) (4)--(39) (4)--(40) (5)--(7);
        \draw (5)--(8) (5)--(10) (5)--(14) (5)--(15) (5)--(16) (5)--(17);
        \draw (5)--(18) (5)--(19) (5)--(22) (5)--(23) (5)--(24) (5)--(25);
        \draw (5)--(30) (5)--(33) (5)--(34) (5)--(37) (5)--(40) (5)--(41);
        \draw (6)--(11) (6)--(13) (6)--(14) (6)--(16) (6)--(18) (6)--(19);
        \draw (6)--(20) (6)--(21) (6)--(23) (6)--(27) (6)--(29) (6)--(30);
        \draw (6)--(31) (6)--(34) (6)--(37) (6)--(39) (6)--(41) (7)--(10);
        \draw (7)--(12) (7)--(14) (7)--(17) (7)--(18) (7)--(19) (7)--(20);
        \draw (7)--(21) (7)--(22) (7)--(27) (7)--(28) (7)--(30) (7)--(32);
        \draw (7)--(35) (7)--(36) (7)--(39) (7)--(40) (8)--(9) (8)--(12);
        \draw (8)--(13) (8)--(15) (8)--(18) (8)--(19) (8)--(20) (8)--(22);
        \draw (8)--(23) (8)--(28) (8)--(29) (8)--(31) (8)--(32) (8)--(33);
        \draw (8)--(38) (8)--(40) (8)--(41) (9)--(10) (9)--(11) (9)--(12);
        \draw (9)--(13) (9)--(17) (9)--(18) (9)--(21) (9)--(24) (9)--(25);
        \draw (9)--(26) (9)--(27) (9)--(28) (9)--(32) (9)--(33) (9)--(34);
        \draw (9)--(40) (10)--(11) (10)--(12) (10)--(14) (10)--(16) (10)--(20);
        \draw (10)--(23) (10)--(24) (10)--(25) (10)--(26) (10)--(27) (10)--(29);
        \draw (10)--(30) (10)--(33) (10)--(35) (10)--(39) (11)--(13) (11)--(14);
        \draw (11)--(15) (11)--(19) (11)--(22) (11)--(24) (11)--(25) (11)--(26);
        \draw (11)--(28) (11)--(29) (11)--(31) (11)--(34) (11)--(35) (11)--(41);
        \draw (12)--(15) (12)--(19) (12)--(20) (12)--(21) (12)--(24) (12)--(27);
        \draw (12)--(28) (12)--(29) (12)--(30) (12)--(34) (12)--(35) (12)--(37);
        \draw (12)--(38) (13)--(16) (13)--(18) (13)--(20) (13)--(22) (13)--(25);
        \draw (13)--(27) (13)--(28) (13)--(29) (13)--(32) (13)--(33) (13)--(35);
        \draw (13)--(36) (13)--(37) (14)--(17) (14)--(18) (14)--(19) (14)--(23);
        \draw (14)--(26) (14)--(27) (14)--(28) (14)--(29) (14)--(31) (14)--(33);
        \draw (14)--(34) (14)--(36) (14)--(38) (15)--(16) (15)--(17) (15)--(21);
        \draw (15)--(24) (15)--(25) (15)--(29) (15)--(30) (15)--(31) (15)--(33);
        \draw (15)--(36) (15)--(38) (15)--(41) (16)--(22) (16)--(25) (16)--(26);
        \draw (16)--(27) (16)--(30) (16)--(32) (16)--(34) (16)--(37) (16)--(38);
        \draw (16)--(39) (17)--(23) (17)--(24) (17)--(26) (17)--(28) (17)--(31);
        \draw (17)--(32) (17)--(35) (17)--(36) (17)--(37) (17)--(40) (18)--(21);
        \draw (18)--(24) (18)--(28) (18)--(30) (18)--(32) (18)--(33) (18)--(37);
        \draw (18)--(38) (18)--(39) (18)--(41) (19)--(22) (19)--(25) (19)--(29);
        \draw (19)--(31) (19)--(32) (19)--(34) (19)--(36) (19)--(37) (19)--(39);
        \draw (19)--(40) (20)--(23) (20)--(26) (20)--(27) (20)--(30) (20)--(31);
        \draw (20)--(35) (20)--(36) (20)--(38) (20)--(40) (20)--(41) (21)--(23);
        \draw (21)--(24) (21)--(25) (21)--(27) (21)--(28) (21)--(30) (21)--(31);
        \draw (21)--(33) (21)--(34) (21)--(36) (21)--(39) (22)--(25) (22)--(26);
        \draw (22)--(28) (22)--(29) (22)--(30) (22)--(32) (22)--(34) (22)--(35);
        \draw (22)--(38) (22)--(40) (23)--(24) (23)--(26) (23)--(27) (23)--(29);
        \draw (23)--(31) (23)--(32) (23)--(33) (23)--(35) (23)--(37) (23)--(41);
        \draw (24)--(29) (24)--(31) (24)--(32) (24)--(34) (24)--(35) (24)--(37);
        \draw (24)--(38) (24)--(39) (25)--(30) (25)--(31) (25)--(33) (25)--(35);
        \draw (25)--(36) (25)--(37) (25)--(40) (26)--(30) (26)--(32) (26)--(33);
        \draw (26)--(34) (26)--(36) (26)--(38) (26)--(41) (27)--(31) (27)--(32);
        \draw (27)--(33) (27)--(34) (27)--(37) (27)--(38) (27)--(40) (28)--(30);
        \draw (28)--(31) (28)--(34) (28)--(35) (28)--(36) (28)--(37) (28)--(41);
        \draw (29)--(30) (29)--(32) (29)--(33) (29)--(35) (29)--(36) (29)--(38);
        \draw (29)--(39) (30)--(36) (30)--(37) (30)--(38) (30)--(41) (31)--(36);
        \draw (31)--(37) (31)--(38) (31)--(40) (32)--(36) (32)--(37) (32)--(38);
        \draw (32)--(39) (33)--(36) (33)--(39) (33)--(40) (33)--(41) (34)--(38);
        \draw (34)--(39) (34)--(40) (34)--(41) (35)--(38) (35)--(39) (35)--(40);
        \draw (35)--(41) (36)--(39) (36)--(40) (36)--(41) (37)--(39) (37)--(41);
        \draw (38)--(40) (38)--(41) (39)--(40) (39)--(41) (40)--(41);
        \foreach \v in {0,1,...,41} {\node[vertex] at (\v) {};}
    \end{tikzpicture}%
    }%
        \caption{\textsc{HoG 52603}}\label{subfig:52603}
    \end{subfigure}%
    \hfill%
    \begin{subfigure}[b]{0.31\textwidth}
        \centering
    \resizebox{\linewidth}{!}{%
    \begin{tikzpicture}[
        scale=1.45,
        vertex/.style={circle, draw, fill=blue!75, inner sep=1.2pt},
        every edge/.style={draw, line width=0.5pt},
    ]
        \useasboundingbox (-1.5, -1.5) rectangle (1.5, 1.5);
        \node[vertex] (0) at (0.0000, -1.2000) {};
        \node[vertex] (1) at (0.7053, -0.9708) {};
        \node[vertex] (2) at (-0.7053, -0.9708) {};
        \node[vertex] (3) at (0.0000, -0.6000) {};
        \node[vertex] (4) at (-0.7053, 0.9708) {};
        \node[vertex] (5) at (0.0000, 0.6000) {};
        \node[vertex] (6) at (0.7053, 0.9708) {};
        \node[vertex] (7) at (0.0000, 1.2000) {};
        \node[vertex] (8) at (-1.1413, 0.3708) {};
        \node[vertex] (9) at (-0.3527, 0.4854) {};
        \node[vertex] (10) at (1.1413, 0.3708) {};
        \node[vertex] (11) at (0.5706, -0.1854) {};
        \node[vertex] (12) at (0.3527, 0.4854) {};
        \node[vertex] (13) at (-0.5706, -0.1854) {};
        \node[vertex] (14) at (-1.1413, -0.3708) {};
        \node[vertex] (15) at (-0.5706, 0.1854) {};
        \node[vertex] (16) at (-0.3527, -0.4854) {};
        \node[vertex] (17) at (0.5706, 0.1854) {};
        \node[vertex] (18) at (1.1413, -0.3708) {};
        \node[vertex] (19) at (0.3527, -0.4854) {};
        \draw (0)--(1) (0)--(2) (0)--(3) (1)--(18) (1)--(19) (2)--(14);
        \draw (2)--(16) (3)--(15) (3)--(17) (4)--(7) (4)--(8) (4)--(9);
        \draw (5)--(7) (5)--(11) (5)--(13) (6)--(7) (6)--(10) (6)--(12);
        \draw (8)--(14) (8)--(15) (9)--(16) (9)--(17) (10)--(17) (10)--(18);
        \draw (11)--(16) (11)--(18) (12)--(15) (12)--(19) (13)--(14) (13)--(19);
    \end{tikzpicture}%
    }%
        \caption{\textsc{HoG 1036}}
    \end{subfigure}
    \caption{A lower conformally rigid graph (A), an upper conformally rigid graph (B), and a lower and upper conformally rigid graph (C) (also known as the Desargues graph).}\label{fig:cr-examples}
\end{figure}

Now that we know what local obstructions prevent conformal rigidity, we present a few examples of conformally rigid graphs from the \textsc{House of Graphs (HoG)} database~\cite{hog} shown in Figure~\ref{fig:cr-examples}. In this figure and throughout the paper, we provide \textsc{HoG} IDs of any graph we use, which can be easily referenced with the online database. Each of these three graphs visually exhibits a high degree of symmetry as our earlier discussion suggested: no reweighting can exploit any asymmetries because there are none to exploit. It seems that the restrictiveness of conformal rigidity serves as a filter for many graphs that satisfy other nice structural properties. In particular, many conformally rigid graphs are extremal with respect to some combinatorial property, which we will discuss in detail in \S~\ref{sec:main-results}. Thus, it would be nice to identify conformally rigid graphs as these could serve as candidates for graphs that are interesting in another light.

We emphasize, however, that not only is it difficult to find conformally rigid graphs, it is difficult to even show that a given graph is conformally rigid. To prove Examples~\ref{ex:barbell} and~\ref{ex:friendship} were not conformally rigid, we only needed to find a single reweighting that increased $\lambda_2(w)$ or decreased $\lambda_n(w)$. In contrast, to prove conformal rigidity of any of the graphs in Figure~\ref{fig:cr-examples} we need to show that no reweighting away from $w = \one$ can improve the relevant eigenvalue. This highlights two distinct difficulties that any theory of conformal rigidity must address. The first is the \textit{search problem}: as with any highly structured class of graphs, conformally rigid graphs are extremely rare and there is no obvious structural filter for finding them. The second is the \textit{certification problem}: given a graph $G$ suspected to be conformally rigid, how does one rigorously prove it is actually conformally rigid? This amounts to showing $w= \one$ solves an eigenvalue optimization problem that depends on the structure of the graph $G$ in a nontrivial way. 

Clearly the certification problem precedes the search problem. If we compare the search problem for conformally rigid graphs to finding a needle in a haystack, then the certification problem is akin to being unable to easily distinguish between a needle and a piece of hay. Thus, our paper primarily develops theory for conformal rigidity that addresses the certification problem.

\subsection{Known Results} Prior to our results, conformal rigidity was hard to verify rigorously for most conformally rigid graphs. It was known that specific families, such as edge-transitive graphs~\cite[Proposition 2.1]{steinerberger} and 1-walk regular graphs~\cite[Theorem 3.2]{gouveia} are always both lower and upper conformally rigid. Outside of these classes, however, no other known structural results guarantee conformal rigidity. Any other known example had to be certified on a case-by-case basis, typically involving manual inspection or ad-hoc techniques depending on the structure of each individual graph. The paper~\cite{gouveia} leveraged the connection between symmetry and conformal rigidity to help simplify the certification of certain highly symmetric graphs --- vertex-transitive and Cayley graphs to be specific --- but without a theoretical framework explaining when or why this simplification is possible.

We address these issues by building a robust framework for conformal rigidity embedded in classical theory. On the theoretical side, subdifferential analysis and representation theory provide the backbone; on the computational side, semidefinite programming, symmetry reduction, and real algebraic geometry supply the tools for certification. This framework anchors the simplified method of certification developed in~\cite{gouveia} to the geometry of the spectral functions $\lambda_2(w)$ and $\lambda_n(w)$ while also guaranteeing that it works for all vertex-transitive graphs and even a larger class of graphs than previously considered. In particular, our framework subsumes the certification of five ``sporadic'' conformally rigid graphs in~\cite{gouveia} that required the previously mentioned ad-hoc techniques outside the main theory of that paper, bringing them under a single unified treatment.

Our results serve as a starting point for deeper open problems in conformal rigidity. With certification now well-understood for large classes of conformally rigid graphs, attention can shift toward the search problem and toward a more fundamental question: what causes conformal rigidity? Is there a structural or combinatorial property that characterizes it, or are some graphs conformally rigid for no deeper reason than the arithmetic of their spectra? The five sporadic examples illustrate this tension: our framework leverages symmetry to certify their conformal rigidity, yet still offers no explanation for why they are conformally rigid. One promising direction is the observed connection between conformal rigidity and extremal combinatorial properties. Many conformally rigid graphs appear to be extremal in some sense, and making this connection rigorous could shed light on conformal rigidity through the lens of extremal graph theory, and conversely, offer new perspectives on extremal problems through the spectral viewpoint.

\subsection{Organization of the Paper} In \S~\ref{sec:main-results}, we give an overview of our main results with examples on how to apply them. \S~\ref{sec:subdifferential} develops our main framework for conformal rigidity using subdifferential theory. We also explain how edge-isometric spectral embeddings can be certified using semidefinite programming. \S~\ref{sec:perturbation-theory} revisits the main characterization through a perturbation theory lens, giving a variational intuition for the edge-isometry condition as an equilibrium: no weight perturbation can simultaneously increase (decrease) the energy of every eigenvector in the second or last eigenspace. In \S~\ref{sec:sym} we apply techniques of symmetry reduction to characterize conformal rigidity at the level of edge orbits. In \S~\ref{sec:rep-theory} we use representation theory to show that the conformal rigidity of vertex-transitive graphs can be certified by a single eigenvector, answering a question posed in~\cite{gouveia}. This enables exact computational certification through the use of Gr\"{o}bner bases. \S~\ref{sec:polyhedral-cone} specializes to a class of graphs --- including all abelian Cayley graphs --- where exact certification reduces to linear feasibility.

\setcounter{secnumdepth}{1}
\section{Summary of Main Results}\label{sec:main-results}
In this section, we will summarize our results at a high level, grouped  by the section that they are proven in. A running example demonstrates each theorem.

\subsection{\S~\ref{sec:subdifferential} Subdifferentials and Conformal Rigidity} 
In \S~\ref{sec:subdifferential}, we give our first characterization of conformal rigidity, which applies to any graph $G$ with no assumptions of symmetry. We first establish the following notation:

\begin{definition}\label{def:ell}
    For a function $\varphi: V \to \R$, we denote its \textit{edge-energy vector} by 
    \begin{equation}
        \ell(\varphi) = \paren{{(\varphi(u) - \varphi(v))}^2}_{uv \in E}
        \label{eq:ell}.
    \end{equation}    
\end{definition}
We think of the $uv$-coordinate $\ell(\varphi)_{uv}$ as the energy or squared length of the edge $uv \in E$ under the vector $\varphi$. Our first theorem states that a graph $G$ is conformally rigid if and only if we can find a collection of eigenvectors in the eigenspace $\E_\lambda$ for $\lambda = \lambda_2, \lambda_n$ whose total edge-energy is constant across each edge.

\begin{namedtheorem}{\ref{thm:cr-subgradient}}
    A graph $G$ is lower (upper) conformally rigid if and only if there exist nonzero eigenvectors $\varphi_1, \ldots, \varphi_r \in \E_{\lambda_2}$ ($\E_{\lambda_n}$), coefficients $a_i \geq 0$, and a constant $c >0$ such that for all $uv \in E$,
    \begin{equation}\label{eq:edge-iso}
      c = \sum_{i=1}^r a_i {(\varphi_i(u) - \varphi_i(v))}^2 = \sum_{i=1}^r a_i \ell(\varphi_i)_{uv}.
    \end{equation}
\end{namedtheorem}

The proof of this theorem relies on the fact that the spectral functions $-\lambda_2(w)$ and $\lambda_n(w)$ are convex with respect to $w$. Applying subdifferential theory to these functions yields the above result. This theorem recovers the known result that conformal rigidity is equivalent to the existence of \textit{edge-isometric spectral embeddings} of $G$ on the eigenspaces $\E_\lambda$ where $\lambda = \lambda_2, \lambda_n$. 

\begin{definition}
    Let $\lambda > 0$ be a Laplacian eigenvalue of $G$ with multiplicity $d$, and let $\E_\lambda \cong \R^d$ be the corresponding eigenspace. For eigenvectors $\varphi_1, \ldots, \varphi_r \in \E_\lambda$, let $P= \begin{pmatrix}
        \varphi_1 & \cdots & \varphi_r
    \end{pmatrix} \in \R^{n \times r}$. A \textit{spectral embedding of $G$ on $\E_\lambda$} is the collection $\mathcal{P} = \set{p_v}_{v\in V} \subseteq \R^r$ where $p_v^T$ is the $v^{\textrm{th}}$ row of $P$.
\end{definition}

\begin{definition}\label{def:edge-iso}
    We call a spectral embedding $\mathcal{P}$ \textit{edge-isometric} if there exists some $c > 0$ such that $\|p_u - p_v\| = c$ for all $uv \in E$.
\end{definition}

Geometrically, we can visualize a spectral embedding by mapping each vertex $v$ of $G$ to its corresponding vector $p_v \in \R^r$. An edge-isometric spectral embedding is one in which the distance between two vertices sharing an edge is the same (positive) value for all edges in $G$. The term ``embedding'' is a bit of a misnomer as the map does not have to be injective as we will see in the following example.

\begin{namedtheorem}{\ref{thm:edge-iso}}\cite[Proposition 4.3]{steinerberger}
    A graph $G$ is lower (upper) conformally rigid if and only if $G$ has an edge-isometric embedding $\mathcal{P}$ on $\E_{\lambda_2}$ ($\E_{\lambda_n})$.
\end{namedtheorem}

\begin{example}\label{ex:running}
    Let $G = \Cay(\Z_{21}, \set{1,6})$, shown in Figure~\ref{subfig:circulant}. This is an example of a Cayley graph $\Cay(\Gamma,S)$ where $\Gamma$ is a finite group, and $S \subseteq \Gamma$ is a \textit{generating set}. The vertices of $\Cay(\Gamma,S)$ are given by the elements of $\Gamma$ with edges given by $(g , g \circ s)$ for $g \in \Gamma$ and $s \in S$. Since we care about undirected graphs, we assume $S$ is symmetric i.e. $s \in S \implies s^{-1} \in S$. The graph $G$ is both lower and upper conformally rigid as shown in~\cite[Proposition 5.7]{gouveia}.
    
    \definecolor{vcolA}{HTML}{DFF3FF}
    \definecolor{vcolB}{HTML}{A7E3FF}
    \definecolor{vcolC}{HTML}{5BC1FF}
    \definecolor{vcolD}{HTML}{2F86FF}
    \definecolor{vcolE}{HTML}{1E5FD0}
    \definecolor{vcolF}{HTML}{4A9BFF}
    \definecolor{vcolG}{HTML}{9EDCFF}
    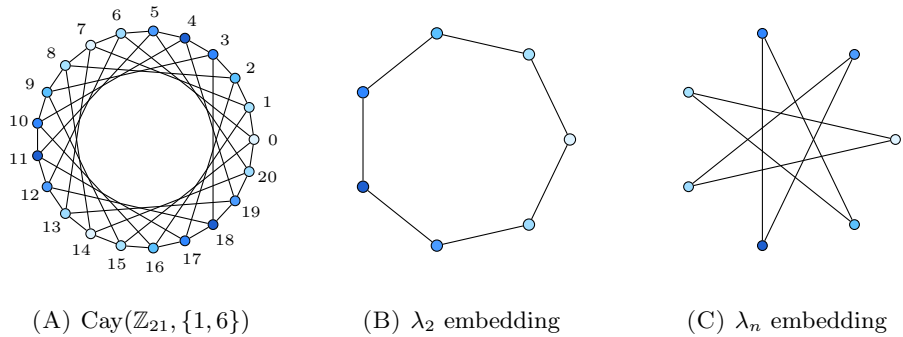
\begin{figure}[h]
        \centering
        \begin{subfigure}[c]{0.31\textwidth}
            \centering
            \begin{tikzpicture}[
                scale=0.72,
                vertex/.style={circle, draw, fill=black, inner sep=1.3pt},
            ]
            \path[use as bounding box] (-2.8,-2.8) rectangle (2.8,2.8);
            \def\r{2.0}
            \def\rlabel{2.35}
            \foreach \k/\col in {
                0/vcolA, 1/vcolB, 2/vcolC, 3/vcolD, 4/vcolE, 5/vcolF, 6/vcolG,
                7/vcolA, 8/vcolB, 9/vcolC, 10/vcolD, 11/vcolE, 12/vcolF, 13/vcolG,
                14/vcolA, 15/vcolB, 16/vcolC, 17/vcolD, 18/vcolE, 19/vcolF, 20/vcolG} {
                \pgfmathsetmacro{\angle}{360*\k/21}
                \node[vertex, fill=\col] (V\k) at ({\r*cos(\angle)}, {\r*sin(\angle)}) {};
                \node[font=\tiny] at ({\rlabel*cos(\angle)}, {\rlabel*sin(\angle)}) {$\k$};
            }
            \foreach \i in {0,...,20} {
                \pgfmathtruncatemacro{\j}{mod(\i+1,21)}
                \draw (V\i) -- (V\j);
            }
            \foreach \i in {0,...,20} {
                \pgfmathtruncatemacro{\j}{mod(\i+6,21)}
                \draw (V\i) -- (V\j);
            }
            \end{tikzpicture}
            \subcaption{$\Cay(\Z_{21},  \set{1,6})$}\label{subfig:circulant}
        \end{subfigure}
        \hfill%
        \begin{subfigure}[c]{0.34\textwidth}
            \centering
            \begin{tikzpicture}[
                scale=0.72,
                vertex/.style={circle, draw, fill=black, inner sep=1.5pt},
            ]
            \path[use as bounding box] (-2.8,-2.8) rectangle (2.8,2.8);
            \def\r{2.0}
            \def\rlabel{2.8}
            \foreach \k/\col in {0/vcolA, 1/vcolB, 2/vcolC, 3/vcolD, 4/vcolE, 5/vcolF, 6/vcolG} {
                \pgfmathsetmacro{\angle}{360*\k/7}
                \node[vertex, fill=\col] (V\k) at ({\r*cos(\angle)}, {\r*sin(\angle)}) {};
            }
            \foreach \i in {0,...,6} {
                \pgfmathtruncatemacro{\j}{mod(\i+1,7)}
                \draw (V\i) -- (V\j);
            }
            \end{tikzpicture}
            \subcaption{$\lambda_2$ embedding}\label{subfig:lambda2}
        \end{subfigure}%
        \hfill%
        \begin{subfigure}[c]{0.34\textwidth}
            \centering
            \begin{tikzpicture}[
                scale=0.72,
                vertex/.style={circle, draw, fill=black, inner sep=1.3pt},
            ]
            \path[use as bounding box] (-2.8,-2.8) rectangle (2.8,2.8);
            \def\r{2.0}
            \def\rlabel{2.8}
            \foreach \k/\col in {0/vcolA, 1/vcolF, 2/vcolD, 3/vcolB, 4/vcolG, 5/vcolE, 6/vcolC} {
                \pgfmathsetmacro{\angle}{360*\k/7}
                \node[vertex, fill=\col] (V\k) at ({\r*cos(\angle)}, {\r*sin(\angle)}) {};
            }
            \foreach \i in {0,...,6} {
                \pgfmathtruncatemacro{\j}{mod(\i+3,7)}
                \draw (V\i) -- (V\j);
            }
            \end{tikzpicture}
            \subcaption{$\lambda_n$ embedding}\label{subfig:lambdan}
        \end{subfigure}
        \caption{$\Cay(\Z_{21},  \set{1,6})$ and edge-isometric embeddings on $\E_{\lambda_2}$ and $\E_{\lambda_n}$.}\label{fig:circulant}
    \end{figure}
    To certify lower conformal rigidity, we must find eigenvectors $\varphi_i \in \E_{\lambda_2}$ such that~\eqref{eq:edge-iso} holds for all $m=42$ edges. One can verify that $\lambda_2 = 4(1 - \cos \paren{2 \pi/7}) $ with multiplicity $2$. The following eigenvectors
    \[\varphi_1 = \paren{\cos \paren{\frac{2\pi j}{7}}}_{j=0}^{20}, \  \varphi_2 = \paren{\sin \paren{\frac{2\pi j}{7}}}_{j=0}^{20}.\]
    satisfy~\eqref{eq:edge-iso} where $c=2\bigl(1 - \cos(2\pi/7)\bigr) \approx 0.753$ and thus give an edge-isometric embedding with $P = \begin{pmatrix}
        \varphi_1 & \varphi_2
    \end{pmatrix}$, on $\E_{\lambda_2}$. The embedding gives a 3-to-1 map onto the regular heptagon shown in Figure~\ref{subfig:lambda2}.
    Similarly, $\lambda_n = 4(1- \cos\paren{6 \pi/7})$ with multiplicity 2. The eigenvectors
    \[\varphi_1' = \paren{\cos \paren{\frac{6\pi j}{7}}}_{j=0}^{20}, \  \varphi_2' = \paren{\sin \paren{\frac{6\pi j}{7}}}_{j=0}^{20},\]
    satisfy~\eqref{eq:edge-iso} for $\lambda_n$ with $c = 2(1-\cos \paren{\frac{6\pi}{7}}) \approx 3.802$. The corresponding edge-isometric embedding on $\E_{\lambda_n}$ gives a 3-to-1 vertex map on the regular star-shaped heptagon shown in Figure~\ref{subfig:lambdan}.
\end{example}

Edge-isometric embeddings formalize the notion that every edge in $G$ must look the same to $\lambda_2$ or $\lambda_n$. Equation~\eqref{eq:edge-iso} says that we can balance the energy across all edges with a collection of eigenvectors $\varphi_1, \ldots, \varphi_r$. This corresponds to all edge lengths being the same in the corresponding spectral embedding. At this stage, we do not explain how we find these eigenvectors, but one can imagine the difficulty of finding an unspecified number of $r$ eigenvectors satisfying $\abs{E}$ equations in general.

\subsection{\S~\ref{sec:perturbation-theory} A Perturbation Theory Perspective}

In \S~\ref{sec:perturbation-theory}, we give an alternative proof of Theorem~\ref{thm:cr-subgradient} that provides variational intuition for the edge-isometry condition. Suppose we have eigenvectors $\varphi_1, \ldots, \varphi_r \in \E_{\lambda_2}$ satisfying the condition in Theorem~\ref{thm:cr-subgradient}. We show that for any perturbation of weights $w = \one + y$ with $\sum_{uv} y_{uv} = 0$, there exists $1 \leq i \leq r$ such that $\langle \varphi_i, L(w) \varphi_i \rangle \leq \langle \varphi_i, L(\one) \varphi_i \rangle = \lambda_2$, which implies $\lambda_2(w) \leq \lambda_2$. In other words, we may be able to perturb the weights $w$ to increase the energy of some eigenvector in $\E_{\lambda_2}$, but we cannot do so simultaneously for all eigenvectors in $\E_{\lambda_2}$. Conversely, if no collection of eigenvectors satisfies Theorem~\ref{thm:cr-subgradient}, we use a separating hyperplane argument to show that there does exist a direction of perturbation $y$ that increases the energy of every eigenvector in $\E_{\lambda_2}$ simultaneously. A careful but routine analytical argument shows that a small enough step in this direction increases $\lambda_2(w)$, which shows $G$ is not lower conformally rigid. The same arguments apply to $\lambda_n$ and upper conformal rigidity.

\subsection{\S~\ref{sec:sym} Symmetry Reduction}\label{subsec:symmetry-reduction} 

In \S~\ref{sec:sym}, we assume that we are working with a graph $G$ that contains a large group of automorphisms $\Psi$. Recall that a graph automorphism is a permutation $\sigma$ of the vertices $V$ such that $(u,v) \in E$ if and only if $(\sigma(u), \sigma(v)) \in E$. If we denote $\Om_1, \ldots, \Om_s$ as the edge orbits of $G$ under the action of $\Psi$, we can introduce the orbit-level analogue of $\ell(\varphi)$:
\begin{definition}\label{def:orbit-energy}
    Let $G=(V,E)$ be a connected graph, and let $\Psi \leq \Aut(G)$ denote a group of automorphisms with edge orbits $\Om_1, \ldots, \Om_s$. For a function $\varphi: V \to \R$, we denote its \textit{orbit-energy vector} by
    \[
        \ell_\Psi(\varphi) = \paren{\sum_{uv \in \Om_i} (\varphi(u) - \varphi(v))^2}_{i=1}^s \in \R^s.
    \]
\end{definition}
We think of the coordinate $\ell_\Psi(\varphi)_i$ as the total energy of $\varphi$ on edges in the orbit $\Om_i$. We then get the symmetry reduced version of Theorem~\ref{thm:cr-subgradient}.
\begin{namedtheorem}{\ref{thm:sym-subgradient}}
    A graph $G$ with a group of automorphisms $\Psi \leq \Aut(G)$ is lower (upper) conformally rigid if and only if there exist nonzero eigenvectors $\varphi_1, \ldots, \varphi_r \in \E_{\lambda_2}$ ($\E_{\lambda_n}$), coefficients $a_i \geq 0$, and $c > 0$ such that for all $i=1, \ldots, s$
    \begin{equation}\label{eq:orbit-iso}
      c \cdot \abs{\Om_i} = \sum_{j=1}^r a_j \sum_{uv \in \Om_i}(\varphi_j(u)-\varphi_j(v))^2 = \sum_{j=1}^r a_j \ell_\Psi(\varphi_j)_i.
    \end{equation}
\end{namedtheorem}
We prove this by showing that we only need to consider weights $w$ that are invariant under the group action $\Psi$, and then we can apply an identical subdifferential argument as in the previous section. By applying symmetry reduction, we only need to find a collection of eigenvectors satisfying $s$ equations instead of $\abs{E}$ equations. When $\Psi$ is large, $s$ can be a lot smaller than $\abs{E}$.
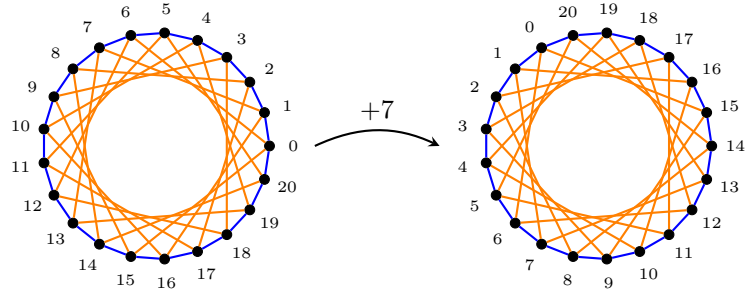
\begin{figure}[h]
        \centering
        \begin{tikzpicture}[vertex/.style={circle, draw, fill=black, inner sep=1.3pt}, scale=0.75]
        \def\r{2}
        \def\rlabel{2.42}
        \def\sep{3.9}

        \foreach \k in {0,...,20} {
            \pgfmathsetmacro{\angle}{360*\k/21}
            \node[vertex] (L\k) at ({-\sep + \r*cos(\angle)}, {\r*sin(\angle)}) {};
            \node[font=\tiny] at ({-\sep + \rlabel*cos(\angle)}, {\rlabel*sin(\angle)}) {$\k$};
        }
        \foreach \i in {0,...,20} {
            \pgfmathtruncatemacro{\j}{mod(\i+1,21)}
            \draw[blue, thick] (L\i) -- (L\j);
        }
        \foreach \i in {0,...,20} {
            \pgfmathtruncatemacro{\j}{mod(\i+6,21)}
            \draw[orange, thick] (L\i) -- (L\j);
        }

        \foreach \k in {0,...,20} {
            \pgfmathsetmacro{\angle}{360*\k/21}
            \pgfmathtruncatemacro{\newlabel}{mod(\k-7+21,21)}
            \node[vertex] (R\k) at ({\sep + \r*cos(\angle)}, {\r*sin(\angle)}) {};
            \node[font=\tiny] at ({\sep + \rlabel*cos(\angle)}, {\rlabel*sin(\angle)}) {$\newlabel$};
        }
        \foreach \i in {0,...,20} {
            \pgfmathtruncatemacro{\j}{mod(\i+1,21)}
            \draw[blue, thick] (R\i) -- (R\j);
        }
        \foreach \i in {0,...,20} {
            \pgfmathtruncatemacro{\j}{mod(\i+6,21)}
            \draw[orange, thick] (R\i) -- (R\j);
        }

        \draw[->, >=stealth, thick]
            (-1.1, 0)
            to[bend left=25]
            node[above, midway] {$+7$}
            (1.1, 0);

        \end{tikzpicture}
        \caption{The map $i \mapsto i+7 \pmod{21}$ is a graph automorphism of $G = \Cay(\Z_{21}, \set{1,6})$. Edge orbits are labeled by color.}\label{fig:automorphism}
    \end{figure}
We introduce the notion of an \textit{orbit-isometric} spectral embedding as the symmetry reduced analogue of edge-isometric embeddings.
\begin{definition}\label{def:orbit-iso}
    We call a spectral embedding $\mathcal{P}$ \textit{orbit-isometric} with respect to $\Psi$ if there exists some $c > 0$ such that 
    \begin{equation}\label{eq:orbit-isometry}
        \frac{1}{\abs{\Om_i}}\sum_{uv \in \Om_i} \|p_u - p_v \|^2 = c
    \end{equation} 
    for all edge orbits $\Om_1, \ldots, \Om_s$ of $\Psi$. Equivalently, a spectral embedding $\mathcal{P}$ is orbit-isometric if the average squared length of an edge in each orbit is constant.
\end{definition}

Instead of balancing edge-energy of eigenvectors, we just need the average energy per edge in each orbit to be constant. This leads to the symmetry reduced version of Theorem~\ref{thm:edge-iso}.

\begin{namedtheorem}{\ref{thm:orbit-iso}}
    A graph $G$ with a group of automorphisms $\Psi \leq \Aut(G)$ is lower (upper) conformally rigid if and only if there exists an orbit-isometric embedding $\mathcal{P}$ on $\E_{\lambda_2}$ ($\E_{\lambda_n}$).
\end{namedtheorem}

\begin{rexample}[continued]\label{rex:2}
    By construction, any element $k \in \Z_{21}$ acts on $G$ by graph automorphisms through addition: $i \mapsto i+k$. Figure~\ref{fig:automorphism} illustrates this for $k=7$. Under the action of $\Z_{21}$, we have two edge orbits of size 21 given by the edge type $s =1,6$. Explicitly, $\Om_1 = \set{(i, i+1) : 0 \leq i \leq 20}$ (colored in blue) and $\Om_6 = \set{(i, i+6) : 0 \leq i \leq 20}$ (colored in orange). If we denote $\Om = (\abs{\Om_1}, \abs{\Om_6})$ we need to find eigenvectors $\varphi_1, \ldots \varphi_r \in \E_{\lambda_2}$ satisfying 
    \[\sum_{i=1}^r a_i \ell_{\Z_{21}}(\varphi_i) = c \cdot \Om = c \cdot (21,21)\]
    for some $ c \in \R$. Instead of finding a collection of eigenvectors satisfying 42 constraints, we now just have to satisfy 2 constraints. Here, the single eigenvector $\varphi_1 = \paren{\cos \paren{2\pi j/7}}_{j=0}^{20}$ satisfies~\eqref{eq:orbit-iso} with $c=\lambda_2/2$.
    We illustrate the length of each edge in the corresponding spectral embedding $P = \begin{pmatrix}
        \varphi_1
    \end{pmatrix}$ on $\E_{\lambda_2}$ in Figure~\ref{fig:orbit-iso}. Even though the length of edges may vary in this embedding, the total squared length in each orbit is the same. Another way to see this is that this 1-dimensional embedding is just the projection of the heptagon embedding in Figure~\ref{subfig:lambda2} onto the $x$-axis. Since each edge in the heptagon is covered by 3 edges in $\Om_1$ and 3 edges in $\Om_6$, both orbits have the same total energy. One can verify with the same logic that $\begin{pmatrix}
        \varphi_1'
    \end{pmatrix}$ gives an orbit-isometric embedding on $\E_{\lambda_n}$. Again, it is still not obvious how to find an orbit-isometric embedding in general, but it is clearly easier to find than an edge-isometric embedding. 
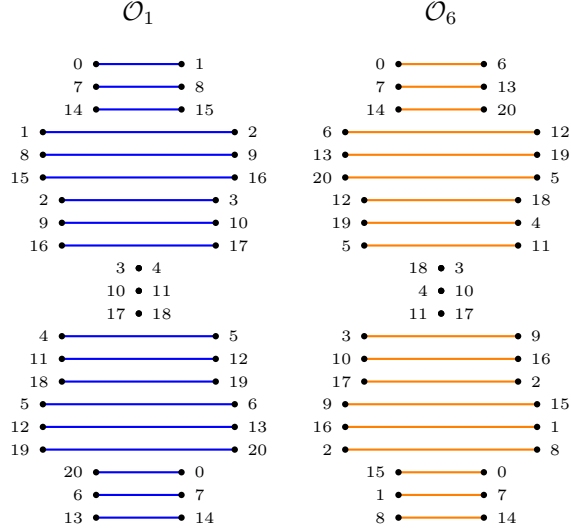
\begin{figure}
    \centering
    \begin{tikzpicture}[scale=1.0]

    \def\vsep{0.30}
    \def\xO{1.2}
    \def\xT{5.2}

    \def\la{0.3765}  
    \def\lb{0.8460}  
    \def\lc{0.6784}  
    \def\ld{0.0000}  

    \def\sc{3}

    \node[above] at (\xO, 21*\vsep + 0.1) {$\Om_1$};
    \foreach \idx/\a/\b/\hl in {
        0/0/1/\la,
        1/7/8/\la,
        2/14/15/\la,
        3/1/2/\lb,
        4/8/9/\lb,
        5/15/16/\lb,
        6/2/3/\lc,
        7/9/10/\lc,
        8/16/17/\lc,
        9/3/4/\ld,
        10/10/11/\ld,
        11/17/18/\ld,
        12/4/5/\lc,
        13/11/12/\lc,
        14/18/19/\lc,
        15/5/6/\lb,
        16/12/13/\lb,
        17/19/20/\lb,
        18/20/0/\la,
        19/6/7/\la,
        20/13/14/\la
    } {
        \pgfmathsetmacro{\y}{(20 - \idx)*\vsep}
        \pgfmathsetmacro{\hls}{\hl*\sc/2}
        \draw[blue, thick] ({\xO - \hls}, \y) -- ({\xO + \hls}, \y);
        \fill[black] ({\xO - \hls}, \y) circle (1.2pt);
        \fill[black] ({\xO + \hls}, \y) circle (1.2pt);
        \node[left, font=\tiny] at ({\xO - \hls - 0.05}, \y) {$\a$};
        \node[right, font=\tiny] at ({\xO + \hls + 0.05}, \y) {$\b$};
    }

    \node[above] at (\xT, 21*\vsep + 0.1) {$\Om_6$};
    \foreach \idx/\a/\b/\hl in {
        0/0/6/\la,
        1/7/13/\la,
        2/14/20/\la,
        3/6/12/\lb,
        4/13/19/\lb,
        5/20/5/\lb,
        6/12/18/\lc,
        7/19/4/\lc,
        8/5/11/\lc,
        9/18/3/\ld,
        10/4/10/\ld,
        11/11/17/\ld,
        12/3/9/\lc,
        13/10/16/\lc,
        14/17/2/\lc,
        15/9/15/\lb,
        16/16/1/\lb,
        17/2/8/\lb,
        18/15/0/\la,
        19/1/7/\la,
        20/8/14/\la
    } {
        \pgfmathsetmacro{\y}{(20 - \idx)*\vsep}
        \pgfmathsetmacro{\hls}{\hl*\sc/2}
        \draw[orange, thick] ({\xT - \hls}, \y) -- ({\xT + \hls}, \y);
        \fill[black] ({\xT - \hls}, \y) circle (1.2pt);
        \fill[black] ({\xT + \hls}, \y) circle (1.2pt);
        \node[left, font=\tiny] at ({\xT - \hls - 0.05}, \y) {$\a$};
        \node[right, font=\tiny] at ({\xT + \hls + 0.05}, \y) {$\b$};
    }

    \end{tikzpicture}
    \caption{Edge lengths of $\Cay(\Z_{21}, \set{1,6})$ under the orbit-isometric embedding given by $\begin{pmatrix}\varphi_1\end{pmatrix}$.}
    \label{fig:orbit-iso}
\end{figure}
\end{rexample}

\subsection{\S~\ref{sec:rep-theory} Representation Theory}
The edge-isometric condition in~\eqref{eq:edge-iso} and the orbit-isometric condition in~\eqref{eq:orbit-iso} both require finding a collection of eigenvectors $\varphi_1,  \ldots, \varphi_r$ in the appropriate eigenspace whose weighted total energy is constant across each edge or orbit. One main difficulty --- apart from actually finding the correct eigenvectors --- is that we do not know a priori the number $r$ of eigenvectors we might need. Symmetry reduction expectedly reduces the complexity of this search. In our example, we needed $r=2$ eigenvectors to satisfy the $\abs{E} = 42$ edge-constraints on $G$. After applying symmetry reduction, we only needed $r=1$ eigenvector to satisfy the $s=2$ edge orbit constraints.

If we somehow knew that we could find a solution with $r=1$, the search for a certificate would be considerably simpler: rather than searching for an unspecified number of eigenvectors, we could just solve a system of quadratic equations on $\E_{\lambda_2}$ ($\E_{\lambda_n}$). In \S~\ref{sec:rep-theory}, we apply techniques from representation theory to show that under certain conditions on $G$ and $\Psi$, we can guarantee such a solution for an orbit-isometric embedding.

\begin{namedtheorem}{\ref{thm:vertex-transitive}}
    If $G$ is vertex-transitive with respect to $\Psi \leq \Aut(G)$, then $G$ is lower (upper) conformally rigid if and only if there exists one eigenvector $\varphi \in \E_{\lambda_2} (\E_{\lambda_n})$ such that for all edge orbits $\Om_1, \ldots, \Om_s$,
    \[\abs{\Om_i} = \sum_{uv \in \Om_i} (\varphi(u) - \varphi(v))^2 = \ell_\Psi(\varphi)_i.\]
\end{namedtheorem}

\begin{rexample}[continued]
    $G$ is vertex-transitive with respect to $\Z_{21}$: looking at Figure~\ref{fig:automorphism}, it is clear any vertex can be rotated to any other vertex. Rephrasing the previous theorem in another way, this means that $G$ is lower conformally rigid if and only if $c \Om \in \set{\ell_{\Z_{21}}(\varphi) : \varphi \in \E_{\lambda_2}}$. Thus, we are justified in only checking the orbit-energy vectors of individual eigenvectors to certify conformal rigidity. As we saw in the previous example, $\ell_{\Z_{21}}(\varphi_1) = (21,21) = \Om$.
\end{rexample}

This theorem follows from a more general, representation-theoretic result that depends on the decomposition of $\E_{\lambda}$ into real irreducible representations. In \S~\ref{subsec:vertex-transitive}, we justify why the more general result typically holds for highly symmetric graphs that are not necessarily vertex-transitive. This resolves the question whether the necessary condition for conformal rigidity given in Theorem 2.3 in~\cite{steinerberger} is also sufficient --- it is --- and certifies the conformal rigidity of the sporadic graphs in~\cite{gouveia}. 

\subsection{\S~\ref{sec:polyhedral-cone} Symmetry Adapted Bases and Block-Diagonalization} For a general vertex-transitive graph, finding the certifying eigenvector $\varphi$ guaranteed in the previous section requires solving a system of quadratic equations on $\E_\lambda$. When $G$ is an abelian Cayley graph, the set of orbit-energy vectors is completely determined by a small number of eigenvectors. 
\begin{namedtheorem}{\ref{thm:abelian-cayley}}
    Let $\Gamma$ be an abelian group and $G = \Cay(\Gamma, S)$ be a Cayley graph. The set $\ell_\Psi(\E_{\lambda}) = \set{\ell_\Psi(\varphi) : \varphi \in \E_\lambda}$
    is a polyhedral cone.
\end{namedtheorem}

The above statement again follows from a more general representation-theoretic result, Theorem~\ref{thm:polyhedral-cone}, which depends on the decomposition of $\E_\lambda$ into real irreducibles. We remark that a version of Theorem~\ref{thm:abelian-cayley} is proven in~\cite[Theorem 5.4]{gouveia} for complex eigenvectors.

\begin{rexample}[continued]
    Since $\Cay(\Z_{21}, \set{1,6})$ is an abelian Cayley graph, the preceding theorem applies. Since $\E_{\lambda_2}$ is only comprised of one irreducible real representation, $\ell_{\Z_{21}}(\E_{\lambda_2})$ is the ray generated by $\ell_{\Z_{21}}(\varphi_1)$. We will justify this fully in \S~\ref{sec:polyhedral-cone}. This means that the orbit-energy of any vector in $\E_{\lambda_2}$ is a scalar multiple of $\ell_{\Z_{21}}(\varphi_1)$, so any choice of eigenvector gives a 1-dimensional orbit-isometric embedding on $\E_{\lambda_2}$. Similarly, $\ell_{\Z_{21}}(\E_{\lambda_n})$ is also a ray generated by $\ell_{\Z_{21}}(\varphi_1')$.
\end{rexample}

\begin{figure}[ht]
\centering
\begin{tikzpicture}[
    scale=0.85,
    vertex/.style={circle, draw, fill=black, inner sep=1.3pt},
]
\def\r{2}
\def\rlabel{2.35}

\foreach \k in {0,...,11} {
    \pgfmathsetmacro{\angle}{30*\k}
    \node[vertex] (V\k) at ({\r*cos(\angle)}, {\r*sin(\angle)}) {};
    \node at ({\rlabel*cos(\angle)}, {\rlabel*sin(\angle)}) {$\k$};
}

\foreach \i in {0,...,11} {
    \pgfmathtruncatemacro{\j}{mod(\i+2,12)}
    \draw[thick, blue] (V\i) -- (V\j);
}

\foreach \i in {0,...,11} {
    \pgfmathtruncatemacro{\j}{mod(\i+3,12)}
    \draw[thick, orange] (V\i) -- (V\j);
}

\end{tikzpicture}
\caption{$G = \Cay(\mathbb{Z}_{12}, \{2,3\})$ with color-coded edge orbits.}\label{fig:running2}
\end{figure}
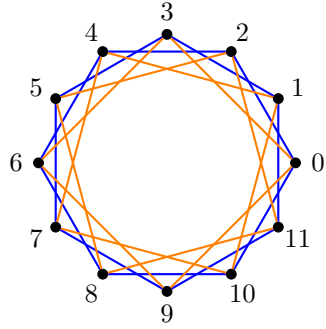

Our running example is a bit too simple to see how to use Theorem~\ref{thm:abelian-cayley} since any choice of eigenvector could certify lower and upper conformal rigidity. We now go over a slightly more complicated example.

\begin{example}\label{ex:running2}
    Let $G$ be the lower conformally rigid circulant $\Cay(\Z_{12}, \set{2,3})$ (\textsc{HoG 33319}) shown in Figure~\ref{fig:running2}. This is vertex-transitive with respect to $\Z_{12}$ with edge orbits $\Om_2$ and $\Om_3$ corresponding to the two generators; both contain 12 edges. One can check $\lambda_2 = 3$ with multiplicity 6 with a basis given by 
    \[\varphi^k_1 = \paren{\cos \paren{\frac{\pi jk}{6}}}_{j=0}^{11}, \  \varphi^k_2 = \paren{\sin \paren{\frac{\pi jk}{6}}}_{j=0}^{11}\]
    for $k = 1,4,5$. Since $G$ is an abelian Cayley graph, we know that the set of orbit-energy distributions
    $\ell_{\Z_{12}}(\E_{\lambda_2}) =\left\{\ell_{\Z_{12}}(\varphi) : \varphi \in \E_{\lambda_2} \right\} \subseteq \R^2$ is a convex polyhedral cone. We will later show that this cone has extreme rays generated by 
    \[\ell_{\Z_{12}}(\varphi_1^1) = (6,12), \ \ell_{\Z_{12}}(\varphi_1^4) = (18,0), \  \ell_{\Z_{12}}(\varphi_1^5) =  (6,12).\]
    In particular, any orbit-energy distribution must be a conical combination
    \[a_1\ell_{\Z_{12}}(\varphi_1^1) + a_2\ell_{\Z_{12}}(\varphi_1^4) + a_3 \ell_{\Z_{12}}(\varphi_1^5) = a_1\cdot (6,12) + a_2 \cdot(18,0) + a_3\cdot (6,12)\]
    where the coefficients $a_i \geq 0$ are all nonnegative. Moreover, representation theory will tell us that the $\varphi_1^k$ diagonalize $\ell_{\Z_{12}}$, that is,
    \[ \ell_{\Z_{12}}(a_1\varphi_1^1+a_2\varphi_1^4+a_3\varphi_1^5) = a_1^2\ell_{\Z_{12}}(\varphi_1^1) + a_2^2\ell_{\Z_{12}}(\varphi_1^4) + a_3^2 \ell_{\Z_{12}}(\varphi_1^5).\]
    \begin{table}[h]
    \centering
    \renewcommand{\arraystretch}{1.4}
    \begin{tabular}{p{0.24\textwidth}|p{0.34\textwidth}|p{0.34\textwidth}}
    \toprule
    \textbf{Result} & \textbf{Unsymmetrized} & \textbf{Symmetrized} \\
    \midrule
    Theorems~\ref{thm:cr-subgradient}, \ref{thm:sym-subgradient}
    & $c\one \in \Conv(\ell(\E_\lambda))$
    & $c\Om \in \Conv(\ell_\Psi(\E_\lambda))$ \\
    \midrule
    Theorems~\ref{thm:edge-iso}, \ref{thm:orbit-iso}
    & Edge-isometric embedding on $\E_\lambda$
    & Orbit-isometric embedding on $\E_\lambda$ \\
    \midrule
    & Satisfies $|E|$ edge constraints
    & Satisfies $s \ll |E|$ orbit constraints \\
    \midrule
    Theorem~\ref{thm:vertex-transitive}
    & Typically need $> 1$ eigenvector to certify
    & Vertex-transitive $\Rightarrow$ 1 eigenvector needed for certificate \\
    \midrule
    Theorem~\ref{thm:abelian-cayley}
    & $\ell(\mathcal{E}_\lambda)$ typically not convex
    & Abelian Cayley $\Rightarrow$ $\ell_\Psi(\mathcal{E}_\lambda)$ is a polyhedral cone \\
    \bottomrule
    \end{tabular}
    \caption{Comparison of unsymmetrized and symmetrized  methods for certifying conformal rigidity for $\lambda= \lambda_2, \lambda_n$.}
    \label{tab:comparison}
    \end{table}

    The payoff is immediate: certifying lower conformal rigidity reduces to the polyhedral feasibility problem of finding $a_i \geq 0$ and $c > 0$ such that
    \begin{align*}
        (a_1^2+ a_3^2) \cdot (6,12) + a_2^2 \cdot (18,0) = c \cdot (12,12).
    \end{align*}
    This system is underconstrained and easy to solve by hand. Setting $c=1$, $a_1^2 = 1$, $a_2^2 = \frac{1}{3}$, $a_3^2 = 0$ gives
    \[ \ell_{\Z_{12}}\paren{\varphi_1^1 + \frac{1}{\sqrt{3}}\varphi_1^4} = \ell_{\Z_{12}}(\varphi_1^1) + \frac{1}{3} \ell_{\Z_{12}}(\varphi_1^4) = (6,12) + \frac{1}{3}(18,0) = (12,12) = \Om,\]
    which gives an orbit-isometric embedding of $G$ on $\E_{\lambda_2}$.
\end{example}

We summarize the full comparison of certifying conformal rigidity in the unsymmetrized and symmetrized cases in Table~\ref{tab:comparison}.

\subsection{Towards Extremal Graph Theory}\label{subsec:extremality} We close this section with a heuristic explanation of why conformally rigid graphs select for graphs with special properties.

\begin{wrapfigure}{R}{0.35\textwidth}
  \centering
  \begin{subfigure}[b]{\linewidth}
    \renewcommand\thesubfigure{(1)}
    \centering
    \begin{tikzpicture}[
        vertex/.style={circle, draw, fill=blue!75, inner sep=1.3pt},
        scale=1.44,
      ]
      \node[vertex] (A0)  at ( 0.6008, -0.1610) {};
      \node[vertex] (A1)  at (-0.6008,  0.1610) {};
      \node[vertex] (A2)  at (-0.4398, -0.4398) {};
      \node[vertex] (A3)  at (-0.1610,  0.6008) {};
      \node[vertex] (A4)  at (-0.6008, -0.1610) {};
      \node[vertex] (A5)  at ( 1.0392, -0.6000) {};
      \node[vertex] (A6)  at ( 0.6000,  1.0392) {};
      \node[vertex] (A7)  at ( 0.0000, -1.2000) {};
      \node[vertex] (A8)  at ( 1.2000,  0.0000) {};
      \node[vertex] (A9)  at (-0.4398,  0.4398) {};
      \node[vertex] (A10) at (-1.0392, -0.6000) {};
      \node[vertex] (A11) at ( 0.1610,  0.6008) {};
      \node[vertex] (A12) at ( 0.6000, -1.0392) {};
      \node[vertex] (A13) at (-0.6000,  1.0392) {};
      \node[vertex] (A14) at (-0.1610, -0.6008) {};
      \node[vertex] (A15) at ( 1.0392,  0.6000) {};
      \node[vertex] (A16) at ( 0.6008,  0.1610) {};
      \node[vertex] (A17) at (-1.0392,  0.6000) {};
      \node[vertex] (A18) at (-0.6000, -1.0392) {};
      \node[vertex] (A19) at ( 0.4398,  0.4398) {};
      \node[vertex] (A20) at ( 0.1610, -0.6008) {};
      \node[vertex] (A21) at ( 0.4398, -0.4398) {};
      \node[vertex] (A22) at (-1.2000,  0.0000) {};
      \node[vertex] (A23) at ( 0.0000,  1.2000) {};
      \begin{pgfonlayer}{background}
      \draw[thick] (A0)--(A1) (A0)--(A2) (A0)--(A3) (A0)--(A4) (A0)--(A5) (A0)--(A6) (A0)--(A7) (A0)--(A8) (A0)--(A9);
      \draw[thick] (A1)--(A16) (A1)--(A17) (A1)--(A18) (A1)--(A19) (A1)--(A20) (A1)--(A21) (A1)--(A22) (A1)--(A23);
      \draw[thick] (A2)--(A3) (A2)--(A10) (A2)--(A11) (A2)--(A12) (A2)--(A16) (A2)--(A17) (A2)--(A18) (A2)--(A19);
      \draw[thick] (A3)--(A13) (A3)--(A14) (A3)--(A15) (A3)--(A20) (A3)--(A21) (A3)--(A22) (A3)--(A23);
      \draw[thick] (A4)--(A7) (A4)--(A10) (A4)--(A11) (A4)--(A13) (A4)--(A16) (A4)--(A19) (A4)--(A21) (A4)--(A22);
      \draw[thick] (A5)--(A8) (A5)--(A10) (A5)--(A12) (A5)--(A14) (A5)--(A17) (A5)--(A19) (A5)--(A21) (A5)--(A23);
      \draw[thick] (A6)--(A9) (A6)--(A11) (A6)--(A12) (A6)--(A15) (A6)--(A18) (A6)--(A19) (A6)--(A22) (A6)--(A23);
      \draw[thick] (A7)--(A12) (A7)--(A14) (A7)--(A15) (A7)--(A17) (A7)--(A18) (A7)--(A20) (A7)--(A23);
      \draw[thick] (A8)--(A11) (A8)--(A13) (A8)--(A15) (A8)--(A16) (A8)--(A18) (A8)--(A20) (A8)--(A22);
      \draw[thick] (A9)--(A10) (A9)--(A13) (A9)--(A14) (A9)--(A16) (A9)--(A17) (A9)--(A20) (A9)--(A21);
      \draw[thick] (A10)--(A15) (A10)--(A18) (A10)--(A20) (A10)--(A22) (A10)--(A23);
      \draw[thick] (A11)--(A14) (A11)--(A17) (A11)--(A20) (A11)--(A21) (A11)--(A23);
      \draw[thick] (A12)--(A13) (A12)--(A16) (A12)--(A20) (A12)--(A21) (A12)--(A22);
      \draw[thick] (A13)--(A17) (A13)--(A18) (A13)--(A19) (A13)--(A23);
      \draw[thick] (A14)--(A16) (A14)--(A18) (A14)--(A19) (A14)--(A22);
      \draw[thick] (A15)--(A16) (A15)--(A17) (A15)--(A19) (A15)--(A21);
      \draw[thick] (A16)--(A23);
      \draw[thick] (A17)--(A22);
      \draw[thick] (A18)--(A21);
      \draw[thick] (A19)--(A20);
      \end{pgfonlayer}
    \end{tikzpicture}
    \subcaption{\textsc{HoG 52605} (lower)}
  \end{subfigure}

  \centering
  \begin{subfigure}[b]{\linewidth}
    \renewcommand\thesubfigure{(2)}
    \centering
    \begin{tikzpicture}[
        vertex/.style={circle, draw, fill=blue!75, inner sep=1.3pt},
        scale=0.40,
      ]
      \node[vertex] (D0) at ( 0.0000, -1.0000) {};
      \node[vertex] (D1) at ( 0.0000, -2.0000) {};
      \node[vertex] (D2) at ( 0.8660, -0.5000) {};
      \node[vertex] (D3) at (-0.8660, -0.5000) {};
      \node[vertex] (D4) at ( 0.0000,  3.0000) {};
      \node[vertex] (D5) at ( 3.4641, -3.0000) {};
      \node[vertex] (D6) at (-3.4641, -3.0000) {};
      \begin{pgfonlayer}{background}
      \draw[thick] (D0)--(D1) (D0)--(D2) (D0)--(D3);
      \draw[thick] (D1) -- (D5) -- (D6) -- (D1);
      \draw[thick] (D2) -- (D4) -- (D5) -- (D2);
      \draw[thick] (D3) -- (D4) -- (D6) -- (D3);
      \end{pgfonlayer}
    \end{tikzpicture}
    \subcaption{\textsc{HoG 1306} (upper)}
  \end{subfigure}
\end{wrapfigure}
\noindent Of the roughly 7000 graphs in the House of Graphs (HoG) database~\cite{hog} that are either lower or upper conformally rigid, many expectedly satisfy nice regularity properties, such as edge- and vertex-transitivity and the weaker notion of vertex- and edge-girth regularity~\cite{goedgebeur,jajcay}. Perhaps more surprisingly, many of these graphs demonstrate some sort of combinatorial extremality: these include extremal Ramsey graphs, $k$-critical graphs, $k$-degenerate graphs, Tur\'an type graphs, and cages. We display six such graphs throughout this section.

The first graph (\textsc{HoG 52605}) is lower conformally rigid and is an extremal construction providing the best known lower bound for the book Ramsey number $R(B_2, B_{10})$~\cite{lidicky}. The second graph (\textsc{HoG 1306}) is the Mycielskian $M(K_3)$, which is 4-critical: its chromatic number is $4$ but removing any edge (or vertex) decreases the chromatic number. This follows from the fact that the Mycielskian of a $k$-critical graph is $(k+1)$-critical~\cite[Exercise 5.2.9]{west}. The third graph (\textsc{HoG 1811}) is both lower and upper conformally rigid; it is a maximal triangle-free graph that occurs as a triangle Ramsey number $R(K_3, G)$ for a graph $G$ of order 9~\cite{brandt} and has the remarkable property that each edge is contained in 33 pairwise distinct shortest cycles~\cite{goedgebeur}. The fourth graph, $K_2 + 10K_1$ (\textsc{HoG 32866}), is upper conformally rigid and is a maximal 2-degenerate graph that is also integral~\cite{bickle,zwierzynski}. The fifth graph is the Tur\'an graph $T(13,3) = K_{4,4,5}$ (\textsc{HoG 51333}), which maximizes the number of edges and triangles in a $K_4$ free graph; it is upper conformally rigid. The last graph (\textsc{HoG 34214}) is both lower and upper conformally rigid. It is a cubic integral graph~\cite{bussemaker}, attains the maximum algebraic connectivity for its diameter~\cite{exoo}, and is a (3; 6, 5)-girth-diameter cage~\cite{cambie}.

We can informally explain why conformal rigidity detects extremal graphs through the heuristic that \textit{asymmetry is inefficient}. This heuristic manifests in the dual notions of \textit{saturation} and \textit{criticality}.

A saturation extremal problem asks for the densest graph that avoids some local forbidden structure. A classic example is Tur\'an's Theorem, which states that the complete $r$-partite graph $T(n,r)$ maximizes the number of edges in a graph with $n$ 
\begin{wrapfigure}{L}{0.35\textwidth}
  \centering
  \begin{subfigure}[b]{\linewidth}
    \renewcommand\thesubfigure{(3)}
    \centering
    \begin{tikzpicture}[
        vertex/.style={circle, draw, fill=blue!75, inner sep=1.3pt},
        scale=1.71,
      ]
      \node[vertex] (C0)  at (-0.4858, -0.3687) {};
      \node[vertex] (C1)  at (-0.5008,  0.3481) {};
      \node[vertex] (C2)  at (-0.2135,  0.9412) {};
      \node[vertex] (C3)  at (-0.8541,  0.5909) {};
      \node[vertex] (C4)  at (-0.9560, -0.1306) {};
      \node[vertex] (C5)  at (-0.3805, -0.8870) {};
      \node[vertex] (C6)  at ( 0.3436, -0.9801) {};
      \node[vertex] (C7)  at ( 0.8502, -0.4562) {};
      \node[vertex] (C8)  at ( 0.2005, -0.5760) {};
      \node[vertex] (C9)  at ( 0.8292,  0.4939) {};
      \node[vertex] (C10) at ( 0.1763,  0.5839) {};
      \node[vertex] (C11) at ( 0.2981,  0.9949) {};
      \node[vertex] (C12) at (-0.4196,  0.8688) {};
      \node[vertex] (C13) at ( 0.6098,  0.0127) {};
      \node[vertex] (C14) at ( 0.6966,  0.6676) {};
      \node[vertex] (C15) at ( 0.7260, -0.6359) {};
      \node[vertex] (C16) at ( 1.0383,  0.0239) {};
      \node[vertex] (C17) at (-0.8259, -0.6297) {};
      \node[vertex] (C18) at (-0.1712, -0.9495) {};
      \node[vertex] (C19) at (-0.9611,  0.0878) {};
      \begin{pgfonlayer}{background}
      \draw[thick] (C0)--(C1) (C0)--(C2) (C0)--(C3) (C0)--(C4) (C0)--(C5) (C0)--(C6) (C0)--(C7) (C0)--(C8);
      \draw[thick] (C1)--(C9) (C1)--(C10) (C1)--(C11) (C1)--(C12) (C1)--(C17) (C1)--(C18) (C1)--(C19);
      \draw[thick] (C2)--(C9) (C2)--(C10) (C2)--(C11) (C2)--(C12) (C2)--(C17) (C2)--(C18) (C2)--(C19);
      \draw[thick] (C3)--(C9) (C3)--(C10) (C3)--(C11) (C3)--(C12) (C3)--(C17) (C3)--(C18) (C3)--(C19);
      \draw[thick] (C4)--(C9) (C4)--(C10) (C4)--(C11) (C4)--(C12) (C4)--(C17) (C4)--(C18) (C4)--(C19);
      \draw[thick] (C5)--(C13) (C5)--(C14) (C5)--(C15) (C5)--(C16) (C5)--(C17) (C5)--(C18) (C5)--(C19);
      \draw[thick] (C6)--(C13) (C6)--(C14) (C6)--(C15) (C6)--(C16) (C6)--(C17) (C6)--(C18) (C6)--(C19);
      \draw[thick] (C7)--(C13) (C7)--(C14) (C7)--(C15) (C7)--(C16) (C7)--(C17) (C7)--(C18) (C7)--(C19);
      \draw[thick] (C8)--(C13) (C8)--(C14) (C8)--(C15) (C8)--(C16) (C8)--(C17) (C8)--(C18) (C8)--(C19);
      \draw[thick] (C9)--(C13) (C9)--(C14) (C9)--(C15) (C9)--(C16);
      \draw[thick] (C10)--(C13) (C10)--(C14) (C10)--(C15) (C10)--(C16);
      \draw[thick] (C11)--(C13) (C11)--(C14) (C11)--(C15) (C11)--(C16);
      \draw[thick] (C12)--(C13) (C12)--(C14) (C12)--(C15) (C12)--(C16);
      \end{pgfonlayer}
    \end{tikzpicture}
    \subcaption{\textsc{HoG 1811} (both)}
  \end{subfigure}
  \centering
  \begin{subfigure}[b]{\linewidth}
    \renewcommand\thesubfigure{(4)}
    \centering
    \begin{tikzpicture}[
        vertex/.style={circle, draw, fill=blue!75, inner sep=1.3pt},
        scale=1.62,
      ]
      \node[vertex] (E0)  at ( 0.8938, -0.4719) {};
      \node[vertex] (E1)  at ( 1.0000,  0.1438) {};
      \node[vertex] (E2)  at ( 0.7245,  0.7034) {};
      \node[vertex] (E3)  at ( 0.1731,  0.9951) {};
      \node[vertex] (E4)  at (-0.4452,  0.9073) {};
      \node[vertex] (E5)  at (-0.8937,  0.4719) {};
      \node[vertex] (E6)  at (-0.9999, -0.1439) {};
      \node[vertex] (E7)  at (-0.7243, -0.7035) {};
      \node[vertex] (E8)  at (-0.1729, -0.9951) {};
      \node[vertex] (E9)  at ( 0.4452, -0.9073) {};
      \node[vertex] (E10) at (-0.0483,  0.0410) {};
      \node[vertex] (E11) at ( 0.0484, -0.0411) {};
      \begin{pgfonlayer}{background}
      \draw[thick] (E0)--(E10) (E0)--(E11);
      \draw[thick] (E1)--(E10) (E1)--(E11);
      \draw[thick] (E2)--(E10) (E2)--(E11);
      \draw[thick] (E3)--(E10) (E3)--(E11);
      \draw[thick] (E4)--(E10) (E4)--(E11);
      \draw[thick] (E5)--(E10) (E5)--(E11);
      \draw[thick] (E6)--(E10) (E6)--(E11);
      \draw[thick] (E7)--(E10) (E7)--(E11);
      \draw[thick] (E8)--(E10) (E8)--(E11);
      \draw[thick] (E9)--(E10) (E9)--(E11);
      \draw[thick] (E10)--(E11);
      \end{pgfonlayer}
    \end{tikzpicture}
    \subcaption{\textsc{HoG 32866} (upper)}
  \end{subfigure}
\end{wrapfigure}

\noindent vertices while containing no $r$-clique $K_r$. In this setting, if a graph $G$ contains no $K_r$ but has two regions of asymmetric density, we can add edges to the less dense region while still avoiding a $K_r$. Zykov symmetrization~\cite{zykov} (see~\cite[Theorem 1.2.4]{zhao} for a modern reference) --- one of the many proofs of Tur\'an's Theorem --- formalizes this by showing that if $G$ has two non-adjacent vertices $u,v$ with $\deg(u) > \deg(v)$, one can replace $v$ with a copy of $u$ (containing all the same neighbors) to increase the number of edges without increasing the clique number. This theme pervades many of the extremal examples in the House of Graphs database, including the six graphs displayed in this section. For example, maximal $k$-degenerate graphs maximize the number of edges while ensuring that every subgraph contains a vertex of degree at most $k$. Similarly, maximal Ramsey graphs give the largest integer $n = R(H, H') - 1$ such that the edges of $K_n$ can be two-colored without a monochromatic $H$ or $H'$; maximality forces every vertex to be at capacity when avoiding $H$ and $H'$.

\begin{wrapfigure}{R}{0.35\textwidth}
  \centering
  \begin{subfigure}[b]{\linewidth}
    \renewcommand\thesubfigure{(5)}
    \centering
    \begin{tikzpicture}[
        vertex/.style={circle, draw, fill=blue!75, inner sep=1.3pt},
        scale=1.98,
      ]
      \node[vertex] (B0)  at ( 0.7699, -0.6594) {};
      \node[vertex] (B1)  at ( 0.6594,  0.7699) {};
      \node[vertex] (B2)  at ( 0.0000,  0.0000) {};
      \node[vertex] (B3)  at (-0.7699,  0.6594) {};
      \node[vertex] (B4)  at (-0.6594, -0.7699) {};
      \node[vertex] (B5)  at ( 0.1505,  0.8170) {};
      \node[vertex] (B6)  at (-0.7825, -0.2765) {};
      \node[vertex] (B7)  at ( 0.2765, -0.7825) {};
      \node[vertex] (B8)  at ( 0.7825,  0.2765) {};
      \node[vertex] (B9)  at (-0.2765,  0.7825) {};
      \node[vertex] (B10) at (-0.8170,  0.1505) {};
      \node[vertex] (B11) at (-0.1505, -0.8170) {};
      \node[vertex] (B12) at ( 0.8170, -0.1505) {};
      \begin{pgfonlayer}{background}
      \draw[thick] (B0)--(B5) (B0)--(B10) (B0)--(B11) (B0)--(B12);
      \draw[thick] (B1)--(B5) (B1)--(B10) (B1)--(B11) (B1)--(B12);
      \draw[thick] (B2)--(B5) (B2)--(B10) (B2)--(B11) (B2)--(B12);
      \draw[thick] (B3)--(B5) (B3)--(B10) (B3)--(B11) (B3)--(B12);
      \draw[thick] (B4)--(B5) (B4)--(B10) (B4)--(B11) (B4)--(B12);
      \draw[thick] (B0)--(B6) (B0)--(B7) (B0)--(B8) (B0)--(B9);
      \draw[thick] (B1)--(B6) (B1)--(B7) (B1)--(B8) (B1)--(B9);
      \draw[thick] (B2)--(B6) (B2)--(B7) (B2)--(B8) (B2)--(B9);
      \draw[thick] (B3)--(B6) (B3)--(B7) (B3)--(B8) (B3)--(B9);
      \draw[thick] (B4)--(B6) (B4)--(B7) (B4)--(B8) (B4)--(B9);
      \draw[thick] (B5)--(B6) (B5)--(B7) (B5)--(B8) (B5)--(B9);
      \draw[thick] (B6)--(B10) (B6)--(B11) (B6)--(B12);
      \draw[thick] (B7)--(B10) (B7)--(B11) (B7)--(B12);
      \draw[thick] (B8)--(B10) (B8)--(B11) (B8)--(B12);
      \draw[thick] (B9)--(B10) (B9)--(B11) (B9)--(B12);
      \end{pgfonlayer}
    \end{tikzpicture}
    \subcaption{\textsc{HoG 51333} (upper)}
  \end{subfigure}
  \centering
  \begin{subfigure}[b]{\linewidth}
    \renewcommand\thesubfigure{(6)}
    \centering
    \begin{tikzpicture}[
        vertex/.style={circle, draw, fill=blue!75, inner sep=1.3pt},
        scale=1.62,
      ]
      \node[vertex] (F0)  at (-0.9744,  0.0000) {};
      \node[vertex] (F1)  at (-0.5847,  0.7795) {};
      \node[vertex] (F2)  at (-0.5847, -0.7795) {};
      \node[vertex] (F3)  at (-0.5847,  0.0000) {};
      \node[vertex] (F4)  at ( 0.5847, -0.7795) {};
      \node[vertex] (F5)  at ( 0.5847,  0.0000) {};
      \node[vertex] (F6)  at ( 0.5847,  0.7795) {};
      \node[vertex] (F7)  at ( 0.9744,  0.0000) {};
      \node[vertex] (F8)  at ( 0.1948, -0.9744) {};
      \node[vertex] (F9)  at ( 0.1948, -0.5847) {};
      \node[vertex] (F10) at ( 0.1948,  0.1948) {};
      \node[vertex] (F11) at ( 0.1948,  0.9744) {};
      \node[vertex] (F12) at ( 0.1948,  0.5847) {};
      \node[vertex] (F13) at ( 0.1948, -0.1948) {};
      \node[vertex] (F14) at (-0.1948, -0.9744) {};
      \node[vertex] (F15) at (-0.1948, -0.1948) {};
      \node[vertex] (F16) at (-0.1948,  0.1948) {};
      \node[vertex] (F17) at (-0.1948,  0.5847) {};
      \node[vertex] (F18) at (-0.1948,  0.9744) {};
      \node[vertex] (F19) at (-0.1948, -0.5847) {};
      \begin{pgfonlayer}{background}
      \draw[thick] (F0)--(F1) (F0)--(F2) (F0)--(F3);
      \draw[thick] (F1)--(F17) (F1)--(F18);
      \draw[thick] (F2)--(F14) (F2)--(F19);
      \draw[thick] (F3)--(F15) (F3)--(F16);
      \draw[thick] (F4)--(F7) (F4)--(F8) (F4)--(F9);
      \draw[thick] (F5)--(F7) (F5)--(F10) (F5)--(F13);
      \draw[thick] (F6)--(F7) (F6)--(F11) (F6)--(F12);
      \draw[thick] (F8)--(F14) (F8)--(F15);
      \draw[thick] (F9)--(F16) (F9)--(F17);
      \draw[thick] (F10)--(F18) (F10)--(F19);
      \draw[thick] (F11)--(F16) (F11)--(F18);
      \draw[thick] (F12)--(F15) (F12)--(F19);
      \draw[thick] (F13)--(F14) (F13)--(F17);
      \end{pgfonlayer}
    \end{tikzpicture}
    \subcaption{\textsc{HoG 34214} (both)}
  \end{subfigure}
\end{wrapfigure}

The dual to saturation is criticality, which is concerned with finding the sparsest graph that maintains a lower bound on some global property. The best example of this is $k$-criticality, which is a property where a graph $G$ has chromatic number $k$, and any proper subgraph of $G$ has chromatic number strictly less than $k$ --- every edge is essential to maintain chromatic number $k$. Similarly, minimal crossing number graphs are the smallest graphs attaining their crossing number. In general, the idea of criticality is that every piece of the graph is load-bearing, and asymmetry implies there are redundancies in some local region.

Lower and upper conformal rigidity can be viewed as spectral versions of saturation and criticality. The problem of maximizing $\lambda_2(w)$ subject to $\sum_{uv \in E} w_{uv} = \abs{E}$ is equivalent to the relaxed version where we require $\sum_{uv \in E} w_{uv} \leq \abs{E}$. This echoes the above description of saturation: maximize connectivity subject to an upper bound on the total weight allowed. This leads to the most uniform weight distribution factoring in the structure of $G$ as hinted at in Example~\ref{ex:barbell}. In the case where $G$ is lower conformally rigid, $w = \one$ is a maximizer, so the most uniform weight distribution is already reflected by the edge structure. Analogously, upper conformal rigidity can be seen as a criticality problem of minimizing $\lambda_n(w)$ subject to $\sum_{uv \in E} w_{uv} \geq \abs{E}$. We remark that the original inspiration for the definition of conformal rigidity came from the criticality problem of graph sparsification~\cite{babecki}. In this setting, the hypercube graph cannot be further sparsified.

\setcounter{secnumdepth}{2}
\section{Subdifferentials and Conformal Rigidity}\label{sec:subdifferential}

\subsection{Convex Optimization Preliminaries}
We can characterize conformal rigidity of a graph $G$ through subdifferential analysis, which is a standard technique in convex optimization. This will give us a geometric understanding of conformal rigidity as well as the tools to certify conformal rigidity of a graph computationally. First, we review basic results from convex optimization, referring to~\cite{hiriartCAMA,hiriart}.

Consider a convex function  $f: \R^n \to \R$. Any point $x \in \R^n$ has a \textit{subgradient} $g \in \R^n$ which satisfies the inequality $f(y) \geq f(x) + \langle g , y -x \rangle$ for all $y \in \R^n$. We can interpret a subgradient geometrically as a vector $g$ such that $(g, -1) \in \R^{n+1}$ is a normal vector to the epigraph of $f$ at $(x,f(x))$. The set of all subgradients of $f$ at $x$ is denoted by $\partial f(x)$ and called the \textit{subdifferential} at $x$.
The minimality of a point $x^*$ for a convex function $f$ can be certified by its subdifferential: consider the constrained optimization problem
\begin{align*}
    \min_{x \geq 0} \ &f(x)\\
    \text{s.t. } &Ax = b,
\end{align*}
where $f: \R^n \to \R$ is convex, $A \in \R^{m \times n}$, and $b \in \R^m$. Assuming there exists a strictly feasible point (i.e. $x \in \R^n$ such that $Ax = b$ and $x > 0$), Slater's condition holds and thus, strong duality holds. This means the Karush-Kuhn-Tucker (KKT) conditions are a necessary and sufficient check for optimality~\cite[Ch. VII, Theorem 2.2.5]{hiriartCAMA}. In our context, $x^*$ is optimal if and only if there exist multipliers $y \in \R^m$, $\lambda \in \R^n$ with $\lambda \geq 0$ such that
\begin{equation}\label{eq:kkt}
    0 \in \partial f(x^*) + A^Ty - \lambda 
\end{equation} 
and $\lambda_i x^*_i = 0$ for all $1 \leq i \leq n$.
We are primarily interested in functions $f: \R^n \to \R$ of the form $f(x) = \sup_{k \in K} f_k(x)$ where $f_k: \R^n \to \R$ is convex for each $k$ in some index set $K$. Any function $f$ of this form is convex, and under certain assumptions of our index set $K$, we can apply Danskin's theorem to obtain a nice description of $\partial f(x)$ in terms of the subdifferentials $\partial f_k (x)$. For a fixed $x \in \R^n$, we denote the \textit{active index set} at $x$ to be $K(x) = \set{k \in K : f_k(x) = f(x)}$. This is the set of indices $k$ where $f_k$ attains the supremum at $x$. 
\begin{theorem}[Danskin]\cite[Theorem D.4.4.2]{hiriart}\label{thm:danskin}
    With the notation above, assume $K$ is a compact metric space, and $k \mapsto f_k(x)$ is continuous for each $x \in \R^n$. Then the active index set $K(x)$ is nonempty, and 
    \[\partial f(x) = \Conv \paren{\bigcup_{k \in K(x)} \partial f_k(x)}. \]
    In other words, the subdifferential $\partial f(x)$ is the convex hull of the subdifferentials $\partial f_k(x)$ for all active indices $k$.
\end{theorem}

\subsection{A Subdifferential Characterization of Conformal Rigidity}\label{subsec:equiv-cr}
We now return our attention to the functions $\lambda_2(w)$ and $\lambda_n(w)$ on a graph $G = (V,E)$ with nonnegative edge weights $w: E \to \R$ normalized to $\sum_{ij \in E} w_{ij} = \abs{E}$. Recall that we can characterize $\lambda_n$ variationally by
\[\lambda_n(w) = \max_{\|\varphi\|=1} \langle \varphi, L(w) \varphi \rangle = \max_{\|\varphi\|=1} \sum_{uv \in E} w_{uv}(\varphi(u) - \varphi(v))^2 = \max_{\|\varphi\|=1} \langle w, \ell(\varphi) \rangle\]
where $\ell(\varphi) = \paren{(\varphi(u) - \varphi(v))^2}_{uv \in E}$ is the edge-energy map given in Definition~\ref{def:ell}.
For a fixed $\varphi$ such that $\| \varphi \| = 1$, the map $w \mapsto \langle w, \ell(\varphi) \rangle$ is linear in $w$ and is thus convex. This means that $\lambda_n(w)$ is the pointwise maximum of convex functions in $w$ and is thus convex. Similarly, 
\[\lambda_2(w) = \min_{\substack{\|\varphi\|=1, \\ \langle \one, \varphi \rangle = 0}} \sum_{uv \in E} w_{uv}(\varphi(u) - \varphi(v))^2 = \min_{\substack{\|\varphi\|=1, \\ \langle \one, \varphi \rangle = 0}} \langle w, \ell(\varphi) \rangle.\]
is a pointwise minimum of linear functions and is thus concave. Equivalently, $-\lambda_2(w)$ is convex.
\begin{lemma}
    A graph $G$ is lower conformally rigid if and only if there exists some $c \in \R$ such that $c \one \in \partial (-\lambda_2)(\one)$. Analogously, $G$ is upper conformally rigid if and only if $c \one \in \partial \lambda_n(\one)$ for some $c \in \R$.
\end{lemma}
\begin{proof}    
    In the context of upper conformal rigidity, we are interested in the linearly-constrained, convex optimization problem
    \begin{equation}\label{eq:lambda_n}
    \begin{split}
        \min_{w\geq 0} \ &\lambda_n(w)\\
        \text{s.t.} \ &\one^T w = \abs{E}.
    \end{split}
    \end{equation}
    Determining whether a graph is upper conformally rigid or not is equivalent to determining whether $w = \one$ is an optimal solution to this problem. Since $\one$ is a strictly feasible solution of the feasible set $\{w \geq 0, \one^\top w = \abs{E}\}$, the KKT conditions~\eqref{eq:kkt} tell us that $\one$ is optimal if and only if there exists some $c \in \R$ such that $0 \in \partial \lambda_n(\one) + c \one$. Equivalently, $\one$ is optimal if and only if there exists some $c \in \R$ such that $c\one \in \partial \lambda_n(\one)$.  For lower conformal rigidity, we can apply the same analysis to the optimization problem
    \begin{equation}\label{eq:lambda_2}
    \begin{aligned}
        \max_{w\geq 0} \ &\lambda_2(w)\\
        \text{s.t.} \ &\one^T w = \abs{E}
    \end{aligned}
    \qquad \text{or, equivalently,} \qquad
    \begin{aligned}
        \min_{w\geq 0} \ &{-\lambda_2(w)}\\
        \text{s.t.} \ &\one^T w = \abs{E}. \qedhere
    \end{aligned}
    \end{equation}
\end{proof}

\begin{lemma}
    The subdifferentials of $-\lambda_2$ and $\lambda_n$ for edge weights $w$ on a graph $G$ are given by 
    \begin{equation}\label{eq:subdifferential}
    \begin{aligned}
    \partial (-\lambda_2)(w) &= \Conv \set{-\ell(\varphi): \varphi \in
        \E_{\lambda_2(w)}, \| \varphi \| = 1}, \\[4pt]
    \partial \lambda_n(w) &= \Conv \set{\ell(\varphi) : \varphi \in
        \E_{\lambda_n(w)}, \| \varphi \| = 1}.
    \end{aligned}
    \end{equation}                                                             
                  
\end{lemma}
\begin{proof}
    Recall that $\lambda_n(w) = \max_{\varphi \in K} \langle w, \ell(\varphi) \rangle$ where $K = \set{\varphi : \| \varphi \| = 1}$ is compact and $\varphi \mapsto \langle w, \ell (\varphi) \rangle$ is continuous for each $w \in W$. Moreover, for each $\varphi$, the map $w \mapsto \langle w, \ell(\varphi) \rangle$ is affine, so its subdifferential just contains its gradient $\ell(\varphi)$. Finally, we observe that for edge weights $w$, the active index set $K(w)$ is precisely the set of unit eigenvectors in $\E_{\lambda_n(w)}$. Indeed,
    \[K(w) = \set{\varphi \in K : \langle w, \ell(\varphi) \rangle = \langle \varphi, L(w) \varphi \rangle = \lambda_n(w)} = \set{\varphi \in \E_{\lambda_n(w)}: \| \varphi \| = 1},\]
    so we can apply Theorem~\ref{thm:danskin} to get
    \[ \partial \lambda_n(w) = \Conv \set{\ell(\varphi) : \varphi \in \E_{\lambda_n(w)}, \| \varphi \| = 1}.\]
    Applying an identical analysis to $-\lambda_2$ and $K = \set{\varphi : \| \varphi \| = 1, \langle \one, \varphi \rangle = 0}$ gives us
    \[\partial (-\lambda_2)(w) = \Conv \set{-\ell(\varphi) : \varphi \in \E_{\lambda_2(w)}, \| \varphi \| = 1}. \qedhere\]
\end{proof}

We now give a characterization of conformal rigidity.

\begin{theorem}\label{thm:cr-subgradient}
    A graph $G$ is lower (upper) conformally rigid if and only if there exist unit-length eigenvectors $\varphi_1, \ldots, \varphi_r \in \E_{\lambda_2}$ ($\E_{\lambda_n}$), a constant $c > 0$  and coefficients $a_i \geq 0, \sum_{i=1}^r a_i = 1$ such that 
    \begin{equation}
      c \one = \sum_{i=1}^r a_i \ell(\varphi_i) = \sum_{i=1}^r a_i \paren{{(\varphi_i(u) - \varphi_i(v))}^2}_{uv \in E}.
      \label{eq:cr-subgradient}  
    \end{equation}
\end{theorem}
\begin{proof}
    We know $G$ is upper conformally rigid if and only if there exists some $c \in \R$ such that $c\one \in \partial \lambda_n(\one)$. Since 
 $\partial \lambda_n(\one) = \Conv \set{ \ell(\varphi)  : \varphi \in \E_{\lambda_n}, \| \varphi \| = 1}$ this means that $G$ is upper conformally rigid if and only if we can find a convex combination 
    $\sum_{i=1}^r a_i \ell(\varphi_i) = c\one$
    for some collection of unit-length eigenvectors $\varphi_1, \ldots \varphi_r \in \E_{\lambda_n}$ and some $c \in \R$. Each $\ell(\varphi_i)$ is nonzero and nonnegative in every coordinate, so we have $c > 0$. We achieve the analogous result for lower conformal rigidity by applying our analysis to the convex function $-\lambda_2(w)$.
\end{proof}

We remark that the statement of this theorem differs slightly from the statement in \S~\ref{sec:main-results}: there the $\varphi_i$ did not have to be unit-length and the coefficients $a_i$ did not need to sum to $1$. We resolve this discrepancy shortly in Lemma~\ref{lemma:cr-subgradient}.

\subsection{Edge-Isometric Spectral Embeddings} The subdifferential characterization of conformal rigidity reflects a special geometric property of a conformally rigid graph. In particular, conformal rigidity corresponds to the existence of an edge-isometric spectral embedding (Definition~\ref{def:edge-iso}) of $G$ on the eigenspace $\E_{\lambda_2}$ or $\E_{\lambda_n}$. For a more detailed account on the connection between spectral embeddings and conformal rigidity, we refer to~\cite[\S~2]{gouveia} and~\cite[\S~4]{steinerberger}. In the latter paper, the authors prove the following result.

\begin{theorem}\cite[Proposition 4.3]{steinerberger}\label{thm:edge-iso}
    A graph $G$ is lower (upper) conformally rigid if and only if $G$ has an edge-isometric embedding $\mathcal{P}$ on $\E_{\lambda_2}$ ($\E_{\lambda_n}$)
\end{theorem}

We now show that the condition given by~\eqref{eq:cr-subgradient} in Theorem~\ref{thm:cr-subgradient} is directly equivalent to the existence of an edge-isometric spectral embedding on the relevant eigenspace. This gives us an alternate proof of the above proposition. First we show that we can somewhat relax the conditions assumed in Theorem~\ref{thm:cr-subgradient}.

\begin{lemma}\label{lemma:cr-subgradient}
    A graph $G$ is lower (upper) conformally rigid if and only if there exist nonzero eigenvectors $\varphi_1, \ldots, \varphi_r \in \E_{\lambda_2}$ ($\E_{\lambda_n}$) such that 
    \begin{equation}
        \one = \sum_{i=1}^r \ell(\varphi_i) = \sum_{i=1}^r \paren{{(\varphi_i(u) - \varphi_i(v))}^2}_{uv \in E}.
        \label{eq:cr-edge-iso}
    \end{equation}
\end{lemma}
\begin{proof}
    Observe that the edge-energy vector corresponding to any $\varphi$ is always nonnegative and is a homogeneous quadratic polynomial in $\varphi$. This means for any $c \in \R, \ \ell(c \varphi) = c^2 \ell(\varphi)$. Assume we have unit-length eigenvectors $\varphi_1,\ldots, \varphi_r$ satisfying \eqref{eq:cr-subgradient}. Then there exists $c > 0$ such that 
    \[\one = \frac{1}{c}\sum_{i=1}^r a_i \ell(\varphi_i) = \sum_{i=1}^r \ell\paren{\sqrt{\frac{a_i}{c}} \cdot \varphi_i}\]
    where $\sqrt{a_1/c} \cdot \varphi_1, \ldots, \sqrt{a_r/c} \cdot \varphi_r$ are nonzero eigenvectors. Now assume we have nonzero eigenvectors $\varphi_1, \ldots, \varphi_r$ that are of arbitrary length satisfying \eqref{eq:cr-edge-iso}. If we let $M = \sum_{i=1}^r \|\varphi_i \|^2 > 0$, we have 
    \[\frac{1}{M} \one = \frac{1}{M} \sum_{i=1}^r \ell(\varphi_i) = \frac{1}{M} \sum_{i=1}^r \|\varphi_i\|^2 \ell\left(\frac{\varphi_i}{\|\varphi_i\|}\right)\]
    where each $\varphi_i/\|\varphi_i\|$ is of unit-length and 
    \[\sum_{i=1}^r \frac{\|\varphi_i \|^2}{M} = \frac{1}{M} \sum_{i=1}^r \|\varphi_i\|^2 = 1.\]
    Thus, we can obtain a convex combination of unit vectors that is constant.
\end{proof}

We can now give an alternate proof that conformal rigidity is equivalent to the existence of edge-isometric spectral embeddings on $\E_{\lambda_2}$ and $\E_{\lambda_n}$.

\begin{proof}[Proof of Theorem~\ref{thm:edge-iso}]
    Let $\mathcal{P}$ be an edge-isometric spectral embedding of $G$ on $\E_{\lambda_2}$ corresponding to the matrix $\begin{pmatrix}
        \varphi_1 & \cdots & \varphi_r
    \end{pmatrix}$
    where $\varphi_1, \ldots, \varphi_r \in \E_{\lambda_2}$. We may assume without loss of generality that each $\varphi_i$ is nonzero. This means there exists a constant $c > 0$ such that 
    \[c = \|p_u - p_v \|^2 = \sum_{i=1}^r {(\varphi_i(u) - \varphi_i(v))}^2 = \sum_{i=1}^r \ell(\varphi_i)_{uv}\]
    for all $uv \in E$. Thus, we have the convex combination
    \[\sum_{i=1}^r \frac{1}{\abs{E}} \ell(\varphi_i) = \frac{c}{\abs{E}} \one,\]
    so by Lemma~\ref{lemma:cr-subgradient}, $G$ is lower conformally rigid.
    Conversely, suppose $G$ is lower conformally rigid. Then by Theorem~\ref{thm:cr-subgradient} there exists unit eigenvectors $\varphi_1, \ldots, \varphi_r$ and $a_1, \ldots, a_r \geq 0$ with $\sum_{i=1}^r a_i = 1$ such that $\sum_{i=1}^r a_i \ell(\varphi_i) = c \one$ for some constant $c > 0$. Now let $\mathcal{P}$ be the spectral embedding corresponding to the matrix
    $\begin{pmatrix}
        \sqrt{a_1} \varphi_1 & \cdots & \sqrt{a_r} \varphi_r
    \end{pmatrix}$.
    Then
    \[\|p_u - p_v\|^2 = \sum_{i=1}^r a_i (\varphi_i(u) - \varphi_i(v))^2 = \sum_{i=1}^r a_i {\ell(\varphi_i)}_{uv} = c > 0\]
    for all $uv \in E$. Hence $\mathcal{P}$ is an edge-isometric spectral embedding of $G$ on $\E_{\lambda_2}$. The same proof works for $\E_{\lambda_n}$ and upper conformal rigidity.
\end{proof}

\begin{remark}
    This proof explains our choice for the notation $\ell(\varphi)$. If $\varphi \in \E_{\lambda}$ is a column of $P$ corresponding to a spectral embedding $\mathcal{P}$, then $\varphi$ contributes $\ell(\varphi)_{uv}$ to the squared length of the edge $uv$ in $\mathcal{P}$.    
\end{remark}

Taking convex combinations in~\eqref{eq:cr-subgradient} of edge-energy vectors $\ell(\varphi)$ is necessary. This is equivalent to saying that certain graphs need more than one eigenvector to certify their conformal rigidity. In the language of spectral embeddings, this corresponds to the minimum number of vectors needed for an edge-isometric embedding. 

\begin{lemma}\label{lemma:bipartite}
    Suppose that $G$ has a non-constant 1-dimensional edge-isometric (not necessarily spectral) embedding $\varphi: V \to \R$. Then $G$ is bipartite.
\end{lemma}
\begin{proof}
    Since $\varphi$ is not constant, we know there exists $c >0$ such that $\abs{\varphi(u)- \varphi(v)} = c$ for all $uv \in E$. Given a cycle $C$ in $G$ and any orientation of edges, we know 
    \[0 = \sum_{uv \in C} \varphi(u) - \varphi(v) = \sum_{\set{i,j} \in C} \pm c.\]
    Since $c > 0$, the number of edges in $C$ must be even. This means $G$ contains no odd cycles and is thus bipartite.
\end{proof}

\begin{corollary}\label{cor:bipartite}
    If $G$ is not bipartite, then no 1-dimensional edge-isometric embedding exists.
\end{corollary}

For example, cycles $C_n$ of odd length $n$ are conformally rigid (because they are edge-transitive~\cite{steinerberger}) and not bipartite. Thus, any edge-isometric spectral embedding of $C_n$ must be of dimension at least 2. We also saw that Example~\ref{ex:running} requires two eigenvectors for both of its edge-isometric spectral embeddings. We close this section with a partial converse of Lemma~\ref{lemma:bipartite}.

\begin{proposition}\label{prop:regular-bipartite}
    If $G$ is regular and bipartite, it is upper conformally rigid.
\end{proposition}
\begin{proof}
    Since $G$ is bipartite, we can assign a vector $\varphi$ that is $1$ on one partition and $-1$ on the other partition. By regularity, this must be an eigenvector of $L$ whose eigenvalue is twice the max degree of $G$. Since $\lambda_n$ is bounded above by twice the max degree of $G$ (by Gershgorin's Theorem~\cite[Theorem 6.1.1]{horn}), $\varphi \in \E_{\lambda_n}$. This embedding is edge-isometric as $\abs{\varphi(u) - \varphi(v)} = \abs{\pm 1 - \mp 1} = 2$ for all $uv \in E$.
\end{proof}

\subsection{SDP Formulation}

Now that we know we can certify conformal rigidity with an edge-isometric spectral embedding on a suitable eigenspace, we want to verify computationally whether such an edge-isometric spectral embedding exists for a specific graph $G$. We can do so through \textit{semidefinite programming}~\cite{bental, bvconvex}, which falls under the umbrella of convex optimization and generalizes linear programming to matrix-valued variables. We will use $\s^n$ to denote the set of $n \times n$ real symmetric matrices and $\s^n_+$ to denote the cone of real symmetric positive semidefinite (psd) matrices. We also write $Z \succeq 0$ when $Z$ is psd.

To phrase the problem of finding an edge-isometric embedding as a semidefinite program (SDP), we need to associate a psd matrix to any spectral embedding.  For an edge $uv \in E$ denote the matrix $L^{uv}$ to be the Laplacian of the subgraph induced by the edge $uv$, so that for $\varphi: V \to \R$, we have $\langle \varphi, L^{uv} \varphi \rangle = \ell(\varphi)_{uv}$. Recall that given two matrices $A, B \in \R^{n \times n}$, their \textit{Frobenius inner product} is given by 
\begin{equation}\label{eq:frobenius-ip}
    \langle A, B \rangle_F := \Tr(A^TB) = \sum_{i,j=1}^n A_{ij} B_{ij}.
\end{equation}

\begin{lemma}\label{lemma:embed-to-psd}
    Let $\varphi_1, \ldots, \varphi_r \in \R^V$ be vectors and $\Phi = \sum_{i=1}^r \varphi_i \varphi_i^T$. Then $\Phi \succeq 0$, and $\langle L^{uv}, \Phi \rangle_F = \sum_{i=1}^r \ell(\varphi_i)_{uv}$ for all $uv \in E$.
\end{lemma}
\begin{proof}
    Given $\varphi_1, \ldots, \varphi_r \in \R^V$ and $uv \in E$, we have
    \begin{align*}
      \langle L^{uv}, \Phi \rangle_F &= \sum_{i=1}^r \langle L^{uv}, \varphi_i \varphi_i^T \rangle_F = \sum_{i=1}^r  \Tr(L^{uv} \varphi_i \varphi_i^T) = \sum_{i=1}^r  \Tr(\varphi_i^T L^{uv} \varphi_i) \\
      &= \sum_{i=1}^r \langle \varphi_i, L^{uv} \varphi_i \rangle = \sum_{i=1}^r \ell(\varphi_i)_{uv}.
    \end{align*}
    Since $\Phi$ is a sum of rank-1 psd matrices $\varphi_i \varphi_i^T$, it is also psd.
\end{proof}

\begin{proposition}
    Let $\lambda > 0$ be an eigenvalue of $G$. Then $G$ has an edge-isometric embedding on $\E_{\lambda}$ if and only if we can solve the SDP feasibility problem:
    \begin{equation}\label{eq:cr-sdp}
    \begin{split}
        \text{find} \ &\Phi \succeq 0,\\
        \text{s.t.} \ &L\Phi = \lambda \Phi, \\
        & \langle L^{uv}, \Phi \rangle_F = 1 \text{ for all } uv \in E.
    \end{split}
\end{equation}
\end{proposition}
\begin{proof}
By Lemma~\ref{lemma:cr-subgradient}, if $G$ has an edge-isometric embedding on $\E_\lambda$, we can find eigenvectors $\varphi_1, \ldots, \varphi_r$ such that~\eqref{eq:cr-edge-iso} holds. By Lemma~\ref{lemma:embed-to-psd}, the matrix $\Phi = \sum_{i=1}^r \varphi_i \varphi_i^T \succeq 0$ and satisfies $\langle L^{uv}, \Phi \rangle_F = \sum_{i=1}^r \ell(\varphi_i)_{uv} = 1$ for all $uv \in E$, and
    \[L\Phi = \sum_{i=1}^r L\varphi_i \varphi_i^T = \sum_{i=1}^r \lambda \varphi_i \varphi_i^T = \lambda \Phi.\]

    Conversely, suppose we have $\Phi$ satisfying~\eqref{eq:cr-sdp}. By the spectral theorem for symmetric matrices, since $\Phi \succeq 0$, we can write $\Phi = \sum_{i=1}^r \mu_i \psi_i \psi_i^T$ where $\psi_i$ is an eigenvector of $\Phi$ with eigenvalue $\mu_i \geq 0$. We can assume $r$ is the rank of $\Phi$, so that all the $\mu_i > 0$. Since $L\Phi = \lambda \Phi$, the columns of $\Phi$ are all in $\E_\lambda$, and thus, $\psi_i = \frac{1}{\mu_i} \Phi \psi_i \in \E_\lambda$. Taking $\varphi_i = \sqrt{\mu_i} \psi_i \in \E_\lambda$, we get $\Phi = \sum_{i=1}^r \varphi_i \varphi_i^T$, and the corresponding spectral embedding $\begin{pmatrix}
        \varphi_1 & \ldots & \varphi_r
    \end{pmatrix}$ on $\E_\lambda$ makes each edge of $G$ length one.
\end{proof}

The above SDP feasibility problem is equivalent to finding an edge-isometric embedding on $\E_\lambda$ where the rank of $\Phi$ corresponds to the dimension of the spectral embedding. This is equivalent to the complementary slackness condition in Proposition 3.3 of~\cite{steinerberger} derived from their SDP formulations of conformal rigidity. In practice, we can parametrize~\eqref{eq:cr-sdp} using an orthonormal basis $B := \begin{pmatrix}
    \varphi_1 & \ldots & \varphi_d
\end{pmatrix}$ of $\E_\lambda$~\cite[Corollary 6.2]{steinerberger}. 

\begin{corollary}
    Let $\lambda > 0$ be a Laplacian eigenvalue of $G$, and $B := \begin{pmatrix} \varphi_1 & \ldots & \varphi_d \end{pmatrix}$ an orthonormal basis of $\E_\lambda$. If we denote $\tilde{L}^{uv} = B^TL^{uv}B$, then~\eqref{eq:cr-sdp} reduces to 
    \begin{equation}\label{eq:cr-sdp-param}
        \begin{split}
            \text{find} \ &Z \succeq 0,\\
            \text{s.t.} \ & \langle \tilde{L}^{uv}, Z \rangle_F = 1 \text{ for all } uv \in E.
        \end{split}
    \end{equation}
\end{corollary}
Under this parametrization, any solution $Z$ for~\eqref{eq:cr-sdp-param} corresponds to the solution $\Phi = BZB^T$ for~\eqref{eq:cr-sdp}. The advantage of this parametrization is that we reduce the matrix size from $n \times n$ to $d \times d$, and we do not have to check the constraint $L\Phi = \lambda\Phi$. This gives us an SDP problem of matrix size $d = \dim \E_\lambda$ with $m = \abs{E}$ constraints. Modern SDP solvers can solve these problems on the order of $O(m^3 + m^2d^2 + md^3)$~\cite[Table 1.1]{kathuria}. Since $m$ grows quadratically with the number of vertices $n = \abs{V}$, this approach quickly becomes expensive for larger graphs. 

Even when an SDP can be successfully solved, the solution returned by a solver is numerical and does not give a formal proof of conformal rigidity. The edge constraints in~\eqref{eq:cr-sdp} and~\eqref{eq:cr-sdp-param} are typically only satisfied up to some tolerance, and the equality $L \Phi = \lambda \Phi$ depends on the approximate orthonormal basis $B$.

\begin{rexample2}[continued]
    We continue the example of $G = \Cay(\Z_{12}, \set{2,3})$. This graph is lower conformally rigid and has $\lambda_2=3$ with multiplicity 6. Solving~\eqref{eq:cr-sdp-param} for $G$ using the SDP solver MOSEK~\cite{mosek}, we obtain 
    \begin{equation}\label{eq:sdp-certificate}
        Z \approx \begin{pmatrix}
            1.3499 & 0.1321 & -0.0216 & -0.0523 & -0.1784 & 0.0078 \\
            0.1321 & 1.3066 & 0.1622 & 0.0558 & 0.0763 & -0.0813 \\
            -0.0216 & 0.1622 & 1.2313 & -0.0257 & 0.1239 & 0.1392 \\
            -0.0523 & 0.0558 & -0.0257 & 1.4805 & -0.0519 & 0.0122 \\
            -0.1784 & 0.0763 & 0.1239 & -0.0519 & 1.2038 & -0.0684 \\
            0.0078 & -0.0813 & 0.1392 & 0.0122 & -0.0684 & 1.4278
        \end{pmatrix}.
    \end{equation}
    Each edge constraint is satisfied within $3 \times 10^{-13}$, so this computation shows $G$ is very likely lower conformally rigid, but does not rigorously prove so. We can decrease the solver tolerance to improve the residuals to be within $2 \times 10^{-15}$, but even if we were able to get 0 residuals, $Z$ still depends on the numeric parametrization of eigenvectors in $B$. For implementation details, see our accompanying code~\cite{me}.
\end{rexample2}

If $G$ is not conformally rigid, there are no known bounds on how close some $Z \succeq 0$ can be to satisfying the edge-isometry constraints. In the needle in a haystack analogy, our SDP solution cannot rigorously distinguish between a needle and an arbitrarily needle-like piece of hay. Although no such false positive has been observed in practice, there is no theoretical obstruction to a non-conformally rigid graph fooling the SDP solver. In~\cite[\S 6]{steinerberger}, the authors present several heuristic techniques to upgrade an approximate certificate of conformal rigidity to an exact certificate. These approaches, however, are somewhat ad-hoc and typically require detailed manual inspection for each graph. In  \S~\ref{sec:sym}, we address how to speed up the runtime of the SDP~\eqref{eq:cr-sdp-param} in the presence of symmetry and how to obtain exact certificates algebraically.
\section{A Perturbation Theory Perspective}\label{sec:perturbation-theory}
The goal of this section is to give a new proof of Theorem~\ref{thm:cr-subgradient} from the viewpoint of perturbation theory. The subdifferential language in the previous section implicitly casts conformal rigidity as a first-order optimality condition on $\lambda_2(w)$ and $\lambda_n(w)$; we wish to make this perspective explicit through eigenvalue perturbation theory. This generalizes the hyperplane discussion given in~\cite[\S 5.3]{gouveia} from Cayley graphs to all graphs. Throughout this section, we work with $\lambda_2$ and lower conformal rigidity, but the same results hold for $\lambda_n$ and upper conformal rigidity. First we show why equation~\eqref{eq:cr-subgradient} implies lower conformal rigidity. 

\begin{definition}
    Let $w= \one  + y$ where $\sum_{uv \in E} y_{uv} = 0$ be a perturbation from uniform edge-weights. For a vector $\varphi : V \to \R$, we denote its change in energy by
    \[\Delta L^y (\varphi) = \langle \varphi, (L(w) - L) \varphi \rangle = \langle y, \ell(\varphi) \rangle.\]
\end{definition}

The above equation holds because the weighted Dirichlet energy $\langle \varphi, L(w) \varphi \rangle$ of an eigenvector $\varphi$ is linear with respect to $w$, which we showed in \S~\ref{subsec:equiv-cr}. We can interpret $\Delta L^y(\varphi)$ as the directional derivative at $w=\one$ in the direction $y$ for the Dirichlet energy of $\varphi$ with respect to $L(w)$.

\begin{lemma}\label{lemma:energy-change}
    Suppose there exist unit eigenvectors $\varphi_1, \ldots, \varphi_r \in \E_{\lambda_2}$ of $G$ and  $a_i \geq 0, \sum_{i=1}^r a_i = 1$ such that for any perturbation $w = \one + y$ where $\sum_{uv} y_{uv} = 0$,
    \begin{equation}\label{eq:equilibrium}
        \sum_{i=1}^r a_i \Delta L^y(\varphi_i) = 0.
    \end{equation}
    Then $G$ is lower conformally rigid.
\end{lemma}
\begin{proof}
    For any $w = \one + y$, 
    \begin{align*}
        \lambda_2(w) &\leq \sum_{i=1}^r a_i \langle \varphi_i, L(w) \varphi_i \rangle = \sum_{i=1}^r a_i \paren{\langle \varphi_i, L \varphi_i \rangle + \Delta L^y(\varphi_i)} \\
        &= \sum_{i=1}^r a_i\paren{\lambda_2 + \Delta L^y(\varphi_i) } = \lambda_2 + 0 = \lambda_2. \qedhere
    \end{align*} 
\end{proof}
The equation~\eqref{eq:equilibrium} is one of equilibrium --- there is no perturbation of weights $w = \one + y$ that simultaneously increases the Dirichlet energy of \textit{every} vector in $\E_{\lambda_2}$. Since $\lambda_2(w)$ is the minimum Dirichlet energy over unit vectors orthogonal to $\one$, any perturbation that decreases (or maintains) the Dirichlet energy of even a single eigenvector in $\E_{\lambda_2}$ forces $\lambda_2(w) \leq \lambda_2$; the equilibrium condition ensures every perturbation does exactly this to at least one eigenvector. We now show that the subdifferential condition~\eqref{eq:cr-subgradient} implies~\eqref{eq:equilibrium}.
\begin{proof}[Alternate proof of ($\impliedby$) for Theorem~\ref{thm:cr-subgradient}]
    Suppose we have eigenvectors $\varphi_1, \ldots, \varphi_r$, convex coefficients $a_i$ and a constant $c > 0$ satisfying~\eqref{eq:cr-subgradient}. Then for any perturbation $w = \one + y$ where $\sum_{uv \in E} y_{uv} = \langle y, \one \rangle =0$, we have
    \begin{equation*}
    \sum_{i=1}^r a_i \Delta L^y(\varphi_i) = \sum_{i=1}^r a_i \langle y, \ell(\varphi_i) \rangle = \left\langle y, \sum_{i=1}^r a_i\ell(\varphi_i) \right\rangle = \langle y, c \one \rangle = 0.
    \end{equation*}
    Thus, by Lemma~\ref{lemma:energy-change}, $G$ is lower conformally rigid.
\end{proof}

This proof is essentially showing that $G$ has no feasible direction of 
subgradient descent for $-\lambda_2$ at $w = \one$: by~\eqref{eq:subdifferential}, 
the subgradients of $-\lambda_2$ at $\one$ are $\{-\ell(\varphi) : \varphi \in 
\E_{\lambda_2},\, \|\varphi\| = 1\}$, and the condition $\sum_{i=1}^r a_i 
\ell(\varphi_i) = c\one$ is precisely the Lagrange multiplier condition for 
the constraint $\sum_{uv \in E} w_{uv} = |E|$. When $\lambda_2$ 
is simple, this is especially transparent: $\lambda_2(w)$ is differentiable 
at $w = \one$ with gradient $\nabla \lambda_2(\one) = \ell(\varphi)$, so the 
condition reduces to $\ell(\varphi) = c\one$, which is the classical 
Lagrange multiplier condition. The multiple eigenvalue case is the 
nonsmooth generalization, where the unique gradient is replaced by a convex 
combination of subgradients. 

We now show that when this condition fails, one can explicitly construct a perturbation that increases $\lambda_2(w)$, proving that $G$ is not lower conformally rigid. It suffices to show that if~\eqref{eq:cr-subgradient} does not hold, then we can find nonconstant weights $w$ where $\langle \varphi, L(w) \varphi \rangle > \lambda_2$ for all $\|\varphi \| = 1$ such that $\langle \one, \varphi \rangle = 0$. We first show that we can do this for all unit-length $\varphi \in \E_{\lambda_2}$.

\begin{proposition}\label{prop:hyperplane}
    Suppose that~\eqref{eq:cr-subgradient} does not hold for a graph $G$, that is, the line $c\one$ does not intersect $\partial(-\lambda_2)(\one)$. Then there exists a perturbation $y$ with $\sum_{uv \in E} y_{uv} = 0$ such that $\langle y, \ell(\varphi) \rangle > 0$ for all unit-length $\varphi \in \E_{\lambda_2}$. 
\end{proposition}
\begin{proof}
    Suppose that $\partial(-\lambda_2)(\one)$ does not intersect the subspace $c \one$. Since $\partial(-\lambda_2)(\one)$ is a compact, convex set and $c \one$ is closed and convex, we can find a separating hyperplane $y$ such that $\langle y, \ell(\varphi) \rangle > \langle y, c \one \rangle = c \langle y, \one \rangle$ for all unit-length $\varphi \in \E_{\lambda_2}$ and $c > 0$. This means the separating hyperplane must be parallel to the subspace $c \one$ i.e. $\sum_{uv \in E} y_{uv} = \langle y , \one \rangle = 0$. In other words, $w = \one + y$ is a valid perturbation such that the energy change $\Delta L^y (\varphi)$ is strictly positive for all $\varphi \in \E_{\lambda_2}$ with $\| \varphi \| = 1$. 
\end{proof}

We have shown that when~\eqref{eq:cr-subgradient} cannot be satisfied for $G$, we can simultaneously increase the energy for all eigenvectors in $\E_{\lambda_2}$. However, other vectors not in $\E_{\lambda_2}$ might have their energy decrease, so we have to verify that we can choose weights where such that $\langle \varphi, L(w) \varphi \rangle > \lambda_2$ still holds for $\varphi \notin \E_{\lambda_2}$. Informally, we will argue that any unit vector $\varphi$ ``mostly'' in $\E_{\lambda_2}$ will have its energy increase under a perturbation given in Proposition~\ref{prop:hyperplane}. For any unit-length $\varphi$ with ``large'' component outside of $\E_{\lambda_2}$, we know its initial Dirichlet energy $\langle \varphi, L \varphi \rangle$ is ``a lot larger'' than $\lambda_2$. Thus, we can choose $y$ small enough to make sure the energy of $\varphi$ stays above $\lambda_2$.

Denote $S = \set{\varphi : \| \varphi \| = 1, \ \langle \one, \varphi \rangle = 0}$. For any eigenvalue $\lambda > 0$, let $S_\lambda = \set{\varphi \in \E_\lambda : \|\varphi \| = 1}$ and $S_\lambda^\bot = \set{\varphi \in \E_\lambda^\bot \cap S}$. Then any $\varphi \in S$ can be written uniquely as the sum $\varphi = a \psi + b \eta$ with $\psi \in S_{\lambda_2}, \eta \in S_{\lambda_2}^\bot$ and $a^2 + b^2 = 1$.
\begin{lemma}\label{lemma:nbd}
    Suppose there exists some $y \in \R^\abs{E}$ and $M > 0$ such that $\Delta L^y (\varphi) = \langle y, \ell(\varphi) \rangle \geq M$ for all $\varphi \in S_{\lambda_2}$. Then there exists some $\varepsilon > 0$ such that for all $\varphi = a \psi + b \eta \in S$ with $b^2 \leq \varepsilon$,  $\Delta L^y (\varphi) = \langle y, \ell(\varphi) \rangle \geq \frac{M}{2}$.
\end{lemma}
\begin{proof}
    The map $\varphi \mapsto \langle y, \ell(\varphi) \rangle$ on $S$ is continuous, so the image of the compact set $S_{\lambda_2}$ is a compact set contained in the interval $[M, \infty)$. Thus there exists an open interval $U' \subseteq \R$ containing $\langle y, \ell(S_{\lambda_2}) \rangle$ where $x \geq \frac{M}{2} $ for all $ x \in U'$. Then the preimage $U$ of $U'$ gives us an open neighborhood of $S_{\lambda_2}$ in $S$ where $\langle \varphi, L \varphi \rangle \geq \frac{M}{2}$ for all $\varphi \in U$. Finally, any open neighborhood of $S_{\lambda_2}$ in $S$ contains another neighborhood of $S_{\lambda_2}$ of the form $\set{\varphi = a \psi + b \eta : b^2 \leq \varepsilon}$ for some $ \varepsilon > 0$.
\end{proof}
The above lemma gives us a precise form of the statement that if a perturbation $y$ increases the energy on $\E_{\lambda_2}$, it also increases the energy near $\E_{\lambda_2}$. Now we formalize the second part of our argument.

\begin{lemma}\label{lemma:spectral-gap}
    Let $\lambda' > \lambda_2$ be the next largest distinct Laplacian eigenvalue of $G$. For $\varphi = a \psi + b \eta \in S$ with $\psi \in S_{\lambda_2}$ and $\eta \in S_{\lambda_2}^\bot$, we have $\langle \varphi, L \varphi \rangle - \lambda_2 \geq b^2(\lambda' - \lambda_2)$.
\end{lemma}
\begin{proof}
    Since $\psi$ and $\eta$ are eigenvectors of $L$ with $\langle \psi, \eta \rangle = 0$, we get that 
    \begin{align*}
        &\langle \varphi, L \varphi \rangle = a^2 \langle \psi, L \psi \rangle + b^2 \langle \eta, L \eta \rangle \geq a^2 \lambda_2 + b^2 \lambda' \\[4pt]
        \implies &\langle \varphi, L \varphi \rangle - \lambda_2 \geq a^2 \lambda_2 + b^2 \lambda' - \lambda_2 = b^2(\lambda' - \lambda_2). \qedhere
    \end{align*}
\end{proof}

We can now prove the forward direction of Theorem~\ref{thm:cr-subgradient}.

\begin{proof}[Alternate proof of ($\implies$) for Theorem~\ref{thm:cr-subgradient}]
    Let $M$ be as in Lemma~\ref{lemma:nbd} and choose $\varepsilon > 0$ so that $\langle y, \ell(\varphi) \rangle \geq \frac{M}{2}$ for all $\varphi = a \psi + b \eta$ where $b^2 \leq \varepsilon$. We know that $\langle w, \ell(\varphi) \rangle$ is continuous as a function of $w$ and $\varphi$ on the compact domain $\set{w \geq 0 : \langle w , \one \rangle = |E|} \times S$. Thus, $\langle w, \ell(\varphi) \rangle$ is uniformly continuous, so we can choose some $\delta > 0$ such that for all $\varphi \in S$, 
    \begin{equation}
        \delta \abs{\Delta L^y(\varphi)} = \delta \abs{ \langle y, \ell(\varphi) \rangle} = \abs{\langle \varphi, L \varphi \rangle - \langle \varphi, L(\one + \delta y) \varphi \rangle} < \varepsilon(\lambda' - \lambda_2).
        \label{eq:bdd-energy}
    \end{equation}
    Now we will show that $\lambda_2(w) > \lambda_2$ for $w = \one + \delta y$. If $b^2 \leq \varepsilon$, then by Lemma~\ref{lemma:nbd},
    \[ \langle \varphi, L(\one + \delta y) \varphi \rangle = \langle \varphi, L \varphi \rangle + \delta \Delta L^y(\varphi) \geq \lambda_2 + \delta \cdot \frac{M}{2} > \lambda_2.\]
    If $b^2 > \varepsilon$ and $\lambda' > \lambda_2$ is the next largest Laplacian eigenvalue, then 
    \begin{align*}
        \langle \varphi, L(\one + \delta y)\varphi \rangle - \lambda_2 \geq (\langle \varphi, L \varphi \rangle - \lambda_2) - \delta \abs{\Delta L^y(\varphi)} > (b^2 - \varepsilon)(\lambda' - \lambda_2) > 0.
    \end{align*}
    where the second-to-last inequality holds from Lemma~\ref{lemma:spectral-gap} and our choice of $\delta$ in \eqref{eq:bdd-energy}. We have shown $\langle \varphi, L(w) \varphi \rangle > \lambda_2$ for all $\varphi \in S$, so $\lambda_2(w) > \lambda_2$. Hence, $G$ is not lower conformally rigid.
\end{proof}
\section{Symmetry Reduction}\label{sec:sym}
\subsection{Symmetry Reduction}
If $G$ has a large group of automorphisms, we can drastically reduce the number of constraints to check whether $G$ is conformally rigid or not. In this section, we let $\Psi \leq \Aut(G)$ be a subgroup of automorphisms of $G$. By definition of a graph automorphism, every element $\sigma \in \Psi$ acts on $V$ and $E$ by permutation. These actions lift to $\R^{V}$ and $\R^{E}$ by permuting coordinates. For $\varphi: V \to \R$ and $\sigma \in \Psi$, we have $(\sigma \cdot \varphi)(u) = \varphi(\sigma^{-1}(u))$ and similarly, for $y: E \to \R$, we have 
$(\sigma \cdot y)_{uv} = y_{\sigma^{-1}(u)\sigma^{-1}(v)}$.
We observe that the map $\ell: \R^V \to \R^E$ given in \eqref{eq:ell} is $\Psi$-equivariant i.e. for all $\varphi: V \to \R$ and $\sigma \in \Psi$,
\begin{equation}
    \ell( \sigma \cdot \varphi) = \sigma \cdot \ell(\varphi)
    \label{eq:equivariance}
\end{equation}
Indeed, we can check that for any $\sigma \in \Psi$ and $uv \in E$,
\begin{align*}
    \ell(\sigma \cdot \varphi)_{uv} &= \paren{\sigma \cdot \varphi(u)- \sigma \cdot \varphi(v)}^2 = \paren{\varphi( \sigma^{-1} (u)) - \varphi(\sigma^{-1}(v))}^2 \\
    &= \ell(\varphi)_{\sigma^{-1}(u)\sigma^{-1}(v)} = \paren{\sigma \cdot \ell(\varphi)}_{uv}.
\end{align*}
We also observe that $\Psi$ acts on $\R^{E}$ by permutation matrices. Since permutation matrices are orthogonal, we know that for any $y, y': E \to \R$ and $\sigma \in \Psi$,
\begin{equation}
    \langle \sigma \cdot y, y' \rangle = \langle y, \sigma^{-1} \cdot y' \rangle
    \label{eq:ortho-action}
\end{equation}
With these observations, we can symmetry reduce our conformal rigidity check.
\begin{definition}
    We call edge weights $w: E \to \R, \ w \geq 0$, \textit{symmetric} (with respect to $\Psi$) if $\sigma \cdot w = w$ for all $\sigma \in \Psi$.
\end{definition}
Given any normalized edge weights $w \geq 0$ with $\sum_{uv \in E} w_{uv} = \abs{E}$, we can symmetrize $w$ to get a symmetric weighting $w'$ given by 
\[w' = \frac{1}{\abs{\Psi}} \sum_{\tau \in \Psi} \tau \cdot w.\]
Clearly $w' \geq 0$ and 
\[ \sum_{uv \in E} w'_{uv} = \frac{1}{\abs{\Psi}} \sum_{\tau \in \Psi} \sum_{uv \in E}(\tau \cdot w)_{uv} = \frac{1}{\abs{\Psi}}\sum_{\tau \in \Psi} \abs{E} = \abs{E}.\]
Moreover, for any $\sigma \in \Psi$,
\[ \sigma \cdot w' = \frac{1}{\abs{\Psi}} \sum_{\tau \in \Psi} (\sigma \circ \tau) \cdot w = \frac{1}{\abs{\Psi}} \sum_{\tau \in \Psi} \tau \cdot w = w'.\]
We call $w'$ the \textit{symmetrization} of $w$.
\begin{proposition}
    For any edge weights $w: E \to \R$, its symmetrization $w'$ satisfies $\lambda_2(w') \geq \lambda_2(w)$ and $\lambda_n(w') \leq \lambda_n(w)$.
\end{proposition}
\begin{proof}
    Let $w' = \frac{1}{\abs{\Psi}} \sum_{\tau \in \Psi} \tau \cdot w$ and $\varphi$ with $\|\varphi \| = 1$ be orthogonal to the constant vector $\one$. By the variational characterization of $\lambda_2(w)$, we have 
    \[ \langle w, \ell(\varphi) \rangle = \langle \varphi, L(w) \varphi \rangle \geq \lambda_2(w).\]
    Then for any $\sigma \in \Psi$, we know $\sigma \cdot \varphi$ is still unit-length and orthogonal to $\one$, so 
    \[ \langle w, \sigma \cdot \ell (\varphi) \rangle = \langle w, \ell(\sigma \cdot \varphi) \rangle = \langle \sigma \cdot \varphi, L(w) (\sigma \cdot \varphi) \rangle \geq \lambda_2(w),\]
    where the first equality holds from the equivariance of $\ell$ in \eqref{eq:equivariance}. Thus,
    \begin{align*}
        \lambda_2(w) &\leq \frac{1}{\abs{\Psi}} \sum_{\sigma \in \Psi} \langle w, \sigma \cdot \ell(\varphi) \rangle = \frac{1}{\abs{\Psi}} \sum_{\sigma \in \Psi} \langle \sigma^{-1} \cdot w, \ell(\varphi) \rangle \\
        &= \left\langle \left(\frac{1}{\abs{\Psi}} \sum_{\sigma \in \Psi} \sigma^{-1} \cdot w \right), \ell(\varphi) \right\rangle = \langle w', \ell(\varphi) \rangle = \langle \varphi, L(w') \varphi \rangle,
    \end{align*}
    where the first equality holds from \eqref{eq:ortho-action}. Since $\varphi$ was arbitrary, we know $\lambda_2(w') \geq \lambda_2(w).$
    The same argument with inequalities flipped shows $\lambda_n(w') \leq \lambda_n(w)$.
\end{proof}

This proof is similar in spirit to the classical result in mathematical physics that symmetrizing a domain can only decrease its principal Dirichlet eigenvalue~\cite[\S 5.3]{baernstein}. In our case, symmetrizing the edge weights $w$ can only increase $\lambda_2(w)$ and decrease $\lambda_n(w)$. A similar argument was used to prove Theorem 2.3 in~\cite{steinerberger}. With this proposition, we can reduce the search space of the problems in \eqref{eq:lambda_n} and \eqref{eq:lambda_2}
from all normalized edge weights $w: E \to \R$ to only those that are symmetric. Let $\Om_1, \ldots, \Om_s$ be the edge orbits of $G$ under $\Psi$, and let $L^i$ be the graph Laplacian of the subgraph induced by $\Om_i$. Explicitly, for $\varphi : V \to \R$,
\begin{equation}\label{eq:Li}
    \langle \varphi,  L^i \varphi \rangle = \sum_{uv \in \Om_i} \ell(\varphi)_{uv}
\end{equation}
measures the total squared edge length of $\varphi$ on the orbit $\Om_i$. This yields the orbit-energy vector $\ell_\Psi(\varphi) = \paren{\langle \varphi, L^i \varphi \rangle}_{i=1}^s = \paren{\sum_{uv \in \Om_i} (\varphi(u) - \varphi(v))^2}_{i=1}^s$ given in Definition~\ref{def:orbit-energy}. Since symmetric edge weights are constant on edge orbits, we denote $w \in \R^s$ when $w$ is symmetric. Then for any $w \in \R^s$ and any $\varphi$, we have 
\begin{align*}
    \langle \varphi, L(w) \varphi \rangle = \sum_{i=1}^s w_i \sum_{uv \in \Om_i} \ell(\varphi)_{uv} = \sum_{i=1}^s w_i \ell_\Psi(\varphi)_i = \langle w, \ell_\Psi (\varphi) \rangle.
\end{align*}
We use the notation
\[\lambda_2^\Psi(w) = \min_{\|\varphi\| = 1} \langle w, \ell_\Psi(\varphi) \rangle, \text{ and } \lambda_n^\Psi(w) = \max_{\|\varphi\| = 1} \langle w, \ell_\Psi(\varphi) \rangle\]
to indicate that we are considering Laplacian eigenvalues of symmetric edge weights $w$. As in the unsymmetrized case, $\lambda_2^\Psi(w)$ and $\lambda_n^\Psi(w)$ are the pointwise optimum of linear functions and are thus concave/convex.
We care about optimizing these functions over all normalized symmetric weights $w$. For a symmetric weighting $w \geq 0$ to be normalized, we need $\sum_{i=1}^s w_i \abs{\Om_i} = \abs{E}$ since $w$ assigns weight $w_i$ to $\abs{\Om_i}$ edges. Equivalently, if  $\Om := (\abs{\Om_1}, \ldots, \abs{\Om_s})$, we have the problems
\begin{equation} \label{eq:sym-eigenvalues}
    \max_w \lambda_2^\Psi(w), \ \min_w \lambda_n^\Psi(w)
\end{equation}
where $w \geq 0$ is symmetric with $ \langle w, \Om \rangle = \abs{E}$. For conformal rigidity, it now suffices to check that uniform edge weights optimize the problems in~\eqref{eq:sym-eigenvalues}, which is strictly easier than checking if they optimize~\eqref{eq:lambda_n} or~\eqref{eq:lambda_2}. The uniform weights $\one \in \R^\abs{E}$ are symmetric and correspond to $w = \Om$ as symmetrized weights in $\R^s$. Thus $G$ is lower (upper) conformally rigid if and only if $\lambda_2^\Psi(\Om)$ ($\lambda_n^\Psi(\Om)$) are maximal (minimal) amongst all symmetric weights. Again, by the KKT conditions, $\lambda_2^\Psi(\Om)$ is maximal if and only if there exists some $c \in \R$ such that $c \Om \in \partial(-\lambda_2^\Psi)(\Om)$. Likewise, $\lambda_n(\Om)$ is minimal if and only if there exists some $c \in \R$ such that $c \Om \in \partial\lambda_n^\Psi(\Om)$. Applying Theorem~\ref{thm:danskin} to the symmetry reduced problems in~\eqref{eq:sym-eigenvalues}, we know 
\begin{align*}
    \partial(-\lambda_2^\Psi)(w) &= \Conv\set{- \ell_\Psi(\varphi) : \varphi \in \E_{\lambda_2(w)}, \| \varphi \| = 1}, \text{ and} \\
    \partial \lambda_n^\Psi(w) &= \Conv\set{\ell_\Psi(\varphi) : \varphi \in \E_{\lambda_n(w)}, \| \varphi \| = 1}.
\end{align*}
Adapting the proof from Theorem~\ref{thm:cr-subgradient}, we get the following result.
\begin{theorem}\label{thm:sym-subgradient}
    A graph $G$ with a group of automorphisms $\Psi \leq \Aut(G)$ is lower (upper) conformally rigid if and only if there exist unit-length eigenvectors $\varphi_1, \ldots, \varphi_r \in \E_{\lambda_2}$ ($\E_{\lambda_n}$) and coefficients $a_i \geq 0, \sum_{i=1}^r a_i = 1$ such that 
    \begin{equation}
      c \Om = \sum_{i=1}^r a_i \ell_\Psi(\varphi_i)
    \end{equation}
    for some constant $c > 0$.
\end{theorem}

\subsection{Orbit-Isometric Embeddings} We would now like to describe how symmetry reduction works in terms of spectral embeddings. First we give the symmetry reduced analogues to the results in \S~\ref{sec:subdifferential}. Recall that an orbit-isometric embedding on $\E_{\lambda}$ (Definition~\ref{def:orbit-iso}) is a spectral embedding such that $\frac{1}{\abs{\Om_i}}\sum_{uv \in \Om_i} \|p_u - p_v \|^2$ is independent of $1 \leq i \leq s$. The following two results correspond to Lemma~\ref{lemma:cr-subgradient} and Theorem~\ref{thm:edge-iso} respectively.

\begin{corollary}\label{cor:orbit-iso}
    A graph $G$ with a group of automorphisms $\Psi \leq \Aut(G)$ is lower (upper) conformally rigid if and only if there exist nonzero eigenvectors $\varphi_1, \ldots, \varphi_r \in \E_{\lambda_2}$ ($\E_{\lambda_n}$) such that 
    \begin{equation}\label{eq:sum-ell-psi}
        \Om = \sum_{i=1}^r \ell_\Psi(\varphi_i).
    \end{equation}
\end{corollary}

\begin{theorem}\label{thm:orbit-iso}
    A graph $G$ with a group of automorphisms $\Psi \leq \Aut(G)$ is lower (upper) conformally rigid if and only if there exists an orbit-isometric embedding $\mathcal{P}$ on $\E_{\lambda_2}$ ($\E_{\lambda_n}$).
\end{theorem}

Any edge-isometric spectral embedding $\mathcal{P}$ is orbit-isometric, so it is clear that orbit-isometry is a necessary condition for conformal rigidity. To see why this condition is sufficient, we show that we can symmetrize an orbit-isometric embedding to get an edge-isometric embedding. First, we observe the following.

\begin{proposition}\label{prop:Li-invariance}
    For each orbit Laplacian $L^1, \ldots, L^s$, we have $\langle \sigma \cdot \varphi, L^i (\sigma \cdot \varphi)\rangle = \langle \varphi, L^i \varphi \rangle$ for all $\sigma \in \Psi$ and $\varphi: V \to \R$. Consequently, 
    \[\langle \sigma \cdot \varphi, L(\sigma \cdot \varphi) \rangle = \sum_{i=1}^s\langle \sigma \cdot \varphi,  L^i(\sigma \cdot \varphi) \rangle = \sum_{i=1}^s \langle \varphi, L^i \varphi \rangle = \langle \varphi, L \varphi \rangle.\]
    for all $\sigma \in \Psi$ and $\varphi: V \to \R$.
\end{proposition}
\begin{proof}
    Applying~\eqref{eq:Li} to $\sigma \cdot \varphi$, we get $\langle \sigma \cdot \varphi,  L^i (\sigma \cdot \varphi) \rangle = \sum_{uv \in \Om_i} \ell(\sigma \cdot \varphi)_{uv}$. By equivariance in~\eqref{eq:equivariance}, this is equal to 
    \[\sum_{uv \in \Om_i} \paren{\sigma \cdot \ell(\varphi)}_{uv} = \sum_{uv \in \Om_i} \ell(\varphi)_{\sigma^{-1}(u)\sigma^{-1}(v)}= \sum_{uv \in \Om_i} \ell(\varphi)_{uv} = \langle \varphi, L^i \varphi \rangle. \qedhere\]
\end{proof}

The above proposition tells us that acting on $\varphi$ by any $\sigma \in \Psi$ does not change its Dirichlet energy. In particular, each eigenspace of $L$ is invariant under the action of $\Psi$, so we can define the notion of a \textit{symmetrized embedding}.

\begin{definition}\label{def:symmetrized-embedding}
    Let $\mathcal{P}$ be a spectral embedding on $\E_\lambda$ corresponding to the matrix $P = \begin{pmatrix}
        \varphi_1 & \ldots & \varphi_r
    \end{pmatrix}$ where each $\varphi_i \in \E_\lambda$. Then for any $\sigma \in \Psi$, we have $\sigma \cdot \varphi_i \in \E_\lambda$, so we can define the $\Psi$-\textit{symmetrized embedding} of $\mathcal{P}$ to be the spectral embedding $\mathcal{P}_\Psi$ on $\E_\lambda$ corresponding to the matrix $P_\Psi$ whose columns are given by 
    \[\set{\sigma \cdot \varphi_i : \sigma \in \Psi, 1 \leq i \leq r}.\]
\end{definition}

\begin{proposition}
    Let $\mathcal{P}$ be an orbit-isometric embedding on $\E_\lambda$. Then the symmetrized embedding $\mathcal{P}_\Psi$ is edge-isometric.
\end{proposition}
\begin{proof}
    First we show that the squared length of each edge in $\mathcal{P}_\Psi = \set{p_1, \ldots, p_n}$ is constant within an edge orbit. Indeed, for any $uv \in E$
    \begin{equation}\label{eq:sym-embedding}
            \|p_u - p_v\|^2 = \sum_{\sigma \in \Psi} \sum_{i=1}^r \ell(\sigma \cdot \varphi_i)_{uv} = \sum_{\sigma \in \Psi} \sum_{i=1}^r \paren{\sigma \cdot \ell(\varphi_i)}_{uv} 
    \end{equation}
    which depends only on the edge orbit of $uv$. Hence for each $\Om_i$, there exists some $c_i \geq 0$ such that $\| p_u - p_v\|^2 = c_i$ for all $uv \in \Om_i$. 

    It remains to show that each $c_i$ is the same for all $1 \leq i \leq s$. Denote $\mathcal{P} = \set{q_1, \ldots, q_n}$. Since $\mathcal{P}$ is orbit-isometric, there exists $c >0$ such that~\eqref{eq:orbit-isometry} holds. By~\eqref{eq:sym-embedding}, we know 
    \[\|p_u-p_v\|^2 = \sum_{\sigma \in \Psi} \sum_{i=1}^r \ell(\varphi_i)_{\sigma^{-1}(u)\sigma^{-1}(v)} = \sum_{\sigma \in \Psi} \|q_{\sigma(u)} - q_{\sigma(v)}\|^2.\]
    Thus, for any edge orbit $\Om_i$
    \begin{align*}
        c_i = \frac{1}{\abs{\Om_i}} \sum_{uv \in \Om_i} \|p_u - p_v\|^2 &= \frac{1}{\abs{\Om_i}} \sum_{uv \in \Om_i} \sum_{\sigma \in \Psi} \|q_{\sigma(u)} - q_{\sigma(v)}\|^2 \\
        &= \frac{\abs{\Psi}}{\abs{\Om_i}} \sum_{uv \in \Om_i} \|q_{u} - q_{v}\|^2 = \abs{\Psi} \cdot c.
    \end{align*}
    Hence, $\mathcal{P}_\Psi$ is edge-isometric as desired.
\end{proof}

This proof gives a physical intuition of the symmetry reduction in terms of embeddings. Since we can always symmetrize an embedding to have constant squared edge length within each edge orbit of $\Psi$, $G$ is conformally rigid if and only if we can balance the average squared edge length across different orbits. This recovers the known result that edge-transitive graphs are both lower and upper conformally rigid since every edge is in the same orbit.
\begin{theorem} \cite[Proposition 2.1]{steinerberger}
    Every edge-transitive graph is both lower and upper conformally rigid.
\end{theorem}

Orbit-isometric embeddings also recover Theorem 4.8 in~\cite{gouveia}, which proves a version of Theorem~\ref{thm:sym-subgradient} when $G$ is vertex-transitive with respect to $\Psi$. The upshot of our perspective is that we do not assume vertex-transitivity; this allows us to address the sporadic examples brought up in~\cite[\S 1.5]{gouveia}. The authors list five graphs that have multiple vertex and edge orbits under their respective automorphism groups. We can now certify these without using the full unsymmetrized SDP. Moreover, we can somewhat explain why these graphs and other non-edge-transitive, conformally rigid graphs are conformally rigid: they contain few edge orbits that can be spectrally balanced. However, we still do not know any structural properties (other than edge-transitivity) that say when the orbits can be balanced. Below is an example of a certificate of one of the five previously unexplained graphs.

\begin{figure}[h]
    \begin{subfigure}{0.45 \textwidth}
        \centering
        \begin{tikzpicture}[
            scale=1,
            vertex/.style={circle, draw, fill=black, inner sep=1.2pt},]

            \node[vertex, label=above: $a$] (A) at (3.873, 4.000) {};
            \node[vertex, label=above: $b$] (B) at (2.620, 3.311) {};
            \node[vertex, label=right: $c$] (C) at (4.000, 2.749) {};
            \node[vertex, label=above: $d$] (D) at (2.944, 3.818) {};
            \node[vertex, label=below: $e$] (E) at (1.468, 0.696) {};
            \node[vertex, label=left: $f$] (F) at (0.222, 1.469) {};
            \node[vertex, label=below: $g$] (G) at (1.152, 0.000) {};
            \node[vertex, label=left: $h$] (H) at (0.193, 0.154) {};
            \node[vertex, label=above right: $i$] (I) at (1.263, 2.506) {};
            \node[vertex, label=below: $j$] (J) at (3.280, 0.473) {};
            \node[vertex, label=below: $k$] (K) at (2.625, 0.336) {};
            \node[vertex, label=above: $l$] (L) at (2.135, 1.105) {};
            \node[vertex, label=left: $m$] (M) at (1.074, 1.749) {};
            \node[vertex, label=above: $n$] (N) at (0.000, 3.194) {};
            \node[vertex, label=above: $o$] (O) at (2.325, 2.556) {};
            \node[vertex, label=above: $p$] (P) at (1.241, 3.999) {};
            \node[vertex, label=right: $q$] (Q) at (3.911, 1.204) {};
            \node[vertex, label=right: $r$] (R) at (3.452, 1.814) {};
            \node[vertex, label=right: $s$] (S) at (2.800, 1.621) {};
            \node[vertex, label=above: $t$] (T) at (0.990, 3.239) {};

            \draw[blue, thick] (R) -- (K);
            \draw[blue, thick] (R) -- (D);
            \draw[blue, thick] (L) -- (F);
            \draw[blue, thick] (M) -- (O);
            \draw[blue, thick] (N) -- (P);
            \draw[blue, thick] (Q) -- (L);
            \draw[blue, thick] (D) -- (P);
            \draw[blue, thick] (Q) -- (C);
            \draw[blue, thick] (K) -- (G);
            \draw[blue, thick] (C) -- (O);
            \draw[blue, thick] (M) -- (G);
            \draw[blue, thick] (F) -- (N);
            \draw[orange, thick] (Q) -- (J);
            \draw[orange, thick] (I) -- (O);
            \draw[orange, thick] (A) -- (D);
            \draw[orange, thick] (S) -- (L);
            \draw[orange, thick] (G) -- (H);
            \draw[orange, thick] (A) -- (C);
            \draw[orange, thick] (T) -- (M);
            \draw[orange, thick] (T) -- (N);
            \draw[orange, thick] (J) -- (R);
            \draw[orange, thick] (I) -- (P);
            \draw[orange, thick] (S) -- (K);
            \draw[orange, thick] (F) -- (H);
            \draw[violet, thick] (A) -- (B);
            \draw[violet, thick] (B) -- (S);
            \draw[violet, thick] (B) -- (T);
            \draw[violet, thick] (J) -- (E);
            \draw[violet, thick] (I) -- (E);
            \draw[violet, thick] (E) -- (H);

        \end{tikzpicture}
        \subcaption{Crossing Number 6B}\label{fig:cn6b}
    \end{subfigure}
    \begin{subfigure}{0.45 \textwidth}
        \centering
        \begin{tikzpicture}[
            scale=1.00000000000000,
            vertex/.style={circle, draw, fill=black, inner sep=1.2pt},]

            \node[vertex, label=above: $a$] (A) at (2.293, 4.000) {};
            \node[vertex, label=above: $b$] (B) at (4.000, 3.688) {};
            \node[vertex, label=above: {$c,d$}] (C) at (1.293, 3.156) {};
            \node[vertex] (D) at (1.293, 3.156) {};
            \node[vertex, label=below: $e$] (E) at (0.000, 0.312) {};
            \node[vertex, label=right: {$f,g$}] (F) at (2.707, 0.844) {};
            \node[vertex] (G) at (2.707, 0.844) {};
            \node[vertex, label=below: $h$] (H) at (1.707, 0.000) {};
            \node[vertex, label=left: $i$] (I) at (0.000, 1.451) {};
            \node[vertex, label=right: $j$] (J) at (0.293, 1.173) {};
            \node[vertex, label={[yshift=-2pt]right: {$k,l$}}] (K) at (3.000, 1.705) {};
            \node[vertex] (L) at (3.000, 1.705) {};
            \node[vertex, label={[xshift=-3pt] above: {$m,n$}}] (M) at (2.707, 1.983) {};
            \node[vertex] (N) at (2.707, 1.983) {};
            \node[vertex, label=above left: {$o,p$}] (O) at (1.000, 2.295) {};
            \node[vertex] (P) at (1.000, 2.295) {};
            \node[vertex, label={[xshift=2pt]below: {$q,r$}}] (Q) at (1.293, 2.017) {};
            \node[vertex] (R) at (1.293, 2.017) {};
            \node[vertex, label=right: $s$] (S) at (4.000, 2.549) {};
            \node[vertex, label=above left: $t$] (T) at (3.707, 2.827) {};

            \draw[blue, thick] (R) -- (K);
            \draw[blue, thick] (R) -- (D);
            \draw[blue, thick] (L) -- (F);
            \draw[blue, thick] (M) -- (O);
            \draw[blue, thick] (N) -- (P);
            \draw[blue, thick] (Q) -- (L);
            \draw[blue, thick] (D) -- (P);
            \draw[blue, thick] (Q) -- (C);
            \draw[blue, thick] (K) -- (G);
            \draw[blue, thick] (C) -- (O);
            \draw[blue, thick] (M) -- (G);
            \draw[blue, thick] (F) -- (N);
            \draw[orange, thick] (Q) -- (J);
            \draw[orange, thick] (I) -- (O);
            \draw[orange, thick] (A) -- (D);
            \draw[orange, thick] (S) -- (L);
            \draw[orange, thick] (G) -- (H);
            \draw[orange, thick] (A) -- (C);
            \draw[orange, thick] (T) -- (M);
            \draw[orange, thick] (T) -- (N);
            \draw[orange, thick] (J) -- (R);
            \draw[orange, thick] (I) -- (P);
            \draw[orange, thick] (S) -- (K);
            \draw[orange, thick] (F) -- (H);
            \draw[violet, thick] (A) -- (B);
            \draw[violet, thick] (B) -- (S);
            \draw[violet, thick] (B) -- (T);
            \draw[violet, thick] (J) -- (E);
            \draw[violet, thick] (I) -- (E);
            \draw[violet, thick] (E) -- (H);

        \end{tikzpicture}
        \subcaption{Orbit-isometric embedding on $\E_{\lambda_2}$}\label{fig:orb-iso}
    \end{subfigure}
    \caption{}
\end{figure}
\begin{example}\label{ex:cn6b}
    Let $G$ be the Crossing Number 6B graph (HoG ID 1004) on 20 vertices with 30 edges. $G$ has $\lambda_2 = 1$ with multiplicity 3. One can check that the following is a basis for $\E_{\lambda_2}$ where the ordering of vertices is given by Figure~\ref{fig:cn6b}.
    \begin{align*}
        \varphi_1^T &= (1, 0, 1, 1, 0, -1, -1, -1, 0, 1, 0, 0, -1, -1, 0, 0, 1, 1, 0, -1)\\
        \varphi_2^T &= (0, 1, -\frac{1}{2}, -\frac{1}{2}, -1, \frac{1}{2}, \frac{1}{2}, 0, 0, -2, -\frac{1}{2}, -\frac{1}{2}, \frac{3}{2}, \frac{3}{2}, \frac{1}{2}, \frac{1}{2}, -\frac{3}{2}, -\frac{3}{2}, 0, 2)\\
        \varphi_3^T &= (0, 0, 0, 0, 0, 0, 0, 0, 1, -1, -1, -1, 1, 1, 1, 1, -1, -1, -1, 1)
    \end{align*}
    Under the full automorphism group $\Psi = \Aut(G)$, there are three edge orbits, colored in Figure~\ref{fig:cn6b}, and three vertex orbits. Now instead of trying to satisfy 30 edge constraints, we only have to satisfy three edge orbit constraints. One can check that the $1$-dimensional spectral embedding given by 
    \[\varphi = \paren{-\frac{\sqrt{2}}{2} + 1} \cdot \varphi_1 + 2 \cdot \varphi_2 - 2 \cdot \varphi_3\]
    is orbit-isometric. Figure~\ref{fig:orb-iso} shows a $2$-dimensional orbit-isometric embedding on $\E_{\lambda_2}$ where the $x$-coordinates are given by $\varphi$, and the $y$-coordinates are given by another 1-dimensional orbit-isometric embedding, which we specify in Example~\ref{ex:cn6b2}. We remark that this orbit-isometric embedding is not injective. $G$ is also upper conformally rigid as it is regular and bipartite, so the alternating sign vector gives a $1$-dimensional edge-isometric embedding on $\E_{\lambda_n}$.
\end{example}

\subsection{Symmetry Reduced SDP}\label{subsec:sym-sdp} We now give the symmetry reduced SDPs corresponding to~\eqref{eq:cr-sdp} and~\eqref{eq:cr-sdp-param}. We can apply the exact same logic to verify an orbit-isometric embedding on $\E_\lambda$:
\begin{equation}\label{eq:sym-sdp}
    \begin{split}
        \text{find} \ &\Phi \succeq 0,\\
        \text{s.t.} \ &L\Phi = \lambda \Phi, \\
        & \langle L^i, \Phi \rangle_F = \abs{\Om_i} \text{ for all } 1 \leq i \leq s.
    \end{split}
\end{equation}
We can parametrize~\eqref{eq:sym-sdp} with an orthonormal basis $B$ of $\E_\lambda$ to get
\begin{equation}\label{eq:sym-sdp-param}
    \begin{split}
        \text{find} \ &Z \succeq 0,\\
        \text{s.t.} \ & \langle \tilde{L}^i, Z \rangle_F = \abs{\Om_i}  \text{ for all } 1 \leq i \leq s.
    \end{split}
\end{equation}
where $\tilde{L}^i$ is the compressed operator $B^TL^iB$. The same correspondence of solutions $\Phi = BZB^T$ for~\eqref{eq:cr-sdp} holds as before. Now, to check if a graph $G$ with some group of automorphisms $\Psi$ is conformally rigid, we only have to check the consistency of a system of $s$ equations instead of $\abs{E}$ equations. In the case of highly symmetric graphs, $s$ can be a lot less than $\abs{E}$. Since the runtime to solve an SDP depends cubically on the number of constraints, the computational speedup can be massive. Before we get to a large-scale example, let us revisit our running example.

\begin{rexample2}[continued]
    Recall the graph $G = \Cay(\Z_{12}, \set{2,3})$ has two edge orbits $\Om_2, \Om_3$ under the group of automorphisms induced by $\Z_{12}$. This reduces the number of SDP constraints from $m=24$ to $s=2$. Interestingly, solving this SDP with MOSEK returns a solution nearly identical to the solution~\eqref{eq:sdp-certificate} in the unsymmetrized SDP.    
\end{rexample2}

All the graphs that we tested from the House of Graphs database had at most $n=250$ vertices, and their unsymmetrized SDPs could be solved quickly by MOSEK. To see the full power of symmetry reduction, we need a much bigger graph.

\begin{example}\label{ex:8000}
    Let $G_1$ be the Crossing Number 6B graph, $G_2$ be the (10,3)-Incidence Graph (HoG ID 33538), and $G_3$ be the Desargues Graph (HoG ID 1036). All three graphs have $\lambda_2=1$ with multiplicities $3,3,$ and $4$ respectively. The first two come from the sporadic examples in~\cite[\S 1.5]{gouveia}, and the last one is edge-transitive, so all are lower conformally rigid. By Lemma 2.6 in~\cite{gouveia}, the Cartesian product $G = G_1 \square G_2 \square G_3$ must also be lower conformally rigid. $G$ has $n=8000$ vertices, $m=36000$ edges, and has $\lambda_2=1$ with multiplicity 10. To test the efficacy of our symmetry reduction, we certify the lower conformal rigidity of $G$ without using any knowledge of the structure of $G$ as a Cartesian product, using MOSEK. 

    As a baseline, we compute the unsymmetrized SDP with matrices of size $\dim \E_{\lambda_2} = 10$ and $36000$ edge constraints. The solver took 70 seconds to solve the SDP and returned the certificate $\Phi$ shown in Figure~\ref{fig:8000certificate}. Because we constructed $G$ as a Cartesian product, we can see the $400 \times 400$ and $20 \times 20$ subblocks given by the product structure, but the solver had no knowledge of this structure.

    \begin{figure}[h]
        \centering
        \includegraphics[width=\textwidth]{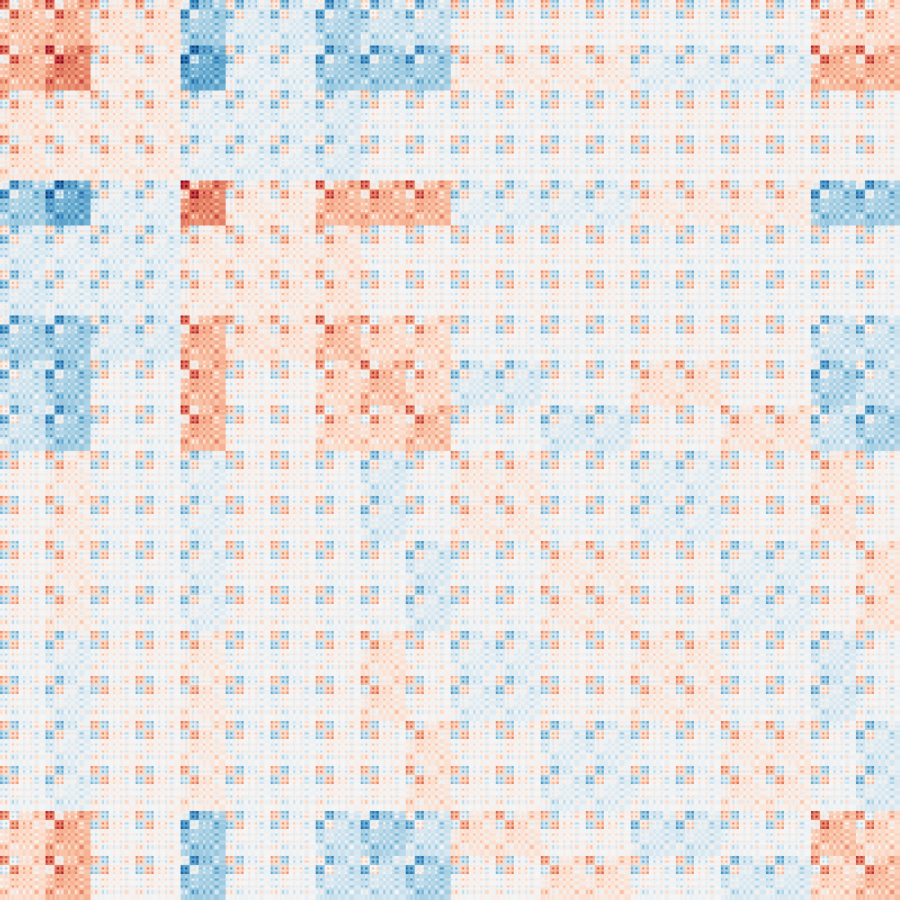}
        \caption{Lower conformal rigidity certificate to~\eqref{eq:cr-sdp} for the graph on 8000 vertices.}
        \label{fig:8000certificate}
    \end{figure}

    We then solved the symmetry reduced SDP with the full automorphism group $\Psi = \Aut(G)$, which has order $46080$, and breaks up $G$ into $51$ edge orbits. The SDP ran in under one-tenth of a second. The main bottleneck was computing the automorphism group using \textsc{SageMath}'s~\cite{sage} \verb|automorphism_group()| method, which took 32 seconds. In practice, even using a handful of automorphisms dramatically reduces the SDP computation, so rather than taking $\Psi$ to be the full automorphism group, one can use a smaller subgroup of automorphisms. \textsc{SageMath}'s built-in method computes automorphisms through partition refinement, so we can find all graph automorphisms that fix a subset of vertices. Doing this with 1 and 2 fixed vertices, we achieved the runtimes shown in Table~\ref{tab:8000runtimes}. We see that even working with just $\abs{\Psi} = 4$ reduces the time to solve the SDP to just 15 seconds while only taking 2 seconds to compute the automorphisms. The sweet spot seemed to be the (relatively) small automorphism group with $\abs{\Psi} = 96$.

\begin{table}[h]
    \centering
    \begin{tabular}{rrrrrr}
        \toprule
        $|\Psi|$ & \makecell{Number of \\edge orbits} & \makecell{Time to \\ compute $\Psi$} & \makecell{Time to find \\ edge orbits} & \makecell{SDP \\ solve time} & \makecell{Total \\ time} \\
        \midrule
        46080 & 51    & 31.94s & 0.10s & 0.09s  & 32.13s \\
        96    & 2528  & 7.52s  & 0.07s & 1.23s  & 8.82s  \\
        4     & 14900 & 2.43s  & 0.20s & 15.11s & 17.74s \\
        1     & 36000 & 0s     & 0s    & 70.40s & 70.40s \\
        \bottomrule
    \end{tabular}
    \caption{Runtime comparison for symmetry reduction on the 8000-vertex graph.}
    \label{tab:8000runtimes}
\end{table}

\end{example}

\subsection{Exact Certificates}

We have seen how symmetry reduction can drastically reduce the SDP computation to certify conformal rigidity. However, this still does not address the issue that the certificates returned by SDPs are only approximate. We present a method for exact certification assuming the existence of rank-1 solutions for~\eqref{eq:sym-sdp} and~\eqref{eq:sym-sdp-param}.

Observe that in Example~\ref{ex:cn6b}, we could certify conformal rigidity with just one eigenvector. This corresponds to a 1-dimensional orbit-isometric embedding or a rank-1 SDP certificate. A natural question to ask is: does a conformally rigid graph always have a rank-1 SDP certificate? This has been asked before for vertex-transitive and Cayley graphs in~\cite[Question 1]{gouveia}. In the absence of any symmetry reduction, Corollary~\ref{cor:bipartite} tells us that no 1-dimensional edge-isometric embedding exists unless $G$ is bipartite. We will show in \S~\ref{sec:rep-theory}, however, that it is sufficient to check for 1-dimensional orbit-isometric embeddings for a large class of graphs. 

For now, let us assume we are working with candidate graphs $G$ that are conformally rigid if and only if they have a 1-dimensional orbit-isometric embedding. Under this assumption,~\eqref{eq:sym-sdp} becomes a quadratic system of equations on $\E_{\lambda}$:
\begin{equation}\label{eq:grobner}
    \begin{split}
        \text{find} \ &\varphi \in \E_{\lambda}\\
        \text{s.t.} \ &L\varphi = \lambda \varphi, \\
        & \langle \varphi, L^i \varphi \rangle = \abs{\Om_i} \text{ for all } 1 \leq i \leq s.
    \end{split}
\end{equation}
Since $\lambda$ is a real algebraic number, this can be parametrized by a basis $B$ (not necessarily orthonormal) of $\E_{\lambda}$ with real algebraic coefficients, yielding the smaller quadratic system
\begin{equation}\label{eq:grobner-param}
    \begin{split}
        \text{find} \ &x \in \R^d\\
        \text{s.t.} \ &\langle Bx, L^i Bx \rangle = \abs{\Om_i} \text{ for all } 1 \leq i \leq s
    \end{split}
\end{equation}
in $d$ variables and $s$ constraints. This defines a variety over $\bar{\Q}$, and by the Tarski-Seidenberg principle, the existence of a real solution implies a real algebraic solution~\cite[Corollary 1.4.7]{bochnak}. Hence, we can solve for an exact solution using standard tools from computer algebra, such as Gr\"{o}bner bases~\cite[\S 2]{cox} or cylindrical algebraic decomposition (CAD)~\cite{collins}. Both of these methods involve at least exponential complexity: the worst-case space complexity for Gr\"{o}bner bases and worst-case time complexity for CAD is doubly exponential in the number of variables~\cite{collins, mayr}. Because of the prohibitively expensive runtime, this approach is only possibly feasible after symmetry reduction and parametrization. For example, we could not solve~\eqref{eq:grobner} for Crossing Number 6B using \textsc{SageMath}, but~\eqref{eq:grobner-param} is simple enough to solve by hand.

\begin{example}\label{ex:cn6b2}
    Recall that $G$ from Example~\ref{ex:cn6b} had 3 edge orbits of order 12, 12, and 6. Using the basis $B = \begin{pmatrix}
        \varphi_1 & \varphi_2 & \varphi_3
    \end{pmatrix}$ as our parametrization, we get the system~\eqref{eq:grobner-param} of 3 quadratics in 3 variables given by
    \begin{align*}
        12 &= 8x_0^2 - 16x_0 x_1 + 12x_1^2 - 8x_0 x_2 + 16x_1 x_2 + 8x_2^2 , \\
        12 &= 3x_1^2, \\
        6 &= 4x_0^2 - 8x_0 x_1 + 6x_1^2 - 4x_0 x_2 + 8x_1 x_2 + 4x_2^2.
    \end{align*}
    Computing a Gr\"{o}bner basis for the ideal generated by these polynomials yields 
    \begin{equation*}
        \left(x_0^2 - 2x_0 x_1 - x_0 x_2 + 2x_1 x_2 + x_2^2 
        + \frac{9}{2}, \ x_1^2 - 4\right).
    \end{equation*}
    This variety decomposes into two irreducible components corresponding to $x_1 = 2$ and $x_1 = -2$, each of which is an ellipse in the $(x_0, x_2)$-plane. The two components arise from the fact that $-\varphi$ is a solution whenever $\varphi$ is. Each point in this ellipse gives an exact 1-dimensional orbit-isometric embedding for $\E_{\lambda_2}$. 

    In contrast, the projection method used to certify the same graph in~\cite[\S 6.2]{steinerberger} manually inspects a numerical certificate to~\eqref{eq:cr-sdp} returned by an SDP solver and (successfully) attempts to round to a nearby solution with rational or algebraic coefficients of low degree. This approach can be finicky and likely has varying success depending on the solver and the chosen basis of $\E_\lambda$. Our version computes \textit{all} 1-dimensional orbit-isometric embeddings exactly. This is how we generated Figure~\ref{fig:orb-iso}: the $x$-values are specified by the eigenvector corresponding to $x_1 = 2, x_2=-2$, and $y$-values corresponding to $x_1=2, x_2=-0.65$ as shown in Figure~\ref{fig:ellipses}.
\end{example}
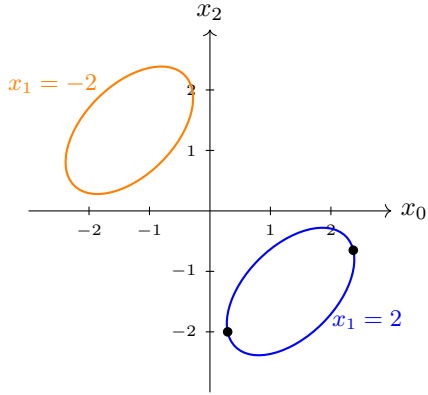
\begin{figure}[h]
    \centering
    \begin{tikzpicture}[scale=0.8]
        \draw[->] (-3,0) -- (3,0) node[right] {$x_0$};
        \draw[->] (0,-3) -- (0,3) node[above] {$x_2$};
        \foreach \x in {-2,-1,1,2} {
            \draw (\x,2pt) -- (\x,-2pt) node[below, font=\tiny] {$\x$};
            \draw (2pt,\x) -- (-2pt,\x) node[left,  font=\tiny] {$\x$};
        }

        \draw[blue, thick, rotate around={45:(1.333,-1.333)}]
            (1.333,-1.333) ellipse (1.291 and 0.745);
        \node[blue, font=\small] at (2.6,-1.8) {$x_1=2$};

        \draw[orange, thick, rotate around={45:(-1.333,1.333)}]
            (-1.333,1.333) ellipse (1.291 and 0.745);
        \node[orange, font=\small] at (-2.6,2.1) {$x_1=-2$};

        \filldraw[black] (2.370,-0.65) circle (2pt);
        \filldraw[black] (0.293,-2) circle (2pt);
    \end{tikzpicture}
    \caption{Valid $(x_0, x_2)$ values for orbit-isometric embeddings of $G$ on $\E_{\lambda_2}$. Points indicate the embeddings used in Figure~\ref{fig:orb-iso}.}
    \label{fig:ellipses}
\end{figure}

We remark that our method still has some major drawbacks. In the previous example, the dimension of the ideal corresponding to valid orbit-isometric embeddings was one, so it was easy to find a real solution. Apart from the exponential runtime to compute Gr\"{o}bner bases, the ideal of orbit-isometric embeddings for some graphs can have dimension two or larger. This is somewhat expected because conformally rigid graphs have a lot of symmetry, so these systems of equations are typically degenerate. In these cases, it can be hard to read off a real algebraic solution by naively computing Gr\"{o}bner bases without additional structure. We will address this issue in \S~\ref{sec:polyhedral-cone}.
\section{Representation Theory}\label{sec:rep-theory}

In the previous section, we showed how we could leverage symmetry reduction to speed up the numerical certification of conformal rigidity. We also showed that in the case where a 1-dimensional orbit-isometric embedding exists for a conformally rigid graph $G$, we can get an exact certificate using Gr\"{o}bner bases. In this section, we will prove that for a wide class of graphs, including all vertex-transitive graphs, the image $\ell_\Psi(\E_\lambda)$ is convex in Theorems~\ref{thm:convex-image} and~\ref{thm:vertex-transitive}. Here is how this helps us: by Corollary~\ref{cor:orbit-iso}, $G$ is conformally rigid if and only if there exist nonzero eigenvectors satisfying~\eqref{eq:sum-ell-psi}, which is a conical combination of eigenvector orbit-energies. Thus, another way of phrasing Corollary~\ref{cor:orbit-iso} is that $G$ is conformally rigid if and only if $\Om \in \Cone \set{\ell_\Psi(\E_\lambda)}$. If we can show that the image $\ell_\Psi(\E_\lambda)$ is convex, then this would be equivalent to finding
\[\Om \in \Cone \set{\ell_\Psi(\E_\lambda)} = \Conv \set{\ell_\Psi(\E_\lambda)} = \ell_\Psi(\E_\lambda),\]
where the first equality holds because $\ell_\Psi(\E_\lambda)$ is already a cone containing the origin. In this case, $G$ is conformally rigid if and only if there exists $\varphi \in \E_{\lambda}$ such that $\ell_\Psi(\varphi) = \Om$ i.e., there exists a 1-dimensional orbit-isometric embedding. The proof that this holds for vertex-transitive graphs requires representation theory and finding a \textit{symmetry adapted} basis of $\E_\lambda$ with respect to $\Psi$. First, we establish notation and rephrase our symmetry reduction in the language of representation theory. For representation theory basics, we refer to~\cite{fulton} or~\cite{serre}.

Recall that a \textit{(real) linear representation} of a finite group $\Gamma$ is a real vector space $X$ together with a group homomorphism $\rho: \Gamma \to \GL(X)$. The \textit{dimension} of the representation is given by $n = \dim X$, which we assume is finite. We follow standard notation and denote $\Hom_\Gamma(X, X')$ as the set of $\Gamma$-linear maps between two $\Gamma$ representations $(X, \rho)$ and $(X', \rho')$. Explicitly, for any $T \in \Hom_\Gamma(X,X')$ and $g \in \Gamma$, we have $\rho'(g)T= T\rho(g)$. In the case that $X=X'$, we denote the algebra of $\Gamma$-endomorphisms by $\End_\Gamma(X) = \Hom_\Gamma(X,X)$.

We have seen that the action of a group of automorphisms $\Psi \leq \Aut(G)$ on the vertices $V$ induces an action by permutation matrices on $\R^V$. This can be interpreted as a \textit{permutation representation} $\rho: \Psi \to \GL(\R^V)$, which is a representation where all matrices $\rho(\sigma)$ are permutation matrices. In particular, the representation is \textit{unitary} with respect to the standard dot product, which means $\rho(\sigma)^{-1} = \rho(\sigma)^T$ for all $\sigma \in \Psi$.  Rewriting Proposition~\ref{prop:Li-invariance} with this notation, we get that 
\begin{equation}\label{eq:Li-invariant2}
    \langle \varphi, (\rho(\sigma)^{-1} L^i \rho(\sigma)) \varphi \rangle = \langle \varphi, (\rho(\sigma)^T L^i \rho(\sigma)) \varphi \rangle = \langle \varphi, L^i \varphi \rangle
\end{equation}
for all $\varphi \in \R^V$ and $1 \leq i \leq s$. Hence, $\rho(\sigma)L^i = L^i \rho(\sigma)$, so $L^i \in \End_\Psi(\R^V)$, and consequently $L \in \End_\Psi(\R^V)$. We have seen that this means every eigenspace $\E_\lambda$ of $L$ is $\rho(\sigma)$-invariant for all $\sigma \in \Psi$, which means $\E_\lambda$ is a \textit{subrepresentation} of $\R^V$. 

To determine whether there is an orbit-isometric embedding on $\E_\lambda$ for an eigenvalue $\lambda > 0$, we need to understand how the quadratic map $\ell_\Psi(\varphi) = \paren{\langle \varphi, L^i \varphi \rangle}_{i=1}^s$ behaves when restricted to $\E_\lambda$. Determining the convexity of a complex image of a tuple of quadratic forms --- called the joint numerical range --- is a rich and classical area of research~\cite{muller}. Our setting differs from standard joint numerical range problems as we are only considering tuples of real symmetric matrices applied to real vectors. We are also considering matrices $L^i$ that all respect the same algebraic structure: equation~\eqref{eq:Li-invariant2} tells us that each $\langle \varphi, L^i \varphi \rangle$ defines a \textit{$\Psi$-invariant quadratic form }on $\E_\lambda$. In the next section, we will show that $\Psi$-invariant quadratic forms can be described nicely in terms of the representation theory of $\Psi$. Under certain assumptions on the subrepresentation $\E_\lambda$, we can conclude that $\ell_\Psi(\E_\lambda)$ is convex.

\subsection{Orbit-Energy Structure Theorem}

In this section let $\Gamma$ denote a finite group, and let $(X, \rho)$ be a real $n$-dimensional $\Gamma$-representation. Given a quadratic map $Q: X \to \R^s$ where $Q = (Q_i)_{i=1}^s$ and each $Q_i: X \to \R$ is a $\Gamma$-invariant quadratic form, we would like to determine when $Q(X)$ is convex. In this section, we give a necessary condition for when $Q(X)$ is convex that depends only on the representation $X$ and not the map $Q$. 

Given a quadratic form $Q$, there is a corresponding symmetric bilinear form $B: X \times X \to \R$ such that $Q(x) = B(x,x)$. If we fix an inner product on $X$, any such bilinear form has a corresponding \textit{self-adjoint} endomorphism $T: X \to X$ satisfying $B(x,y) = \langle x, Ty \rangle_X$.
If the inner product is taken to be $\Gamma$-invariant (whose existence is guaranteed by Maschke's theorem~\cite[Theorem 1]{serre}), that is,
\[ \langle \rho(g) x, \rho(g) y\rangle_X = \langle x, y \rangle_X,\]
then $Q$ is $\Gamma$-invariant if and only if $T \in \End_\Gamma(X)$:

\begin{proposition}
    Fix a $\Gamma$-invariant inner product $\langle \cdot, \cdot \rangle_X$ on $X$. If $Q$ is a quadratic form given by $Q(x) = \langle x, Tx \rangle_X$ for some self-adjoint $T: X \to X$, then $Q$ is $\Gamma$-invariant if and only if $T \in \End_\Gamma(X)$.
\end{proposition}
\begin{proof}
    If $T \in \End_\Gamma(X)$, then for all $g \in \Gamma$ and $x \in X$, we have 
    \[Q(\rho(g)x) = \langle \rho(g) x, T \rho(g) x \rangle_X = \langle \rho(g) x, \rho(g) Tx \rangle_X = \langle x, T x \rangle_X = Q(x).\]
    On the other hand, if $Q(x)$ is $\Gamma$-invariant, we know 
    \[\langle x, Tx \rangle_X = Q(x) = Q(\rho(g)x) = \langle \rho(g) x, T \rho(g) x \rangle_X = \langle x, \rho(g)^{-1} T \rho(g) x \rangle_X\]
    for all $x \in X$. Thus, $T = \rho(g)^{-1}T \rho(g)$ for all $g \in \Gamma$, so $T \in \End_\Gamma(X)$.
\end{proof}

From now on, we fix a $\Gamma$-invariant inner product $\langle \cdot, \cdot \rangle_X$ on $X$, so we can identify any $\Gamma$-invariant quadratic form $Q$ with a self-adjoint endomorphism $T \in \End_\Gamma(X)$. This identification with $\End_\Gamma(X)$ allows us to apply \textit{Schur's lemma} to describe the structure of $Q$. Schur's lemma describes $\Gamma$-linear maps between \textit{irreducible representations} of $\Gamma$, that is, representations of $\Gamma$ with no nontrivial subrepresentations.

\begin{lemma}[Schur]~\cite[\S 13.2]{serre}\label{lemma:schur}
    Let $U$ and $U'$ be real irreducible representations of $\Gamma$. If
    \begin{itemize}
        \item $U \not \cong U'$, then $\Hom_\Gamma(U, U') = \set{0}$, so there are no nonzero $\Gamma$-invariant maps between $U$ and $U'$;
        \item $U = U'$, then $\End_\Gamma(U) \cong D$ where $D$ is a finite dimensional division algebra over $\R$. Thus, $D = \R, \C$, or the quaternions $\HH$.
    \end{itemize}
\end{lemma}
We call $U$ of \textit{real, complex,} or \textit{quaternionic type} when $\End_\Gamma(U) \cong \R, \C,$ or $\HH$ respectively. For simplicity, we assume for the rest of this section that $U$ is of real type --- sometimes referred to as an \textit{absolutely irreducible} representation. The generalized cases of the following results where $U$ is of complex or quaternionic type follow essentially the same proof but involve some additional technicalities. We prove the generalized case in \S~\ref{subsec:complex-quaternionic}. When $U$ is of real type, $\End_\Gamma(U) \cong \R$ consists of scalar multiples of the identity: $\End_\Gamma(U) = \set{a \cdot \Id_U : a \in \R}$. By Maschke's theorem~\cite[\S 1.3]{serre}, $X$ is completely reducible and has a canonical decomposition 
\begin{equation}\label{eq:isotypics}
    X = X_1 \oplus X_2 \oplus \ldots \oplus X_h.
\end{equation}
Here each $X_i$ is called an \textit{isotypic} component and is isomorphic to the representation $(U_i \otimes \R^{m_i}, \rho_i \otimes 1)$ where $(U_i, \rho_i)$ is a real irreducible representation of $\Gamma$ of dimension $d_i$ and $(\R^{m_i}, 1)$ is the trivial representation. Each $U_i$ is pairwise nonisomorphic, and $m_i$ denotes the multiplicity of $U_i$ in the representation $X$. Thus, we can think of $X_i$ as $m_i$ copies of the irreducible representation $U_i$. Moreover, with respect to any $\Gamma$-invariant inner product, the isotypic components are pairwise orthogonal: $X_i \bot X_j$ for $i \neq j$. 

Schur's lemma tells us that $\Gamma$-invariant maps between absolutely irreducible representations can only be scalar multiples of the identity; the isotypic decomposition tells us that we can decompose $\Gamma$-invariant maps of arbitrary representations in terms of many smaller maps between irreducible representations. For any $T \in \End_\Gamma(X)$ and $1 \leq i \leq h$, we have $T(X_i) \subseteq X_i$.
This means if we choose an orthonormal basis for each $X_i$, we can block diagonalize $T$. Equivalently, 
\[\End_\Gamma(X) \cong \End_\Gamma(X_1) \oplus \ldots \oplus \End_\Gamma(X_h).\]
If $T$ corresponds to a quadratic form $Q$, this corresponds to a block-diagonalization of $Q$. If we let $x = \sum_{i=1}^h x_i \in X$ where $x_i \in X_i$, we have
\begin{align}
    Q(x) = \langle x, Tx \rangle_X &= \sum_{i,j=1}^h \langle x_i, T x_j \rangle_X = \sum_{i=1}^h \langle x_i, T x_i \rangle_X = \sum_{i=1}^h Q(x_i), \text{ so} \nonumber\\
    Q(X) &= \sum_{i=1}^h Q(X_i). \label{eq:diagonal-quadratic}
\end{align}
In other words, the image $Q(X)$ is the Minkowski sum of the images $Q(X_i)$. Thus, to describe $Q(X)$, it suffices to describe each $\Gamma$-invariant quadratic form $Q\vert_{X_i}$. With this reduction, we can assume $X \cong U \otimes \R^m$ is an isotypic representation given by $m$ copies of an irreducible representation $(U, \rho)$ of $\Gamma$. We also fix $\Gamma$-invariant inner products $\langle \cdot, \cdot \rangle_X$ and $\langle \cdot, \cdot \rangle_U$. First, we show that we can choose an isomorphism $\Phi: U \otimes \R^m \to X$ that is compatible with the inner products.

\begin{proposition}~\label{prop:iso-inner-prod}
    Assume $U$ is absolutely irreducible, and $X \cong U \otimes \R^m$ as representations. Fix $\Gamma$-invariant products $\langle \cdot, \cdot \rangle_U$ and $\langle \cdot, \cdot \rangle_X$. There exists a $\Gamma$-isomorphism $\Phi: U \otimes \R^m \to X$ such that the induced inner product acts on simple tensors by $\langle \Phi(u \otimes z), \Phi(u' \otimes z')\rangle_X = \langle u, u' \rangle_U \cdot \langle z, z' \rangle_{\R^m}$.
\end{proposition}
\begin{proof}
    Since $X$ is isomorphic to $m$ copies of $U$, we have inclusions $\iota_i: U \to X$ that are $\Gamma$-equivariant. Moreover, we can choose these inclusions so that $\iota_i(U) \bot \iota_j(U)$ are orthogonal subspaces with respect to the $\Gamma$-invariant inner product on $X$. Each $\iota_i$ induces a $\Gamma$-invariant inner product on $U$ given by $\langle u, u' \rangle_i = \langle \iota_i(u), \iota_i(u')\rangle_X$. Since $\langle \cdot, \cdot \rangle_i$ is a symmetric $\Gamma$-invariant bilinear form, there exists a self-adjoint $T_i \in \End_\Gamma(U)$ such that 
    $\langle u, u' \rangle_i = \langle u, T_i(u') \rangle_U$.
    But $U$ is absolutely irreducible, so $T_i = a_i \cdot \Id_U$ for some $a_i \in \R$. Thus, $\langle u, u' \rangle_i = a_i \langle u, u' \rangle_U$.
    Moreover, since $T_i$ corresponds to a positive-definite bilinear form, $a_i > 0$, so we may replace $\iota_i$ with $\sqrt{a_i} \cdot \iota_i$. With this substitution, the induced bilinear forms $\iota_i$ will be equal to $\langle \cdot, \cdot \rangle_U$. 
    
    Now define the linear map $\Phi: U \otimes \R^m \to X$ by $\Phi(u \otimes e_i) = \iota_i(u)$, and extend linearly. We see for simple tensors $u \otimes e_i$ and $u' \otimes e_j$, we have
    \[ \langle \Phi(u \otimes e_i), \Phi(u' \otimes e_j) \rangle_X  = \langle u, u' \rangle_U \cdot \langle e_i, e_j \rangle_{\R^m}.\]
    Extending bilinearly gives us our result.
\end{proof}
From now on we identify $X = U \otimes \R^m$ with such an isomorphism so we can write any $x \in X$ uniquely as $x = \sum_{i=1}^m u_i \otimes e_i, \ u_i \in U$. We are working with finite-dimensional representations of a finite group $\Gamma$, so we have a natural isomorphism 
\[\End_\Gamma(X) = \End_\Gamma(U \otimes \R^m) \cong \End_\Gamma(U) \otimes \End_\Gamma(\R^m).\]
Since $\R^m$ is the trivial representation of $\Gamma$, we identify $\End_\Gamma(\R^m)$ with $\R^{m \times m}$ via the standard basis $e_1, \ldots, e_m$ to get
\begin{equation}\label{eq:end-u-otimes-rm}
    \End_\Gamma(X) \cong \End_\Gamma(U) \otimes \R^{m \times m}.
\end{equation}

With the identification in~\eqref{eq:end-u-otimes-rm}, any $T \in \End_\Gamma(X)$ is of the form $T = (a \cdot \Id_U) \otimes M$ where $M \in \R^{m \times m}$, and by bilinearity, we can rescale $M$ to get the simpler form 
\begin{equation}\label{eq:tensor-linear-map}
    T = \Id_U \otimes M.
\end{equation}
We can describe $\Gamma$-invariant quadratic forms on $X$ using the Frobenius inner product on matrices given in~\eqref{eq:frobenius-ip}. Given $x = \sum_{i=1}^m u_i \otimes e_i \in X$, we can assign a symmetric matrix $G(x) \in \s^m_+$ given by the Gram matrix of the vectors $u_1, \ldots, u_m$.
\begin{definition}
    Let $x = \sum_{i=1}^m u_i \otimes e_i$. The \textit{Gram matrix} of $x$ is the symmetric psd matrix $G(x) \in \s^m_+$ given by 
    \[G(x)_{ij} = \langle u_i, u_j \rangle_U.\]
\end{definition}
\begin{proposition}\label{prop:gram-matrix}
    Assume $U$ is absolutely irreducible, and let $Q$ be a $\Gamma$-invariant quadratic form on $X$. Letting $T$ be the associated $\Gamma$-linear map as in~\eqref{eq:tensor-linear-map}, we have
    \begin{equation}
        Q(x) = \langle x, Tx \rangle_{X} = \langle G(x), M \rangle_F
    \end{equation}
    for any $x \in X$, where $\langle \cdot, \cdot \rangle_F$ is the Frobenius inner product on $\s^m$.
\end{proposition}
\begin{proof}
    Letting $x = \sum_{i=1}^m u_i \otimes e_i$, we have 
    \begin{align*}
        Q(x) &= \langle x, Tx \rangle_{X} = \sum_{i,j=1}^m \langle u_i \otimes e_i, T(u_j \otimes e_j) \rangle_{X} \\
        &= \sum_{i,j=1}^m \langle u_i, u_j \rangle_U \cdot \langle e_i, Me_j \rangle_{\R^m} = \sum_{i,j=1}^m G(x)_{ij} \cdot M_{ij} \\
        &= \langle G(x), M \rangle_F. \qedhere
    \end{align*}
\end{proof}

Since $G(x)$ is positive semidefinite, this proposition tells us that the map $Q(x)$ can be factored through the space of $m \times m$ psd matrices: 
\[X \xrightarrow{G}  \s_+^m \xrightarrow{\langle \cdot, M \rangle_F}  \R.\]
Crucially, we know $\s_+^m$ is convex, and that the map $G: X \to \s_+^m$ is surjective if and only if $m \leq \dim U$. In this case, $Q$ is a linear functional applied to $\s_+^m$. 

We now give our main statement where $U$ can be complex or quaternionic type, but we only prove the case where $U$ is of real type. We will prove the result in full generality in \S~\ref{subsec:complex-quaternionic}. Since $D \cong \End_\Gamma(U)$ is a division ring, we can view $U$ as a $D$-vector space when $D= \R,\C$ or as a free left $D$-module when $D = \HH$. This allows us to talk about the dimension of $U$ as a $D$-vector space or module
\[ \dim_D U = \dim_\R U / \dim_\R D,\]
where $\dim_\R \C = 2$, and $\dim_\R \HH = 4$.

\begin{proposition}\label{prop:convex-image}
    Assume $U$ is an irreducible representation of $\Gamma$ with $\End_\Gamma(U) \cong D$, and suppose we have $\Gamma$-invariant quadratic forms $Q_1, \ldots, Q_s: X \to \R$. This defines the quadratic map $Q: X \to \R^s$ where
    \[Q(x) = (Q_1(x), \ldots, Q_s(x)).\]
    If the multiplicity $m$ of $U$ in $X$ satisfies $m \leq \dim_D U$, then $Q(X)$ is convex.
\end{proposition}
\begin{proof}[Proof (real case)]
    Each $Q_j$ has an associated $\Gamma$-linear map $T_j = \Id_U \otimes M_j$ where $M_j \in \R^{m \times m}$. Then by Proposition~\ref{prop:gram-matrix}, $Q_j(x) = \langle G(x), M_j \rangle_F$, and
    \begin{equation*}
     Q(x) = \paren{\langle G(x), M_1 \rangle_F, \ldots, \langle G(x), M_s \rangle_F }.
    \end{equation*}
    The map $Z \mapsto (\langle Z, M_j\rangle_F)_{j=1}^s$ is linear from $\s^m$ to $\R^s$, so $Q(x)$ is the linear image of $G(x)$ under this map. But since $m \leq \dim_\R U$, $G: X \to \s_+^m$ is surjective, so $G(X)$ is convex. Hence, $Q(X)$ is the linear image of a convex set and is itself convex.
\end{proof}

We now have a necessary condition for when $\ell_\Psi(\E_\lambda)$ is convex, assuming the general statement of Proposition~\ref{prop:convex-image}.

\begin{theorem}\label{thm:convex-image}
    Let $G$ be a graph with a subgroup of automorphisms $\Psi$, and $\lambda > 0$ be a Laplacian eigenvalue. Let
    $\mathcal{E}_\lambda = X_1 \oplus \cdots \oplus X_h$ be the isotypic
    decomposition of $\mathcal{E}_\lambda$ as a $\Psi$-representation, where
    $X_i \cong \R^{m_i} \otimes U_i$ for distinct irreducibles $U_i$
    with $D_i \cong \End_\Psi(U_i)$. If $m_i \leq \dim_{D_i} U_i$ for all $1 \leq i \leq h$, then 
    \begin{enumerate}
        \item $\ell_\Psi(\mathcal{E}_\lambda)$ is a convex cone, and
        \item $G$ has an orbit-isometric embedding on $\E_\lambda$ if and only if there exists a single $\varphi \in \mathcal{E}_\lambda$ with $\ell_\Psi(\varphi) = \Om$.
    \end{enumerate}
\end{theorem}
\begin{proof}
    Since $\ell_\Psi$ is $\Psi$-invariant and $m_i \leq \dim_{D_i} U_i$, we can apply Proposition~\ref{prop:convex-image} to conclude that each $\ell_\Psi(X_i)$ is convex. By~\eqref{eq:diagonal-quadratic}, we know
    \[\ell_\Psi(\E_\lambda) = \sum_{i=1}^h \ell_\Psi(X_i),\]
    which is the Minkowski sum of $h$ convex sets, and is thus convex.
\end{proof}

\subsection{Vertex-Transitive Graphs}\label{subsec:vertex-transitive}

Now we show that if $G$ is vertex-transitive with respect to $\Psi$, then any Laplacian eigenspace $\E_\lambda$ satisfies the assumptions of Theorem~\ref{thm:convex-image}. We first observe by Schur's lemma that if $m$ is the multiplicity of an irreducible representation $U$ in $\E_\lambda$ with $\End_\Psi(U) \cong D$, then $\Hom_\Psi(\E_\lambda, U) \cong \End_\Psi(U)^m \cong D^m$. This implies $\dim_\R(\Hom_\Psi(\E_\lambda, U)) = \dim_\R(D^m) = m \cdot \dim_\R(D)$, which gives us an upper bound on $m$:
\begin{equation}\label{eq:mult-bound}
    m = \dim_\R(\Hom_\Psi(\E_\lambda, U)) / \dim_\R D \leq \dim_\R(\Hom_\Psi(\R^V, U)) / \dim_\R D.
\end{equation}
We can bound the right-hand side using \textit{Frobenius reciprocity}. Here we use $\res_H^\Psi U$ to denote the \textit{restriction} of the $\Psi$-representation $U$ to the subgroup $H \leq \Psi$. We use $\ind_H^\Psi 1_H$ to denote the \textit{induced representation} of the trivial representation $1_H$ of $H$. For a standard treatment of these definitions over $\C$, we refer to~\cite[\S 7]{serre} or~\cite[\S 3.3]{fulton}. We cite~\cite[Chapter 3 \S 6]{brocker}, which gives the same definitions and results for arbitrary fields, including $\R$.

\begin{lemma}[Frobenius reciprocity]\label{lemma:frobenius}
    Fix $v \in V$, and let $H \leq \Psi$ be the stabilizer of $v$. Then $\Hom_\Psi(\R^V, U) \cong \Hom_H(1_H, \res_H^\Psi U)$.
\end{lemma}
\begin{proof}
    The induced representation $\ind_H^\Psi 1_H$ is $\Psi$-isomorphic to the permutation representation of the action of $\Psi$ on the cosets $\Psi/H$~\cite[Example 6.3]{brocker}. By the orbit-stabilizer theorem, $\Psi /H$ is isomorphic as a $\Psi$-set to $V$, so $\ind_H^\Psi 1_H \cong \R^V$ as $\Psi$-representations. Applying Frobenius reciprocity~\cite[Theorem 6.2]{brocker}, we get 
    \[ \Hom_\Psi(\R^V, U) \cong \Hom_\Psi(\ind_H^\Psi 1_H, U) \cong \Hom_H(1_H, \res_H^\Psi U). \qedhere\]
\end{proof}

\begin{theorem}\label{thm:vertex-transitive}
    If $G$ is vertex-transitive with respect to $\Psi \leq \Aut(G)$
    and $\lambda > 0$, then $\ell_\Psi(\mathcal{E}_\lambda)$ is a convex cone.
    In particular, $G$ is conformally rigid if and only if there exists a single
    $\varphi \in \mathcal{E}_\lambda$ with $\ell_\Psi(\varphi) = \Om$.
\end{theorem}
\begin{proof}
    By Theorem~\ref{thm:convex-image}, it suffices to show that any irreducible representation $U$ with endomorphism ring isomorphic to $D = \R, \C,$ or $\HH$ satisfies $m \leq \dim_D U$ where $m$ is the multiplicity of $U$ in $\E_\lambda$. Combining~\eqref{eq:mult-bound} and Lemma~\ref{lemma:frobenius}, we have 
    \begin{align*}
        m &\leq \dim_\R(\Hom_\Psi(\R^V, U)) / \dim_\R D = \dim_\R(\Hom_H(1_H, \res_H^\Psi U)) / \dim_\R D \\
        &\leq \dim_\R(U) / \dim_\R D = \dim_D U. \qedhere
    \end{align*}
\end{proof}

Theorem~\ref{thm:vertex-transitive} gives an affirmative answer to Question~1 in~\cite{gouveia}, providing a full converse to Corollary~4.5 there and to Theorem~2.3 in~\cite{steinerberger} in the special case where $G$ is a Cayley graph.  In practice, most non-vertex-transitive graphs with a large automorphism group still satisfy Theorem~\ref{thm:convex-image}. For example, the five non-vertex-transitive sporadic graphs from~\cite{gouveia}, including Crossing Number 6B, all satisfy this multiplicity bound with their full automorphism groups, and they can be certified by a single $\varphi$ using Theorem~\ref{thm:polyhedral-cone}.

 A natural question to ask is whether every conformally rigid graph $G$ has an automorphism group under which Theorem~\ref{thm:convex-image} holds. The answer to this question is no. This shows that even though Theorem~\ref{thm:convex-image} seems to apply broadly to many graphs, it is not a necessary condition for conformal rigidity.

\begin{example}
    Let $G$ (\textsc{HoG 56682}), pictured in Figure~\ref{fig:P25}, be the $(25,12,5,6)$-strongly regular graph P25.01 (see Brouwer's tables~\cite{brouwer}), which has $\lambda_2 = 10$ and $\lambda_n=15$ both with multiplicity 12. Strongly regular graphs are 1-walk regular and thus lower and upper conformally rigid~\cite[Theorem 3.2]{gouveia}, but $G$ has a trivial automorphism group. This means $G$ does not satisfy Theorem~\ref{thm:convex-image}. Indeed, $G$ is not bipartite, so Corollary~\ref{cor:bipartite} confirms that we need more than one eigenvector to certify the lower and upper conformal rigidity of $G$.
    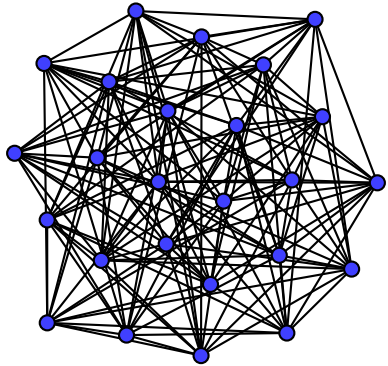
\begin{figure}[h]
        \centering
        \begin{tikzpicture}[
                scale=1,
                vertex/.style={circle, draw, fill=blue!75, inner sep=2pt, line width=0.8pt},
            ]
            \coordinate (V1) at (1.2491, 3.6238);
            \coordinate (V2) at (2.4729, 4.2160);
            \coordinate (V3) at (3.2909, 3.8463);
            \coordinate (V4) at (3.6684, 2.3241);
            \coordinate (V5) at (0.0000, 2.6767);
            \coordinate (V6) at (0.4308, 1.7946);
            \coordinate (V7) at (3.9739, 4.4483);
            \coordinate (V8) at (1.9032, 2.2992);
            \coordinate (V9) at (0.4316, 0.4334);
            \coordinate (V10) at (0.3886, 3.8651);
            \coordinate (V11) at (1.1464, 1.2599);
            \coordinate (V12) at (2.0047, 1.4769);
            \coordinate (V13) at (2.9349, 3.0441);
            \coordinate (V14) at (3.5011, 1.3271);
            \coordinate (V15) at (2.4662, 0.0000);
            \coordinate (V16) at (4.7985, 2.2860);
            \coordinate (V17) at (1.4801, 0.2716);
            \coordinate (V18) at (1.0924, 2.6175);
            \coordinate (V19) at (1.6041, 4.5585);
            \coordinate (V20) at (2.7648, 2.0437);
            \coordinate (V21) at (4.0723, 3.1604);
            \coordinate (V22) at (4.4580, 1.1447);
            \coordinate (V23) at (2.5911, 0.9417);
            \coordinate (V24) at (2.0270, 3.2372);
            \coordinate (V25) at (3.6011, 0.2986);
            \draw[thick] (V1) -- (V2);
            \draw[thick] (V1) -- (V3);
            \draw[thick] (V1) -- (V4);
            \draw[thick] (V1) -- (V5);
            \draw[thick] (V1) -- (V6);
            \draw[thick] (V1) -- (V7);
            \draw[thick] (V1) -- (V8);
            \draw[thick] (V1) -- (V9);
            \draw[thick] (V1) -- (V10);
            \draw[thick] (V1) -- (V11);
            \draw[thick] (V1) -- (V12);
            \draw[thick] (V1) -- (V13);
            \draw[thick] (V2) -- (V3);
            \draw[thick] (V2) -- (V4);
            \draw[thick] (V2) -- (V5);
            \draw[thick] (V2) -- (V6);
            \draw[thick] (V2) -- (V7);
            \draw[thick] (V2) -- (V14);
            \draw[thick] (V2) -- (V15);
            \draw[thick] (V2) -- (V16);
            \draw[thick] (V2) -- (V17);
            \draw[thick] (V2) -- (V18);
            \draw[thick] (V2) -- (V19);
            \draw[thick] (V3) -- (V4);
            \draw[thick] (V3) -- (V5);
            \draw[thick] (V3) -- (V6);
            \draw[thick] (V3) -- (V7);
            \draw[thick] (V3) -- (V20);
            \draw[thick] (V3) -- (V21);
            \draw[thick] (V3) -- (V22);
            \draw[thick] (V3) -- (V23);
            \draw[thick] (V3) -- (V24);
            \draw[thick] (V3) -- (V25);
            \draw[thick] (V4) -- (V8);
            \draw[thick] (V4) -- (V9);
            \draw[thick] (V4) -- (V10);
            \draw[thick] (V4) -- (V14);
            \draw[thick] (V4) -- (V15);
            \draw[thick] (V4) -- (V16);
            \draw[thick] (V4) -- (V20);
            \draw[thick] (V4) -- (V21);
            \draw[thick] (V4) -- (V22);
            \draw[thick] (V5) -- (V8);
            \draw[thick] (V5) -- (V11);
            \draw[thick] (V5) -- (V12);
            \draw[thick] (V5) -- (V14);
            \draw[thick] (V5) -- (V17);
            \draw[thick] (V5) -- (V18);
            \draw[thick] (V5) -- (V20);
            \draw[thick] (V5) -- (V23);
            \draw[thick] (V5) -- (V24);
            \draw[thick] (V6) -- (V9);
            \draw[thick] (V6) -- (V11);
            \draw[thick] (V6) -- (V13);
            \draw[thick] (V6) -- (V15);
            \draw[thick] (V6) -- (V17);
            \draw[thick] (V6) -- (V19);
            \draw[thick] (V6) -- (V21);
            \draw[thick] (V6) -- (V23);
            \draw[thick] (V6) -- (V25);
            \draw[thick] (V7) -- (V10);
            \draw[thick] (V7) -- (V12);
            \draw[thick] (V7) -- (V13);
            \draw[thick] (V7) -- (V16);
            \draw[thick] (V7) -- (V18);
            \draw[thick] (V7) -- (V19);
            \draw[thick] (V7) -- (V22);
            \draw[thick] (V7) -- (V24);
            \draw[thick] (V7) -- (V25);
            \draw[thick] (V8) -- (V9);
            \draw[thick] (V8) -- (V10);
            \draw[thick] (V8) -- (V11);
            \draw[thick] (V8) -- (V14);
            \draw[thick] (V8) -- (V16);
            \draw[thick] (V8) -- (V19);
            \draw[thick] (V8) -- (V23);
            \draw[thick] (V8) -- (V24);
            \draw[thick] (V8) -- (V25);
            \draw[thick] (V9) -- (V10);
            \draw[thick] (V9) -- (V12);
            \draw[thick] (V9) -- (V15);
            \draw[thick] (V9) -- (V17);
            \draw[thick] (V9) -- (V18);
            \draw[thick] (V9) -- (V22);
            \draw[thick] (V9) -- (V23);
            \draw[thick] (V9) -- (V25);
            \draw[thick] (V10) -- (V13);
            \draw[thick] (V10) -- (V17);
            \draw[thick] (V10) -- (V18);
            \draw[thick] (V10) -- (V19);
            \draw[thick] (V10) -- (V20);
            \draw[thick] (V10) -- (V21);
            \draw[thick] (V10) -- (V24);
            \draw[thick] (V11) -- (V12);
            \draw[thick] (V11) -- (V13);
            \draw[thick] (V11) -- (V14);
            \draw[thick] (V11) -- (V15);
            \draw[thick] (V11) -- (V19);
            \draw[thick] (V11) -- (V21);
            \draw[thick] (V11) -- (V22);
            \draw[thick] (V11) -- (V24);
            \draw[thick] (V12) -- (V13);
            \draw[thick] (V12) -- (V15);
            \draw[thick] (V12) -- (V16);
            \draw[thick] (V12) -- (V18);
            \draw[thick] (V12) -- (V20);
            \draw[thick] (V12) -- (V22);
            \draw[thick] (V12) -- (V23);
            \draw[thick] (V13) -- (V14);
            \draw[thick] (V13) -- (V16);
            \draw[thick] (V13) -- (V17);
            \draw[thick] (V13) -- (V20);
            \draw[thick] (V13) -- (V21);
            \draw[thick] (V13) -- (V25);
            \draw[thick] (V14) -- (V16);
            \draw[thick] (V14) -- (V17);
            \draw[thick] (V14) -- (V18);
            \draw[thick] (V14) -- (V21);
            \draw[thick] (V14) -- (V22);
            \draw[thick] (V14) -- (V25);
            \draw[thick] (V15) -- (V16);
            \draw[thick] (V15) -- (V17);
            \draw[thick] (V15) -- (V19);
            \draw[thick] (V15) -- (V20);
            \draw[thick] (V15) -- (V22);
            \draw[thick] (V15) -- (V24);
            \draw[thick] (V16) -- (V19);
            \draw[thick] (V16) -- (V20);
            \draw[thick] (V16) -- (V23);
            \draw[thick] (V16) -- (V25);
            \draw[thick] (V17) -- (V18);
            \draw[thick] (V17) -- (V20);
            \draw[thick] (V17) -- (V24);
            \draw[thick] (V17) -- (V25);
            \draw[thick] (V18) -- (V19);
            \draw[thick] (V18) -- (V21);
            \draw[thick] (V18) -- (V22);
            \draw[thick] (V18) -- (V23);
            \draw[thick] (V19) -- (V21);
            \draw[thick] (V19) -- (V23);
            \draw[thick] (V19) -- (V24);
            \draw[thick] (V20) -- (V21);
            \draw[thick] (V20) -- (V23);
            \draw[thick] (V20) -- (V24);
            \draw[thick] (V21) -- (V22);
            \draw[thick] (V21) -- (V23);
            \draw[thick] (V22) -- (V24);
            \draw[thick] (V22) -- (V25);
            \draw[thick] (V23) -- (V25);
            \draw[thick] (V24) -- (V25);
            \foreach \i in {1,...,25} {
                \node[vertex] at (V\i) {};
            }
        \end{tikzpicture}
        \caption{The graph P25.01 (\textsc{HoG 56682}).}\label{fig:P25}
    \end{figure}
\end{example}
The previous example illustrates that some conformally rigid graphs are degenerate in the sense that their relevant eigenspace $\E_\lambda$ is so large that they are still conformally rigid even in the absence of a large automorphism group. Intuitively, conformal rigidity is certified by finding a weighted collection of eigenvectors in $\E_\lambda$ whose total edge-energy is constant; a larger
eigenspace provides more candidate eigenvectors, making this easier to satisfy independently of any automorphism structure. Conversely, we do not know of a conformally rigid graph that admits a 1-dimensional orbit-isometric embedding but fails to satisfy Theorem~\ref{thm:convex-image} for any subgroup $\Psi$, though we see no theoretical obstruction to such an example existing.

Additionally, there were nine graphs in \textsc{HoG} that were numerically certified to be conformally rigid, but had at least one isotypic where $m > \dim_D U$. The nine graphs have \textsc{HoG} IDs \textsc{730, 1290, 21234, 21609, 50485, 51392, 51393, 51453,} and \textsc{51455}; each has a nontrivial automorphism group but still fails the multiplicity bound of Theorem~\ref{thm:convex-image}. These graphs together with the previous example demonstrate that our framework of symmetry reduction with respect to the automorphism group of a graph $G$ does not explain all conformally rigid graphs. While the previous example is explained by existing theory (1-walk regularity), we were unable to find an exact certificate for the nine \textsc{HoG} graphs. We wonder if there is some underlying combinatorial property that can unify our symmetry reduction techniques with conformally rigid graphs that have little-to-no automorphisms.

\subsection{Proposition~\ref{prop:convex-image}: Complex and Quaternionic Cases}\label{subsec:complex-quaternionic}

Now we prove the general case of Proposition~\ref{prop:convex-image}. We are back in the setting where we fix a real irreducible representation $U$ of $\Gamma$ and $X = U \otimes \R^m$ is an isotypic representation. Now $U$ is of complex or quaternionic type where $\End_\Gamma(U) \cong \C$ or $\HH$ respectively. In the case where $U$ is complex, there exists a map $J \in \End_\Gamma(U)$ such that $J^2 = -\Id_U$.
This turns $U$ into a $\C$-vector space where $i \cdot u = Ju$ for all $u \in U$. Similarly, if $U$ is quaternionic, there exists $J_1, J_2, J_3 \in \End_\Gamma(U)$ satisfying $J_i^2 = -\Id_U$ for each $i=1,2,3$. This allows us to view  $U$ as a left $\HH$-module, satisfying
\[ i \cdot u = J_1u, \ j \cdot u = J_2 u, \ k \cdot u = J_3 u\]
where $i,j,k$ are the standard imaginary basis elements of $\HH$. Notationally, it is unfortunate that these basis elements overlap with the standard indexing variables in linear algebra. For the remainder of this section, we overload these variables, but hopefully the context on whether the variable is being used as an index or unit quaternion should prevent confusion.

We can introduce \textit{complex (Hermitian)} and \textit{quaternionic inner products} on $U$ and similarly construct Gram matrices with entries in $\C$ and $\HH$. Just as we can talk about positive semidefiniteness for real symmetric matrices, we can talk about \textit{psd Hermitian matrices} for $\C$ and $\HH$. Just as in the real case, we will see that the set of all such matrices is convex. Moreover, under the appropriate dimension assumptions, any such psd matrix arises as a Gram matrix with the corresponding complex or quaternionic inner product.

First, we describe what maps $T$ in $\End_\Gamma(X) \cong \End_\Gamma(U) \otimes \R^{m \times m}$ look like. In the complex case, $\Id_U$ and $J$ form a real basis of $\End_\Gamma(U)$, so any $T \in \End_\Gamma(X)$ is of the form 
\begin{equation}\label{eq:tensor-complex}
    T = \Id_U \otimes M^{(0)} + J \otimes M^{(1)}    
\end{equation}
where $M^{(i)} \in \R^{m \times m}$. Similarly for the quaternionic case, $\Id_U, J_1, J_2,$ and $J_3$ form a basis for $\End_\Gamma(U)$, so any $T\in \End_\Gamma(X)$ is of the form 
\begin{equation}\label{eq:tensor-quaternionic}
    T = \Id_U \otimes M^{(0)} + J_1 \otimes M^{(1)} + J_2 \otimes M^{(2)} + J_3 \otimes M^{(3)}
\end{equation}
with each $M^{(i)} \in \R^{m \times m}$.

Now that we have a nice description of maps $T \in \End_\Gamma(X)$, we want to prove the rest of Proposition~\ref{prop:convex-image}. To do this, we need to define appropriate complex and quaternionic inner products on $U$. The discussion of complex inner products, Hermitian matrices, psd matrices, and Gram matrices are all standard and can be read about in~\cite{horn}. The quaternionic versions of these objects require a bit more care due to the noncommutativity of $\HH$, but all the relevant results still hold. For quaternionic linear algebra, we refer to~\cite{farenick, rodman}. We define the following scalar products on $U$ when it is of complex or quaternionic type.

\begin{definition}\label{def:ips}
    Let $\langle \cdot, \cdot \rangle_U$ be a $\Gamma$-invariant real inner product on $U$. If $U$ is of complex type, we define the \textit{complex inner product} on $U$ given by
    \begin{equation}\label{eq:complex-ip}
      \langle u, u' \rangle_\C = \langle u , u' \rangle_U + i \langle u, J u' \rangle_U.  
    \end{equation} 
    If $U$ is of quaternionic type, we define the \textit{quaternionic inner product} on $U$ to be 
    \begin{equation}\label{eq:quaternionic-ip}
        \langle u, u' \rangle_\HH = \langle u , u' \rangle_U + i \langle u, J_1 u' \rangle_U +j \langle u, J_2 u' \rangle_U+k \langle u, J_3 u' \rangle_U.
    \end{equation}
\end{definition}

We want to verify that~\eqref{eq:complex-ip} and~\eqref{eq:quaternionic-ip} are bona fide complex and quaternionic inner products. The definition of a complex inner product is well-known: we need to check that~\eqref{eq:complex-ip} as defined above is conjugate symmetric, $\C$-linear in the first argument, and positive definite. The definition for a quaternionic inner product is not too different; for~\eqref{eq:quaternionic-ip} to be a quaternionic inner product, we also need conjugate symmetry, \textit{left} $\HH$-linearity in the first argument, and positive definiteness. The main difference from complex inner products is the left-linearity of the first argument. Since $\C$ is commutative, conjugate symmetry and $\C$-linearity in the first argument imply conjugate linearity in the second argument:
\[\langle u,z \cdot u' \rangle_\C = \overline{z \cdot \langle u', u\rangle}_\C = \bar{z} \cdot \overline{\langle u', u \rangle}_\C = \bar{z} \cdot \langle u, u' \rangle_\C\]
for any $z \in \C$. In $\HH$, we need to be careful about our order of multiplication. Conjugate symmetry and left $\HH$-linearity in the first argument gives us the following:
\begin{equation}
    \langle u, q \cdot u' \rangle_\HH = \overline{q \cdot \langle u', u \rangle}_\HH = \overline{\langle u', u\rangle}_\HH \cdot \bar{q} = \langle u, u' \rangle_\HH \cdot \bar{q}
\end{equation}
for all $q \in \HH$. So a quaternionic inner product is \textit{right} conjugate linear in the second argument, which comes from the noncommutativity of $\HH$.

\begin{remark}
    Our definition of the quaternionic inner product varies slightly from the definitions in~\cite{farenick,rodman}. In these references, they work with right $\HH$-modules, and their inner products satisfy right $\HH$-linearity in the second argument. We could have just as easily defined $U$ to be a right $\HH$-module under the conjugate action 
    \[ u \cdot i = - J_1u, \ u \cdot j = - J_2 u, \ u \cdot k = - J_3 u,\]
    but since our $\Gamma$-endomorphisms act on the left, we chose to turn $U$ into a left $\HH$-module. Our quaternionic structure and inner products only differ by conjugation, so all the necessary results on convexity and positive semidefiniteness still follow.
\end{remark}

With these definitions in hand, we can now prove~\eqref{eq:complex-ip} and~\eqref{eq:quaternionic-ip} give the appropriate inner products. First we prove a useful lemma.

\begin{lemma}\label{lemma:skew-adjoint}
    Suppose $T \in \End_\Gamma(U)$ where $T^2 = -\Id_U$. Then $T$ is an isometry and skew-adjoint with respect to any $\Gamma$-invariant, real inner product $\langle \cdot, \cdot \rangle_U$.
\end{lemma}
\begin{proof}
    We know $T$ is invertible, so the map $(u,u') \mapsto \langle Tu, Tu' \rangle_U$ is a positive definite, symmetric bilinear form. It is also $\Gamma$-invariant as both $T$ and the inner product are $\Gamma$-invariant. Thus, there exists $S \in \End_\Gamma(U)$ that is self-adjoint satisfying 
    \[\langle Tu, Tu' \rangle_U = \langle u, Su' \rangle_U \]
    for all $u,u' \in U$. By the spectral theorem and positive-definiteness, $S$ has a real eigenvalue $c > 0$. This means $S - c \cdot \Id_U \in \End_\Gamma(U)$ has nontrivial kernel. As kernels of $\Gamma$-endomorphisms are $\Gamma$-stable and $U$ is irreducible, $S = c \cdot \Id_U$. Hence,
    \[\langle Tu , Tu' \rangle_U = c \cdot \langle u, u' \rangle_U.\]
    Combining this with the fact that $T^2 = -\Id_U$, we get
    \[c^2 \cdot \langle u, u' \rangle_U = \langle T^2u , T^2u' \rangle_U = \langle -u, -u' \rangle_U = \langle u, u' \rangle_U,\]
    whence $c=1$. Thus, $T$ is an isometry. Moreover for any $u,u' \in U$,
    \[ \langle u, Tu' \rangle_U = \langle Tu, T^2 u' \rangle_U = - \langle Tu, u' \rangle_U,\]
    so $T$ is also skew-adjoint.
\end{proof}

\begin{proposition}
    The constructions in Definition~\ref{def:ips} give $\Gamma$-invariant complex and quaternionic inner products on $U$.
\end{proposition}
\begin{proof}
    It is clear the inner products are $\Gamma$-invariant, so we just have to show they are complex and quaternionic inner products. We start with the familiar complex inner product. Applying Lemma~\ref{lemma:skew-adjoint} several times, we can check
    \begin{itemize}
        \item \textit{Conjugate-symmetry:} For all $u, u' \in U$, we have \begin{align*}
            \langle u, u' \rangle_\C = \langle u, u' \rangle_U + i \langle u, Ju' \rangle_U = \langle u', u \rangle_U - i \langle u', Ju \rangle_U = \overline{\langle u', u \rangle}_\C.
        \end{align*}
        \item \textit{Complex-linearity in the first argument:} It is clear that $\langle \cdot, u' \rangle_\C$ is $\R$-linear in the first argument so we just need to show $\langle iu, u' \rangle_\C = i \langle u, u' \rangle_\C$ for all $u, u' \in U$:
        \begin{align*}
            \langle iu, u' \rangle_\C &= \langle Ju, u' \rangle_U + i \langle Ju, Ju' \rangle_U = - \langle u, J u' \rangle_U + i \langle u, u' \rangle_U \\
            &= i \cdot (\langle u, u' \rangle_U + i \langle u, Ju' \rangle_U) = i \langle u, u' \rangle_\C.
        \end{align*}
        \item \textit{Positive-definiteness:} Given $u \in U$, we have 
        \[\langle u, u \rangle_\C = \langle u, u \rangle_U + i \langle u, Ju \rangle_U = \langle u, u \rangle_U \geq 0,\]
        with equality if and only if $u=0$.
    \end{itemize}
    If $U$ is of quaternionic type, conjugate-symmetry and positive-definiteness follow identically by applying skew-symmetry of $J_i$ where $i = 1,2,3$. Left $\HH$-linearity in the first argument takes some additional steps. We see 
    \begin{align*}
      \langle iu, u' \rangle_\HH &= \langle J_1 u, u' \rangle_U + i\langle J_1 u, J_1 u' \rangle_U + j \langle J_1 u, J_2 u' \rangle_U + k \langle J_1 u, J_3 u' \rangle_U \\
      &= -\langle  u, J_1 u' \rangle_U + i\langle u, u' \rangle_U - j \langle u, J_1J_2 u' \rangle_U - k \langle u, J_1J_3 u' \rangle_U\\
      &= -\langle  u, J_1 u' \rangle_U + i\langle u, u' \rangle_U - j \langle u, J_3 u' \rangle_U + k \langle u, J_2 u' \rangle_U \\ 
      &= i (\langle u, u'\rangle_U + i\langle  u, J_1 u' \rangle_U + j \langle u, J_2 u' \rangle_U + k \langle u, J_3 u' \rangle_U) \\
      &= i \langle u, u' \rangle_\HH.
    \end{align*}
    The proof for $j$ and $k$ follow similarly, so $\langle \cdot, \cdot \rangle_\HH$ is left $\HH$-linear in the first argument.
\end{proof}

Now we define Hermitian and psd matrices for $\C$ and $\HH$. Recall the adjoint of a complex matrix $A \in \C^{m \times m}$ is given by the conjugate transpose $A^* = \bar{A}^T.$
We can similarly define the adjoint $A^*$ of a quaternionic matrix $A \in \HH^{m \times m}$ to be the conjugate transpose.
\begin{definition}
    A matrix $A \in D^{m \times m}$ where $D = \C, \HH$ is \textit{Hermitian} if it is equal to its conjugate transpose i.e. $A = A^* = \bar{A}^T$. A Hermitian matrix $A$ satisfies $\xi^* A \xi \in \R$ for all $\xi \in D^m$, and is \textit{positive semidefinite} if $\xi^* A \xi \geq 0$ for all $\xi \in D^m$.
\end{definition}

We denote the space of all $m \times m$ positive semidefinite matrices over $D = \C, \HH$ as $\s_+^m(D)$. As in the real case, $\s_+^m(D)$ is convex: given $A,B \in \s_+^m(D), \ \xi \in D$, and $t \in [0,1]$, we have
\[\xi^*(t A + (1-t)B)\xi = t \xi^* A \xi + (1-t) \xi^* B \xi \geq 0.\]
Gram matrices given by inner products over $D$ always produce a psd matrix.

\begin{definition}
    Let $x = \sum_{i=1}^m u_i \otimes e_i \in X$. Then the Gram matrix $G_D(x) \in D^{m \times m}$ for $D = \C, \HH$, is defined by $G_D(x)_{ij} = \langle u_i, u_j \rangle_D$.
\end{definition}
It is well-known in the case $D = \C$ that $G_\C(x) \in \s_+^m(\C)$ and the map $x \mapsto G_\C(x)$ is surjective onto $\s_+^m(\C)$ if $m \leq \dim_\C U$. This is also true when $D = \HH$, but we will need to prove this.

\begin{lemma}
    Suppose $U$ is of quaternionic type. Then $G_\HH(x) \in \s^m_+(\HH)$ for any $x \in U$. Moreover, if $m \leq \dim_\HH U$, the map $G_\HH: X \to \s^m_+(\HH)$ is surjective.
\end{lemma}
\begin{proof}
    Given $x = \sum_{i=1}^m u_i \otimes e_i \in X$ and $\xi \in \HH^m$, we have
    \begin{align*}
        \xi^* G_\HH(x) \xi = \sum_{i, j=1}^m \bar{\xi}_i \cdot \langle u_i, u_j \rangle_\HH \cdot \xi_j = \left \langle \sum_{i=1}^m \bar{\xi}_i u_i, \sum_{j=1}^m \bar{\xi}_j u_j \right \rangle_\HH \geq 0
    \end{align*}
    since $\langle \cdot, \cdot \rangle_\HH$ is positive definite. Thus $G_\HH(x) \in \s_+^m(\HH)$.

    Now suppose that $m \leq \dim_\HH U$. This means that there exists an orthonormal set $v_1, \ldots v_m \in U$, that is, $\langle v_i, v_j \rangle_\HH = 1 \text{ if } i = j$ and $0$ otherwise. This follows from Gram-Schmidt applied to quaternionic inner products~\cite[Theorem 4.3]{farenick}. Given any $A \in \s_+^m(\HH)$, we can write $A = B^*B$ for some $B \in \HH^{m \times m}$ using Cholesky factorization~\cite[Proposition 3.2.7]{rodman}. Let $x = \sum_{i=1}^m u_i \otimes e_i$ where
    \[u_i = \sum_{k=1}^m (B^*)_{ik} v_k = \sum_{k=1}^m \overline{B}_{ki} v_k.\]
    Now we check
    \begin{align*}
        G_\HH(x)_{ij} &= \langle u_i, u_j \rangle_\HH = \left \langle \sum_{k=1}^m \overline{B}_{ki} v_k, \sum_{l=1}^m \overline{B}_{lj} v_l \right \rangle_\HH \\
        &= \sum_{k, l=1}^m \overline{B}_{ki} \cdot \langle v_k, v_l \rangle_\HH \cdot B_{lj} = \sum_{k=1}^m \overline{B}_{ki} \cdot B_{kj} \\
        &= \sum_{k=1}^m (B^*)_{ik} \cdot B_{kj} = A_{ij}.
    \end{align*}
    Hence $G_\HH(x) = A$, so $G_\HH$ is surjective.
\end{proof}

\begin{proposition}\label{prop:gram-matrix-hermitian}
    Let $U$ be a real irreducible representation of complex or quaternionic type, and let $Q$ be a $\Gamma$-invariant quadratic form on $X$. Letting $T$ be the associated $\Gamma$-linear map, there is some $M \in D^{m \times m}$ such that for any $x \in X$
    \[Q(x) = \langle x, T x \rangle_X = \re \Tr(G_D(x)^*M)\]
    where $\langle \cdot, \cdot \rangle_X$ is the $\Gamma$-invariant product in Proposition~\ref{prop:iso-inner-prod}. In particular, $Q(x)$ is the composition of $G_D(x)$ with a real linear functional.
\end{proposition}
\begin{proof}
    Let $x = \sum_{i=1}^m u_i \otimes e_i$. First assume $D= \C$. Writing $T$ as in~\eqref{eq:tensor-complex}, let 
    $M = M^{(0)} + i M^{(1)} \in \C^{m \times m}$. Then we have
    \begin{align*}
        Q(x) &= \langle x, Tx \rangle_X = \sum_{i,j=1}^m \langle u_i \otimes e_i, T(u_j \otimes e_j) \rangle_X \\
        &= \sum_{i,j=1}^m \langle u_i, u_j \rangle_U \cdot M^{(0)}_{ij} + \langle u_i, J u_j \rangle_U \cdot M^{(1)}_{ij} \\
        &= \sum_{i,j=1}^m \re \paren{\overline{\langle u_i, u_j \rangle}_\C \cdot M_{ij}} = \re \paren{\sum_{i,j=1}^m (G_\C(x)^*)_{ji} \cdot M_{ij}} \\
        &= \re \Tr(G_\C(x)^* M).
    \end{align*}
    Now suppose $D = \HH$, and write $T$ as in~\eqref{eq:tensor-quaternionic}. If we let
    \[M = M^{(0)} + i M^{(1)} + jM^{(2)} + kM^{(3)} \in \HH^{m \times m},\]
    we get that $Q(x)$ is equal to 
    \begin{align*}
        &\sum_{i,j=1}^m \langle u_i, u_j \rangle_U \cdot M^{(0)}_{ij} + \langle u_i, J_1 u_j \rangle_U \cdot M^{(1)}_{ij} + \langle u_i, J_2 u_j \rangle_U \cdot M^{(2)}_{ij} + \langle u_i, J_3 u_j \rangle_U \cdot M^{(3)}_{ij} \\
        = & \sum_{i,j=1}^m \re \paren{\overline{\langle u_i, u_j \rangle}_\HH \cdot M_{ij}} = \re \paren{\sum_{i,j=1}^m (G_\HH(x)^*)_{ji} \cdot M_{ij}} = \re \Tr(G_\HH(x)^* M). \qedhere
    \end{align*}
\end{proof}

After all this technical setup, we can now prove the complex and quaternionic versions of Proposition~\ref{prop:convex-image}.

\begin{proof}[Proof of Proposition~\ref{prop:convex-image} (complex and quaternionic cases)] 
    Each $Q_k$ has an associated $\Gamma$-linear map $T_k$ of the form~\eqref{eq:tensor-complex} or~\eqref{eq:tensor-quaternionic} when $D= \C$ or $\HH$ respectively. By Proposition~\ref{prop:gram-matrix-hermitian} each $T_k$ has a corresponding matrix $M_k \in D^{m \times m}$ such that 
    \begin{align*}
        Q_k(x) &= \re \Tr (G_D(x)^*M_k), \text{ so}\\
        Q(x) &= \re (\Tr (G_D(x)^*M_1), \ldots, \Tr (G_D(x)^*M_s)).
    \end{align*}
    As in the real case, this is a real linear map applied after $G_D(x)$. But since $m \leq \dim_D U$, the map $G_D: X \to \s_+^m(D)$ is surjective. Thus $Q(X)$ is the linear image of the convex set $\s_+^m(D)$ and is itself convex.
\end{proof}

\section{Symmetry Adapted Bases and Block-Diagonalization}\label{sec:polyhedral-cone}
The computational speedup given in \S~\ref{sec:sym} is only half the picture of symmetry reduction. We were able to reduce the number of constraints from $m =\abs{E}$ to the number of edge orbits $s$, but we did not leverage any symmetry on the search space $\E_\lambda$, which has the structure of a $\Psi$-representation. We have seen how we can block-diagonalize an invariant quadratic form, such as $\ell_\Psi$, using a basis for the isotypics in~\eqref{eq:diagonal-quadratic}. Furthermore, if we choose each isotypic basis to arise from the identification in Proposition~\ref{prop:iso-inner-prod}, we get a \textit{symmetry adapted basis}. This allows us to simplify the search space for our certificates. 

For this section, we let $\E_{\lambda} = X_1 \oplus \ldots \oplus X_h$ be the isotypic decomposition, where $X_j$ corresponds to the $\Psi$-irreducible representation $U_j$ with dimension $d_j$ and multiplicity $m_j$. Thus, $X_j$ is of dimension $d_j \cdot m_j$ and $\sum_{j=1}^h d_j \cdot m_j = d = \dim \E_\lambda$.

\subsection{Block-Diagonal SDP}\label{subsec:block-diagonal}
If we choose a symmetry adapted basis to get the parametrization $B$ in~\eqref{eq:sym-sdp-param}, the orbit Laplacians become block diagonal of the form
\begin{equation}\label{eq:Li-block}
    \tilde{L}^i = \begin{pmatrix}
        \tilde{L}^i \vert_{X_1} & 0 & \cdots & 0 \\
        0 & \tilde{L}^i \vert_{X_2} & \cdots & 0 \\
        \vdots & \vdots & \ddots & \vdots \\
        0 & 0 & \cdots & \tilde{L}^i \vert_{X_h}
    \end{pmatrix}
\end{equation}
since each $\tilde{L}^i \in \End_\Psi(\E_\lambda)$.

\begin{proposition}
    Using a symmetry adapted basis, the SDP~\eqref{eq:sym-sdp-param} reduces to finding $Z \in \s_+^d$ of the form
    \[Z = \begin{pmatrix}
        Z_1 & 0 & \cdots & 0 \\
        0 & Z_2 & \cdots & 0 \\
        \vdots & \vdots & \ddots & \vdots \\
        0 & 0 & \cdots & Z_h
    \end{pmatrix}\]
    with each $Z_j \in \s_+^{d_jm_j}$, such that for all $1 \leq i \leq s$ 
    \[\langle \tilde{L}^i, Z \rangle_F = \sum_{j=1}^h \langle \tilde{L}^i \vert_{X_j}, Z_j \rangle_F = \abs{\Om_i}\]
    where the $\tilde{L}^i$ are block-diagonal as in~\eqref{eq:Li-block}.
\end{proposition}

Depending on the type of the irreducible representation $U_j$, one can block-diagonalize the $\tilde{L}^i \vert_{X_j}$ further into $d_j$ equal subblocks of size $m_j$. For a detailed overview of symmetry reduction in SDPs, we refer to~\cite{deKlerk, gatermann, vallentin}. In practice, this block-diagonalization can greatly speed up the runtime of the SDP, but in our cases, the dimension $d$ of $\E_\lambda$ is already small. Even in Example~\ref{ex:8000}, $d=10$, which is already trivial for an SDP solver. The real advantage of this block-diagonalization in our setting comes when the multiplicity $m_j = 1$ for each irreducible representation.

\subsection{Polyhedral cone and linear feasibility} In this section, we assume that $m_j=1$ for all $1 \leq j \leq h$. Recall that if $\varphi = \sum_{j=1}^h \varphi_j$ with $\varphi_j \in X_j$, we have $\ell_\Psi(\varphi) = \sum_{j=1}^h \ell_\Psi(\varphi_j)$ by our block-diagonalization. We now show when $m_j=1$, $\ell_\Psi \vert_{X_j}$ is unique up to scaling.

\begin{proposition}
    Let $X$ be an isotypic component of $\E_\lambda$ with multiplicity 1. Then for all $\varphi \in X$, $\ell_\Psi(\varphi)$ only depends on the norm of $\varphi$.
\end{proposition}
\begin{proof}
    Let $\tilde{X} \leq \R^d$ be the isotypic corresponding to $X$ under the parametrization $B$. Then we have $\varphi = Bx$ for some $x \in \tilde{X}$, so $\ell_\Psi(\varphi) = (\langle x, \tilde{L}^i \vert_{X} x \rangle)_{i=1}^s$. By Schur's lemma, we know that $\tilde{L}^i \vert_{X}$ is of the form $a_i \cdot I_{\dim X}$ with $a_i \in \R$ if $X$ is of real type. Then $\langle x, \tilde{L}^i \vert_X x \rangle = a_i \|x\|^2 = a_i \|\varphi\|^2$, so
    \[\ell_\Psi(\varphi) = \paren{a_i\| \varphi\|^2}_{i=1}^s = \|\varphi\|^2 \cdot (a_1, \ldots, a_s)^T.\]
    If $X$ is of complex type then $\tilde{L}^i \vert_X = a_i I_{\dim X} + b_i J$ where $J$ is the skew-symmetric element satisfying $J^2 = -I_{\dim X}$. Since $\langle x, J x \rangle = 0$, we have $\langle x, \tilde{L}^i \vert_X x \rangle = a_i \|x\|^2 = a_i \|\varphi\|^2$ as before. The same is true if $X$ is of quaternionic type as the three generators $J_1, J_2, J_3$ are all skew-symmetric.
\end{proof}

\begin{theorem}\label{thm:polyhedral-cone}
    Suppose each isotypic $X_j$ of $\E_\lambda$ has multiplicity 1, and fix a nonzero $\varphi_j$ in each $X_j$. Then $\ell_\Psi(\E_\lambda)$ is a polyhedral cone given by
    \begin{equation}\label{eq:polyhedral-cone}
        \ell_\Psi(\E_\lambda) = \Cone\set{\ell_\Psi(\varphi_j) : 1 \leq j \leq h}.
    \end{equation}
\end{theorem}
\begin{proof}
    We know $\E_\lambda$ satisfies Theorem~\ref{thm:convex-image}, so 
    \[\Cone\set{\ell_\Psi(\varphi_j) : 1 \leq j \leq h} \subseteq \Cone\set{\ell_\Psi(\varphi) : \varphi \in \E_\lambda} = \ell_\Psi(\E_\lambda).\]
    To show the other inclusion, we observe that any $\varphi \in \E_\lambda$ can be written as $\varphi = \sum_{j=1}^h \psi_j$ with $\psi_j \in X_j$. Since $\ell_\Psi$ is unique up to scaling on each $X_j$, we know there exists $a_j \geq 0$ such that $\ell_\Psi(\psi_j) = a_j \ell_\Psi(\varphi_j)$ for all $1 \leq j \leq h$. Thus, by the block-diagonalization of $\ell_\Psi$, we know \[\ell_\Psi(\varphi) = \sum_{j=1}^h \ell_\Psi(\psi_j) = \sum_{j=1}^h a_j \ell_\Psi(\varphi_j) \in \Cone\set{\ell_\Psi(\varphi_j) : 1 \leq j \leq h}. \qedhere\]
\end{proof}

One class of graphs where this theorem always applies is abelian Cayley graphs. The following corollary gives the real eigenvector analogue of Theorem 5.4 in~\cite{gouveia}.

\begin{theorem}\label{thm:abelian-cayley}
    Let $G = \Cay(\Gamma, S)$ be a Cayley graph where $\Gamma$ is abelian. With respect to the group of automorphisms induced by $\Gamma$, the eigenspace $\E_\lambda$ decomposes into isotypics of multiplicity 1. Thus, $\ell_\Gamma(\E_\lambda)$ is a polyhedral cone for any $\lambda > 0$.
\end{theorem}
\begin{proof}
    The action of $\Gamma$ on $\R^\Gamma$ is the regular representation. Thus $\R^\Gamma$ decomposes into the $\abs{\Gamma}$ complex characters. Thus every real irreducible representation of $\R^\Gamma$ (and hence $\E_\lambda$) is distinct and equal to a 1-dimensional real irreducible of real type or a 2-dimensional irreducible of complex type.
\end{proof}

This polyhedral cone structure of $\ell_\Psi(\E_\lambda)$ reduces the conformal rigidity check to a linear feasibility check in $h$ variables and $s$ constraints.

\begin{corollary}
    If we pick eigenvectors $\varphi_1, \ldots, \varphi_h \in \E_\lambda$ satisfying~\eqref{eq:polyhedral-cone}, then $G$ is conformally rigid if and only if there exists $a_1, \ldots, a_h \geq 0$ with $\sum_{j=1}^h a_j \ell_\Psi(\varphi_j) = \Om$.
\end{corollary}

For our numerical check, we now only have to check feasibility of a linear program instead of an SDP. \textsc{SageMath} allows for polyhedron construction over the algebraic reals and can check if it is empty or not --- if nonempty, \textsc{SageMath} can return an algebraic point inside the polyhedron.

\begin{rexample}[continued]
    For $G = \Cay(\Z_{21}, \set{1,6})$, both $\E_{\lambda_2}$ and $\E_{\lambda_n}$ are comprised of one irreducible representation with multiplicity $1$, so $\ell_{\Z_{21}}$ is constant (up to scaling) on $\E_{\lambda_2}$ and on $\E_{\lambda_n}$. Thus, any eigenvector in $\E_{\lambda}$ for $\lambda = \lambda_2, \lambda_n$  gives a 1-dimensional orbit-isometric embedding.
\end{rexample}

    \begin{figure}[h]
        \centering
        \begin{tikzpicture}[scale=0.3]
            \fill[blue!10] (0,0) -- (22,0) -- (22,15.5) -- (7.75,15.5) -- cycle;

            \draw[->] (0,0) -- (22,0) node[below, midway, yshift=-10pt] {$\Om_2$ energy};
            \draw[->] (0,0) -- (0,15.5) node[rotate=90, midway, yshift=18pt] {$\Om_3$ energy};

            \foreach \x in {6,12,18} {
                \draw (\x,0.15) -- (\x,-0.15) node[below] {\small $\x$};
            }
            \foreach \y in {6,12} {
                \draw (0.15,\y) -- (-0.15,\y) node[left] {\small $\y$};
            }

            \draw[dashed, red, thick] (0,0) -- (15.5,15.5);

            \draw[->, blue, very thick] (0,0) -- (6,12);
            \draw[->, blue, very thick] (0,0) -- (6,0);

            \node[blue, left] at (5.8,12.5) {\small $\ell_{\Z_{12}}(\varphi_1^1)$};

            \node[blue, below] at (6,2) {\small $\tfrac{1}{3}\ell_{\Z_{12}}(\varphi_1^4)$};

            \draw[blue, dashed, thick] (6,12) -- (12,12);
            \draw[blue, dashed, thick] (7,2) -- (12,12);

            \filldraw[red] (12,12) circle (4.5pt);
            \node[red, below right] at (12,12) {\small $\ell_{\Z_{12}}(\varphi)$};
        \end{tikzpicture}
        \caption{$\ell_{\Z_{12}}(\mathcal{E}_{\lambda_2})$ for $G = \Cay(\Z_{12}, \set{2,3})$, with solution eigenvector $\varphi = \varphi^1_1 + \frac{1}{\sqrt{3}}\varphi^4_1$. The line $c \cdot (12,12), \ c > 0$ is in red.}\label{fig:cone}
    \end{figure}
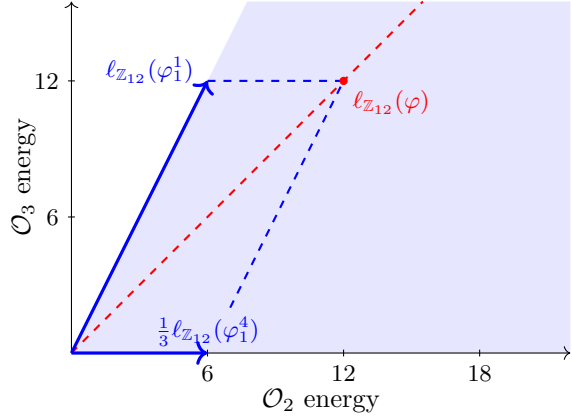

\begin{rexample2}[continued]
We can now fully explain our exact certificate of lower conformal rigidity of $G = \Cay(\Z_{12}, \set{2,3})$. The group $\Z_{12}$ is abelian, so Theorem~\ref{thm:abelian-cayley} applies, and $\ell_\Psi(\E_{\lambda_2})$ 
is a polyhedral cone. The eigenspace $\E_{\lambda_2} = X_1 \oplus X_4 \oplus X_5$ breaks into three 2-dimensional real irreducible representations of complex type corresponding to the complex characters of $\Z_{12}$,
\[
    \chi_k(j) = e^{2\pi i jk/12}, \qquad k = 1, 4, 5.
\]
Taking real and imaginary parts gives the real basis $\varphi_1^k(j) = \cos(\pi 
jk/6)$ and $\varphi_2^k(j) = \sin(\pi jk/6)$ for each $X_k$. This means the orbit-energy vectors of the eigenvectors $\varphi_1^1, \varphi_1^4, \varphi_1^5$ give the extreme rays of the polyhedral cone in~\eqref{eq:polyhedral-cone}:
\[
    \ell_\Psi(\varphi_1^1) = (6, 12), \qquad \ell_\Psi(\varphi_1^4) = (18, 0), 
    \qquad \ell_\Psi(\varphi_1^5) = (6, 12).
\]
The rays for $k = 1$ and $k = 5$ coincide, so conformal rigidity reduces to the linear feasibility problem
\[
    a_1 (6, 12) + a_2 (18, 0) = (12, 12),
\]
such that $a_1, a_2 \geq 0$, which is solved by $a_1 = 1$, $a_2 = 1/3$ pictured in Figure~\ref{fig:cone}. The corresponding certificate $\varphi = \varphi_1^1 + (1/\sqrt{3}) \cdot \varphi_1^4$ gives an exact 1-dimensional orbit-isometric embedding on $\E_{\lambda_2}$.
\end{rexample2}

Experimentally, it seems that for most conformally rigid graphs $G$, the eigenspace $\E_\lambda$ decomposes into distinct irreducibles with respect to the full automorphism group. There are 191 graphs in \textsc{HoG} that we found to be numerically conformally rigid that cannot be explained by a structural result (edge-transitive, 1-walk regular, or regular and bipartite). Nine of these do not satisfy Theorem~\ref{thm:convex-image}, which we discussed in \S~\ref{subsec:vertex-transitive}, so these clearly do not satisfy Theorem~\ref{thm:polyhedral-cone} either. Out of the remaining 182 graphs that satisfy Theorem~\ref{thm:convex-image} with respect to their full automorphism groups, 176 have their relevant eigenspaces decompose into distinct irreducibles to satisfy Theorem~\ref{thm:polyhedral-cone}. We were able to rigorously prove the conformal rigidity of all 176 of these graphs by finding an algebraic certificate inside the corresponding polyhedral cone. Moreover, our certification --- for all but a handful of graphs --- did not require manual inspection of the graph. Most of the graphs we had to inspect had large automorphism groups ($> 15000$), so some work was involved to find a suitable small subgroup of automorphisms. We refer to our \textsc{GitHub} repository~\cite{me} for implementation details of our full audit of \textsc{HoG}. We were also able to find an exact certificate for a graph whose conformal rigidity was previously only numerically certified. 

\begin{figure}[htbp]
  \centering
  \begin{subfigure}[b]{0.30\textwidth}
    \centering
    \begin{tikzpicture}[
        vertex/.style={circle, draw, fill=black, inner sep=1.2pt},
        scale=0.7
    ]
    \coordinate (v0) at (-0.760, 2.500);
    \coordinate (v1) at (1.077, 1.352);
    \coordinate (v2) at (0.010, -0.245);
    \coordinate (v3) at (-0.190, 0.678);
    \coordinate (v4) at (0.148, 2.475);
    \coordinate (v5) at (-1.871, 0.814);
    \coordinate (v6) at (0.367, -2.500);
    \coordinate (v7) at (0.911, -1.438);
    \coordinate (v8) at (0.910, 0.226);
    \coordinate (v9) at (-0.803, -0.404);
    \coordinate (v10) at (-1.664, -2.092);
    \coordinate (v11) at (-0.775, 1.702);
    \coordinate (v12) at (1.597, 1.883);
    \coordinate (v13) at (-2.345, -1.077);
    \coordinate (v14) at (0.164, -1.439);
    \coordinate (v15) at (-1.896, -0.097);
    \coordinate (v16) at (-1.976, 1.540);
    \coordinate (v17) at (2.257, -0.191);
    \coordinate (v18) at (2.345, 0.929);
    \coordinate (v19) at (1.992, -1.276);
    \draw[blue, thick] (v7)--(v17);
    \draw[blue, thick] (v7)--(v15);
    \draw[blue, thick] (v8)--(v18);
    \draw[blue, thick] (v4)--(v5);
    \draw[blue, thick] (v8)--(v16);
    \draw[blue, thick] (v3)--(v5);
    \draw[blue, thick] (v6)--(v8);
    \draw[blue, thick] (v2)--(v19);
    \draw[blue, thick] (v6)--(v7);
    \draw[blue, thick] (v0)--(v2);
    \draw[blue, thick] (v6)--(v10);
    \draw[blue, thick] (v1)--(v19);
    \draw[blue, thick] (v9)--(v12);
    \draw[blue, thick] (v5)--(v10);
    \draw[blue, thick] (v0)--(v1);
    \draw[blue, thick] (v1)--(v12);
    \draw[blue, thick] (v6)--(v9);
    \draw[blue, thick] (v10)--(v14);
    \draw[blue, thick] (v9)--(v11);
    \draw[blue, thick] (v13)--(v16);
    \draw[blue, thick] (v5)--(v9);
    \draw[blue, thick] (v2)--(v14);
    \draw[blue, thick] (v1)--(v11);
    \draw[blue, thick] (v11)--(v16);
    \draw[blue, thick] (v10)--(v13);
    \draw[blue, thick] (v14)--(v18);
    \draw[blue, thick] (v13)--(v15);
    \draw[blue, thick] (v2)--(v13);
    \draw[blue, thick] (v12)--(v18);
    \draw[blue, thick] (v11)--(v15);
    \draw[blue, thick] (v14)--(v17);
    \draw[blue, thick] (v3)--(v17);
    \draw[blue, thick] (v12)--(v17);
    \draw[blue, thick] (v3)--(v15);
    \draw[blue, thick] (v4)--(v18);
    \draw[blue, thick] (v0)--(v4);
    \draw[blue, thick] (v4)--(v16);
    \draw[blue, thick] (v0)--(v3);
    \draw[blue, thick] (v8)--(v19);
    \draw[blue, thick] (v7)--(v19);
    \draw[orange, thick] (v3)--(v7);
    \draw[orange, thick] (v4)--(v8);
    \draw[orange, thick] (v0)--(v5);
    \draw[orange, thick] (v6)--(v19);
    \draw[orange, thick] (v1)--(v2);
    \draw[orange, thick] (v9)--(v10);
    \draw[orange, thick] (v11)--(v12);
    \draw[orange, thick] (v13)--(v14);
    \draw[orange, thick] (v15)--(v16);
    \draw[orange, thick] (v17)--(v18);
    \node[vertex] at (-0.760, 2.500) {};
    \node[vertex] at (1.077, 1.352) {};
    \node[vertex] at (0.010, -0.245) {};
    \node[vertex] at (-0.190, 0.678) {};
    \node[vertex] at (0.148, 2.475) {};
    \node[vertex] at (-1.871, 0.814) {};
    \node[vertex] at (0.367, -2.500) {};
    \node[vertex] at (0.911, -1.438) {};
    \node[vertex] at (0.910, 0.226) {};
    \node[vertex] at (-0.803, -0.404) {};
    \node[vertex] at (-1.664, -2.092) {};
    \node[vertex] at (-0.775, 1.702) {};
    \node[vertex] at (1.597, 1.883) {};
    \node[vertex] at (-2.345, -1.077) {};
    \node[vertex] at (0.164, -1.439) {};
    \node[vertex] at (-1.896, -0.097) {};
    \node[vertex] at (-1.976, 1.540) {};
    \node[vertex] at (2.257, -0.191) {};
    \node[vertex] at (2.345, 0.929) {};
    \node[vertex] at (1.992, -1.276) {};
    \end{tikzpicture}
    \subcaption{Spring layout}
  \end{subfigure}%
  \hfill
  \begin{subfigure}[b]{0.30\textwidth}
    \centering
    \begin{tikzpicture}[
        vertex/.style={circle, draw, fill=black, inner sep=1.2pt},
        scale=0.7
    ]
    \coordinate (v0) at (0.000, 0.000);
    \coordinate (v1) at (-1.780, 1.069);
    \coordinate (v2) at (-1.732, -1.248);
    \coordinate (v3) at (0.078, 0.288);
    \coordinate (v4) at (1.264, -0.220);
    \coordinate (v5) at (2.170, 0.110);
    \coordinate (v6) at (0.000, 0.000);
    \coordinate (v7) at (-1.264, 0.220);
    \coordinate (v8) at (-0.078, -0.288);
    \coordinate (v9) at (1.732, 1.248);
    \coordinate (v10) at (1.780, -1.069);
    \coordinate (v11) at (-0.035, 2.500);
    \coordinate (v12) at (-0.042, 1.249);
    \coordinate (v13) at (0.042, -1.249);
    \coordinate (v14) at (0.035, -2.500);
    \coordinate (v15) at (-0.361, 1.169);
    \coordinate (v16) at (0.372, 0.855);
    \coordinate (v17) at (-0.372, -0.855);
    \coordinate (v18) at (0.361, -1.169);
    \coordinate (v19) at (-2.170, -0.110);
    \draw[blue, thick] (v7)--(v17);
    \draw[blue, thick] (v7)--(v15);
    \draw[blue, thick] (v8)--(v18);
    \draw[blue, thick] (v4)--(v5);
    \draw[blue, thick] (v8)--(v16);
    \draw[blue, thick] (v3)--(v5);
    \draw[blue, thick] (v6)--(v8);
    \draw[blue, thick] (v2)--(v19);
    \draw[blue, thick] (v6)--(v7);
    \draw[blue, thick] (v0)--(v2);
    \draw[blue, thick] (v6)--(v10);
    \draw[blue, thick] (v1)--(v19);
    \draw[blue, thick] (v9)--(v12);
    \draw[blue, thick] (v5)--(v10);
    \draw[blue, thick] (v0)--(v1);
    \draw[blue, thick] (v1)--(v12);
    \draw[blue, thick] (v6)--(v9);
    \draw[blue, thick] (v10)--(v14);
    \draw[blue, thick] (v9)--(v11);
    \draw[blue, thick] (v13)--(v16);
    \draw[blue, thick] (v5)--(v9);
    \draw[blue, thick] (v2)--(v14);
    \draw[blue, thick] (v1)--(v11);
    \draw[blue, thick] (v11)--(v16);
    \draw[blue, thick] (v10)--(v13);
    \draw[blue, thick] (v14)--(v18);
    \draw[blue, thick] (v13)--(v15);
    \draw[blue, thick] (v2)--(v13);
    \draw[blue, thick] (v12)--(v18);
    \draw[blue, thick] (v11)--(v15);
    \draw[blue, thick] (v14)--(v17);
    \draw[blue, thick] (v3)--(v17);
    \draw[blue, thick] (v12)--(v17);
    \draw[blue, thick] (v3)--(v15);
    \draw[blue, thick] (v4)--(v18);
    \draw[blue, thick] (v0)--(v4);
    \draw[blue, thick] (v4)--(v16);
    \draw[blue, thick] (v0)--(v3);
    \draw[blue, thick] (v8)--(v19);
    \draw[blue, thick] (v7)--(v19);
    \draw[orange, thick] (v3)--(v7);
    \draw[orange, thick] (v4)--(v8);
    \draw[orange, thick] (v0)--(v5);
    \draw[orange, thick] (v6)--(v19);
    \draw[orange, thick] (v1)--(v2);
    \draw[orange, thick] (v9)--(v10);
    \draw[orange, thick] (v11)--(v12);
    \draw[orange, thick] (v13)--(v14);
    \draw[orange, thick] (v15)--(v16);
    \draw[orange, thick] (v17)--(v18);
    \node[vertex] at (0.000, 0.000) {};
    \node[vertex] at (-1.780, 1.069) {};
    \node[vertex] at (-1.732, -1.248) {};
    \node[vertex] at (0.078, 0.288) {};
    \node[vertex] at (1.264, -0.220) {};
    \node[vertex] at (2.170, 0.110) {};
    \node[vertex] at (0.000, 0.000) {};
    \node[vertex] at (-1.264, 0.220) {};
    \node[vertex] at (-0.078, -0.288) {};
    \node[vertex] at (1.732, 1.248) {};
    \node[vertex] at (1.780, -1.069) {};
    \node[vertex] at (-0.035, 2.500) {};
    \node[vertex] at (-0.042, 1.249) {};
    \node[vertex] at (0.042, -1.249) {};
    \node[vertex] at (0.035, -2.500) {};
    \node[vertex] at (-0.361, 1.169) {};
    \node[vertex] at (0.372, 0.855) {};
    \node[vertex] at (-0.372, -0.855) {};
    \node[vertex] at (0.361, -1.169) {};
    \node[vertex] at (-2.170, -0.110) {};
    \end{tikzpicture}
    \subcaption{$\E_{\lambda_2}$}
  \end{subfigure}
  \hfill
  \begin{subfigure}[b]{0.30\textwidth}
    \centering
    \begin{tikzpicture}[
        vertex/.style={circle, draw, fill=black, inner sep=1.2pt},
        scale=0.7
    ]
    \coordinate (v0) at (0.000, 0.000);
    \coordinate (v1) at (-0.791, 2.500);
    \coordinate (v2) at (-0.691, -1.699);
    \coordinate (v3) at (1.553, -1.813);
    \coordinate (v4) at (2.329, -0.283);
    \coordinate (v5) at (-2.399, 1.296);
    \coordinate (v6) at (0.000, 0.000);
    \coordinate (v7) at (-2.329, 0.283);
    \coordinate (v8) at (-1.553, 1.813);
    \coordinate (v9) at (0.691, 1.699);
    \coordinate (v10) at (0.791, -2.500);
    \coordinate (v11) at (1.837, -0.112);
    \coordinate (v12) at (-1.775, -2.483);
    \coordinate (v13) at (1.775, 2.483);
    \coordinate (v14) at (-1.837, 0.112);
    \coordinate (v15) at (-0.489, 0.505);
    \coordinate (v16) at (-1.743, -1.970);
    \coordinate (v17) at (1.743, 1.970);
    \coordinate (v18) at (0.489, -0.505);
    \coordinate (v19) at (2.399, -1.296);
    \draw[blue, thick] (v7)--(v17);
    \draw[blue, thick] (v7)--(v15);
    \draw[blue, thick] (v8)--(v18);
    \draw[blue, thick] (v4)--(v5);
    \draw[blue, thick] (v8)--(v16);
    \draw[blue, thick] (v3)--(v5);
    \draw[blue, thick] (v6)--(v8);
    \draw[blue, thick] (v2)--(v19);
    \draw[blue, thick] (v6)--(v7);
    \draw[blue, thick] (v0)--(v2);
    \draw[blue, thick] (v6)--(v10);
    \draw[blue, thick] (v1)--(v19);
    \draw[blue, thick] (v9)--(v12);
    \draw[blue, thick] (v5)--(v10);
    \draw[blue, thick] (v0)--(v1);
    \draw[blue, thick] (v1)--(v12);
    \draw[blue, thick] (v6)--(v9);
    \draw[blue, thick] (v10)--(v14);
    \draw[blue, thick] (v9)--(v11);
    \draw[blue, thick] (v13)--(v16);
    \draw[blue, thick] (v5)--(v9);
    \draw[blue, thick] (v2)--(v14);
    \draw[blue, thick] (v1)--(v11);
    \draw[blue, thick] (v11)--(v16);
    \draw[blue, thick] (v10)--(v13);
    \draw[blue, thick] (v14)--(v18);
    \draw[blue, thick] (v13)--(v15);
    \draw[blue, thick] (v2)--(v13);
    \draw[blue, thick] (v12)--(v18);
    \draw[blue, thick] (v11)--(v15);
    \draw[blue, thick] (v14)--(v17);
    \draw[blue, thick] (v3)--(v17);
    \draw[blue, thick] (v12)--(v17);
    \draw[blue, thick] (v3)--(v15);
    \draw[blue, thick] (v4)--(v18);
    \draw[blue, thick] (v0)--(v4);
    \draw[blue, thick] (v4)--(v16);
    \draw[blue, thick] (v0)--(v3);
    \draw[blue, thick] (v8)--(v19);
    \draw[blue, thick] (v7)--(v19);
    \draw[orange, thick] (v3)--(v7);
    \draw[orange, thick] (v4)--(v8);
    \draw[orange, thick] (v0)--(v5);
    \draw[orange, thick] (v6)--(v19);
    \draw[orange, thick] (v1)--(v2);
    \draw[orange, thick] (v9)--(v10);
    \draw[orange, thick] (v11)--(v12);
    \draw[orange, thick] (v13)--(v14);
    \draw[orange, thick] (v15)--(v16);
    \draw[orange, thick] (v17)--(v18);
    \node[vertex] at (0.000, 0.000) {};
    \node[vertex] at (-0.791, 2.500) {};
    \node[vertex] at (-0.691, -1.699) {};
    \node[vertex] at (1.553, -1.813) {};
    \node[vertex] at (2.329, -0.283) {};
    \node[vertex] at (-2.399, 1.296) {};
    \node[vertex] at (0.000, 0.000) {};
    \node[vertex] at (-2.329, 0.283) {};
    \node[vertex] at (-1.553, 1.813) {};
    \node[vertex] at (0.691, 1.699) {};
    \node[vertex] at (0.791, -2.500) {};
    \node[vertex] at (1.837, -0.112) {};
    \node[vertex] at (-1.775, -2.483) {};
    \node[vertex] at (1.775, 2.483) {};
    \node[vertex] at (-1.837, 0.112) {};
    \node[vertex] at (-0.489, 0.505) {};
    \node[vertex] at (-1.743, -1.970) {};
    \node[vertex] at (1.743, 1.970) {};
    \node[vertex] at (0.489, -0.505) {};
    \node[vertex] at (2.399, -1.296) {};
    \end{tikzpicture}
    \subcaption{$\E_{\lambda_n}$}
  \end{subfigure}
  \caption{Non-Cayley vertex-transitive $(20,10)$ graph: spring layout (left), PCA-projections of edge-isometric symmetrized embeddings on $\E_{\lambda_2}$ (center) and $\E_{\lambda_n}$ (right).}\label{fig:2010}
\end{figure}
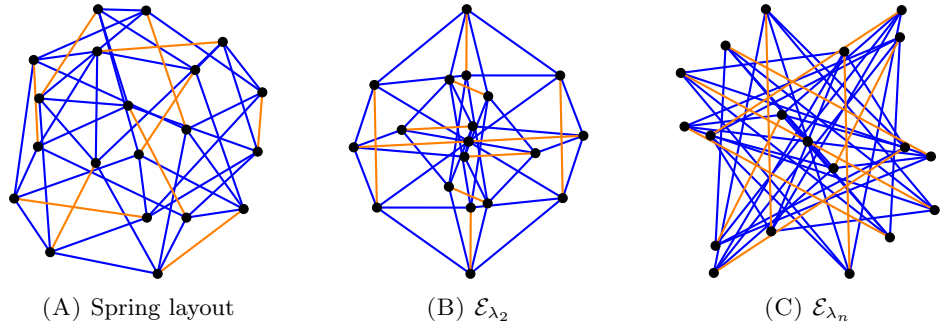

\begin{example}
    Let $G$ be the non-Cayley, vertex-transitive graph $(20,10)$ on $n=20$ vertices from Mathematica (\textsc{HoG 56676}). This graph was previously unable to be certified exactly using the techniques in~\cite{steinerberger}. Under the full automorphism group $\Psi = \Aut(G)$ of order $160$, $G$ decomposes into two edge orbits of size 40 and 10. The eigenspace $\E_{\lambda_2}$ has dimension $7$. Attempting to find an orbit-isometric embedding without using the isotypic decomposition of $\E_{\lambda_2}$ yields a 5-dimensional ideal, and we were unable to read off a real solution from the solved Gr\"{o}bner basis~\cite{me}. Once one considers the isotypic decomposition of $\E_{\lambda_2}$, the problem becomes completely transparent: it splits into one 2-dimensional irreducible $X_1$ of complex type and one 5-dimensional irreducible $X_2$ of real type. Picking $\varphi_j \in X_j$, we get 
    \[\ell_\Psi(\varphi_1) = \paren{230 -90\sqrt{5}, 0}, \ \ell_\Psi(\varphi_2) = \paren{180 - 68\sqrt{5}, 45 - 17 \sqrt{5}}.\]
    Finding a real solution is now trivial as we can take the conical combination with $a_1 = 0$ and $a_2 = \frac{1}{58}(17 + 45\sqrt{5})$ to give us the orbit vector $\Om = (40,10)$. This means any vector in $X_2$ gives an orbit-isometric embedding. One can similarly find an orbit-isometric embedding on $\E_{\lambda_n}$ which also has dimension 7. We display the corresponding edge-isometric symmetrized embeddings on $\E_{\lambda_2}$ and $\E_{\lambda_n}$ in Figure~\ref{fig:2010}. Because the actual edge-isometric embeddings live in a higher dimensional space, we project to $\R^2$ using principal component analysis (PCA).
\end{example}

Remarkably, this method can even certify our 8000 vertex graph in Example~\ref{ex:8000}.

\begin{figure}[h]
        \centering
        \includegraphics[width=\textwidth]{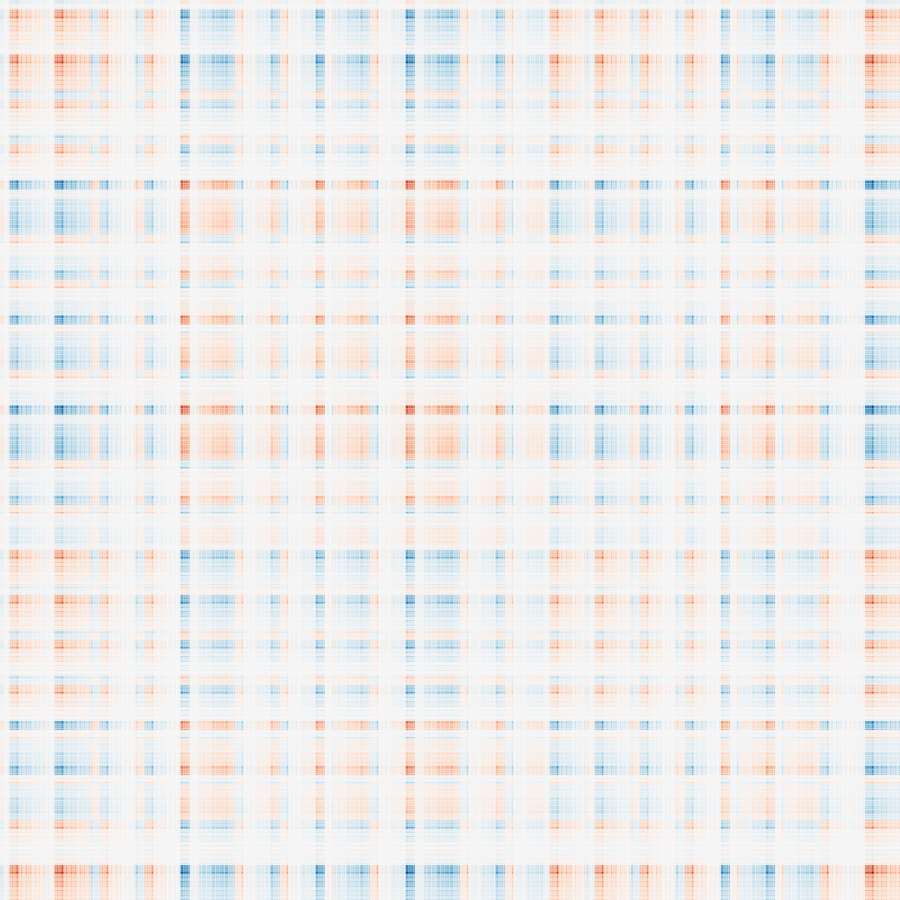}
        \caption{Exact rank-1 lower conformal rigidity certificate $\varphi \varphi^T$ to~\eqref{eq:sym-sdp} for the graph on 8000 vertices.}
        \label{fig:8000exact}
    \end{figure}

\begin{example}
    Let $G$ be the graph from Example~\ref{ex:8000} with $8000$ vertices. Recall that $\dim \E_{\lambda_2} = 10$. Under the full automorphism group $\Psi = \Aut(G)$, $\E_{\lambda_2}$ decomposes into six distinct irreducibles of real type: four 1-dimensional, one 2-dimensional, and one 4-dimensional. Thus, we can get an exact 1-dimensional orbit-isometric embedding $\varphi \in \E_{\lambda_2}$. The corresponding rank-1 SDP certificate $\varphi \varphi^T$ to~\eqref{eq:sym-sdp} is shown in Figure~\ref{fig:8000exact}. At this scale, the linear algebra of computing exact eigenvectors is actually one of the main bottlenecks in the computation.
\end{example}

Even when $\E_\lambda$ does not decompose into distinct irreducibles, it is clear that $\Cone\set{\ell_\Psi(\varphi_j) : 1 \leq j \leq h}$ is always a subset of $\ell_\Psi(\E_\lambda)$. Thus, we can always check this polyhedral cone for a solution. We are not guaranteed a solution this way, but in practice this can still quickly produce a certificate. There are six \textsc{HoG} graphs that satisfy Theorem~\ref{thm:convex-image}, but contain irreducibles with multiplicity. Their \textsc{HoG} IDs are \textsc{50489, 51449, 51450, 51484, 51485}, and \textsc{51486}. Even though these graphs do not satisfy the conditions of Theorem~\ref{thm:polyhedral-cone}, we were able to find an exact certificate in $\Cone\set{\ell_\Psi(\varphi_j) : 1 \leq j \leq h}$ for all of these graphs except for \textsc{HoG 50489}. The graph \textsc{HoG 50489} is likely upper conformally rigid with $\dim \E_{\lambda_n} = 14$ which decomposes into four isotypics of real type: 1-dimensional with multiplicity 1, 2-dimensional with multiplicity 2, 3-dimensional with multiplicity 1, and 3-dimensional with multiplicity 2. Under its automorphism group of size 48, $G$ has 10 edge orbits. Unfortunately, a system of 10 quadratics in 14 variables was already too large for us to find a Gr\"{o}bner basis with \textsc{SageMath}.

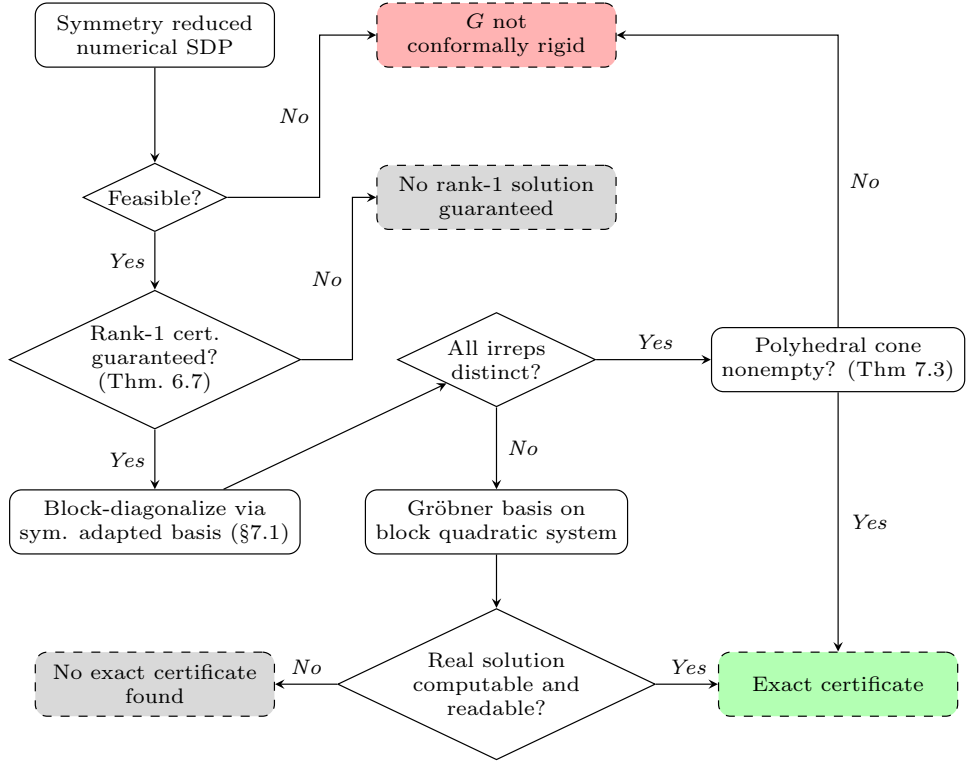
\begin{figure}[h]
\centering
\resizebox{\linewidth}{!}{%
\begin{tikzpicture}[
    x=4.0cm,
    y=1.9cm,
    box/.style = {
        rectangle, rounded corners, draw,
        minimum width=2.8cm,
        minimum height=0.75cm,
        align=center,
        font=\scriptsize
    },
    decision/.style = {
        diamond, draw,
        aspect=2.1,
        inner sep=1pt,
        align=center,
        font=\scriptsize
    },
    leaf/.style = {
        rectangle, rounded corners, draw, dashed,
        minimum width=2.8cm,
        minimum height=0.75cm,
        align=center,
        font=\scriptsize
    },
    arr/.style = {->, >=stealth, thin},
    label/.style = {font=\scriptsize\itshape, midway}
]

\node[box]  (sdp)    at (0,0)
{Symmetry reduced\\numerical SDP};

\node[leaf, fill=red!30] (notcr)  at (1,0)
{$G$ not\\conformally rigid};

\node[decision] (feasible) at (0,-1)
{Feasible?};

\node[leaf, fill=gray!30] (norank1) at (1,-1)
{No rank-1 solution\\guaranteed};

\node[decision] (rank1) at (0,-2)
{Rank-1 cert.\\guaranteed?\\{\normalfont\scriptsize (Thm.~\ref{thm:convex-image})}};

\node[decision] (distinct) at (1,-2)
{All irreps\\distinct?};

\node[box] (lp) at (2,-2)
{Polyhedral cone \\nonempty? (Thm~\ref{thm:polyhedral-cone})};

\node[box] (blockdiag) at (0,-3)
{Block-diagonalize via\\sym. adapted basis (\S\ref{subsec:block-diagonal})};

\node[box] (groebner) at (1,-3)
{Gr\"obner basis on\\block quadratic system};

\node[leaf, fill=gray!30] (noexact) at (0,-4)
{No exact certificate\\found};

\node[decision] (readable) at (1,-4)
{Real solution\\computable and\\ readable?};

\node[leaf, fill=green!30] (exact) at (2,-4)
{Exact certificate};

\draw[arr] (sdp) -- (feasible);
\draw[arr] (feasible) -- node[label,left] {Yes} (rank1);
\draw[arr] (feasible.east) -- ++(0.27,0) -- node[label,left] {No} ++(0,1) --  (notcr.west);

\draw[arr] (rank1) -- node[label,left] {Yes} (blockdiag);
\draw[arr] (rank1.east) -- ++(0.15,0) -- node[label, left] {No}++(0,1) -- (norank1.west);

\draw[arr] (blockdiag) -- (distinct);

\draw[arr] (distinct.east) -- node[label,above] {Yes} (lp.west);
\draw[arr] (distinct) -- node[label,right] {No} (groebner);

\draw[arr] (groebner) -- (readable);

\draw[arr] (readable.west) -- node[label,above] {No} (noexact.east);
\draw[arr] (readable.east) -- node[label,above] {Yes} (exact.west);

\draw[arr] (lp) --node[label,right] {Yes} (exact);
\draw[arr] (lp) --node[label,right] {No} ++(0,2) -- (notcr.east);

\end{tikzpicture}%
}\caption{Symmetry reduction flowchart}\label{fig:flowchart}
\end{figure}

\subsection{Symmetry Reduction Pipeline}
We now give a summary of our entire symmetry reduction pipeline shown in the flowchart shown in Figure~\ref{fig:flowchart}. We assume we have a candidate graph $G$ with a group of automorphisms $\Psi$, and we explain our computational procedure to determine if $G$ is conformally rigid. We begin by solving the symmetry reduced SDP~\eqref{eq:sym-sdp-param}, which should give us a strong indication but not a proof of whether $G$ is or is not conformally rigid. We have already seen that if the solver determines that~\eqref{eq:sym-sdp-param} is feasible, we still need to do more work to determine if $G$ is conformally rigid. If the solver determines $G$ to be infeasible, technically we still need to find an explicit edge weighting $w \neq \one$ such that $\lambda_2(w) > \lambda_2$ or $\lambda_n(w) < \lambda_n$. In contrast to proving conformal rigidity --- where it is difficult to turn an approximate certificate to an exact certificate --- disproving conformal rigidity is straightforward. We can find a certificate disproving conformal rigidity using the perturbation theory in \S~\ref{sec:perturbation-theory} or by solving the problem~\eqref{eq:sym-eigenvalues} to find the weights $w$ that optimize the relevant eigenvalue (see~\cite{gouveia, sun} for SDP formulation). If we only know approximate weights $w$ that increase $\lambda_2$ (decrease $\lambda_n$), we can round to a nearby rational point since the set of all $w$ such that $\lambda_2(w) > \lambda_2$ ($\lambda_n(w) < \lambda_n$) is open. Thus, with this added technicality, we can conclude $G$ is not conformally rigid if the solver determines that the SDP is infeasible.

If $G$ does have an SDP solution, we can be confident that it is conformally rigid, but need to do some more work to prove it. We can look at the isotypic decomposition of $\E_{\lambda}$ as a $\Psi$-representation to see if Theorem~\ref{thm:convex-image} applies. If it does not, then $G$ is not guaranteed to have a 1-dimensional orbit-isometric embedding on $\E_\lambda$, so none of our exact certification methods can prove that $G$ is conformally rigid. In this case, we can try to certify conformal rigidity using the techniques in~\cite{steinerberger} or by using some additional knowledge on the structure of $G$. 

If $G$ does satisfy Theorem~\ref{thm:convex-image}, then we know that we can prove or disprove the conformal rigidity of $G$ by determining if the system of quadratics in~\eqref{eq:grobner-param} has a solution. In the case where $G$ also satisfies Theorem~\ref{thm:polyhedral-cone}, we can rigorously determine the feasibility of this quadratic system exactly through a linear feasibility check. Otherwise, we can attempt to solve the quadratic system with Gr\"obner bases or cylindrical algebraic decomposition. If we can find an exact solution, $G$ is conformally rigid. If the quadratic system is too large or we cannot obtain a real solution from a Gr\"obner basis, then we again need to manually certify conformal rigidity some other way.

\subsection*{Acknowledgements} The author thanks Rekha Thomas and Stefan Steinerberger for carefully reading earlier drafts of this paper and for countless helpful comments.

\bibliographystyle{plain}
\bibliography{references}
\typeout{get arXiv to do 4 passes: Label(s) may have changed. Rerun}
\end{document}